\documentclass[a4paper,fleqn]{cas-sc}
\usepackage[numbers,sort&compress]{natbib}
\usepackage{amssymb,amsxtra,amsmath,amsthm}
\usepackage{latexsym}
\usepackage{ifthen}
\usepackage{graphicx}
\usepackage{calc}
\usepackage{indentfirst}
\usepackage{algorithm,algorithmicx}
\usepackage{algpseudocode}
\usepackage{color}
\usepackage{bigints}
\usepackage{booktabs}
\usepackage{Article}

\newcommand{\R}{\mathbb{R}}

\newcommand{\lstsize}{\fontsize{7}{8.5}\selectfont}

\newtheorem{theorem}{Theorem}[section]
\newtheorem{lemma}[theorem]{Lemma}
\newtheorem{proposition}[theorem]{Proposition}

\theoremstyle{definition}

\theoremstyle{remark}

\usepackage[english]{babel}
\usepackage{listings}
\usepackage{pgfplots}
\usepackage{placeins}
\usepackage{url}

\usepgfplotslibrary{groupplots}
\pgfplotsset{compat=1.18}
\pgfplotsset{
  paper/.style={
    grid=major, grid style={gray!30},
    tick label style={font=\normalsize\boldmath},
    label style={font=\large\bfseries\boldmath},
    title style={font=\large\bfseries},
    legend style={font=\normalsize\bfseries\boldmath, fill=white, fill opacity=.9, text opacity=1, draw=gray!50},
    legend cell align=left,
    every axis plot/.append style={line width=1pt},
  },
  paperlog/.style={paper, xmode=log, ymode=log, log basis y=10, max space between ticks=21pt},
}

\usepackage{casfix}  
\makeatletter\@secpenalty=0\makeatother  

\begin{document}
\let\WriteBookmarks\relax

\shorttitle{KernelPack-JAX}
\shortauthors{Cattaneo, Ballard, Kirby, Wright, and Shankar}

\title[mode = title]{KernelPack-JAX: a meshless differentiable GPU framework for node generation and forward and inverse PDE problems}

\author[1]{Alberto Cattaneo}
\affiliation[1]{organization={Kahlert School of Computing, University of Utah},
  city={Salt Lake City}, state={UT}, postcode={84112}, country={USA}}

\author[2]{M. Keith Ballard}
\affiliation[2]{organization={Materials and Manufacturing Directorate, Composite Performance, Air Force Research Laboratory},
  city={Wright-Patterson AFB}, state={OH}, postcode={45324}, country={USA}}

\author[1,3]{Robert M. Kirby}
\affiliation[3]{organization={Scientific Computing and Imaging Institute, University of Utah},
  city={Salt Lake City}, state={UT}, postcode={84112}, country={USA}}

\author[4]{Grady B. Wright}
\affiliation[4]{organization={Department of Mathematics, Boise State University},
  city={Boise}, state={ID}, country={USA}}

\author[1]{Varun Shankar}
\cormark[1]
\ead{shankar@cs.utah.edu}
\cortext[1]{Corresponding author}

\begin{abstract}
We present KernelPack-JAX, a meshless, differentiable, GPU framework in JAX for solving forward and inverse PDE problems in domains whose boundaries, including internal obstacles, are given only by samples of their smooth pieces. Our node generator is based on a new stacked representation of piecewise-smooth boundaries, in which a parametric partition of unity of each smooth piece induces the supports of a curl-free implicit partition of unity of the whole boundary. This representation keeps the edges and corners of the boundary and classifies every point of space without a global fit. Boundary and interior nodes are then selected by a parallel greedy independent-set algorithm that reproduces sequential Poisson-disk sampling exactly. We discretize elliptic and parabolic problems with radial basis function-generated finite differences (RBF-FD) on stencils selected by a new visibility heuristic to avoid crossing the boundary, with ghost nodes for the boundary conditions. In addition, our framework exposes a GPU sparse direct solver to JAX as a differentiable linear solve, so that Jacobians of time-dependent problems cost a small multiple of a forward solve. We derive error estimates that separate consistency from stability and bound the parameter error by the forward error. We demonstrate high-order convergence for Darcy flow, reaction--diffusion, and advection--diffusion in 2D and 3D domains, including one with an internal obstacle, and show that the recovered parameters largely follow the forward errors. We also present timings for every stage.
\end{abstract}


\hfuzz=130pt\maketitle\hfuzz=0.1pt

\section{Introduction}
\label{sec:intro}
Many applications require the solution of partial differential equations (PDEs) in domains whose geometry is known only through data, such as points sampled from medical images~\cite{antiga2008image,taylor2009patient} or from images of porous rock~\cite{blunt2013pore}, scanned surfaces~\cite{berger2017survey,kazhdan2013screened}, or the output of another computation. Many of these PDEs also contain parameters that must be recovered from measurements of the solution~\cite{stuart2010inverse}, such as the permeability of an aquifer or a reservoir~\cite{carrera1986estimation,yeh1986review,oliver2011history}, the rates of reactions in models of blood clotting~\cite{kuharsky2001surface,fogelson2015fluid}, or a velocity field. We consider such forward and inverse problems in two and three dimensions, in domains whose piecewise-smooth boundaries have edges and corners and may include internal obstacles. Our framework is designed to satisfy four criteria. First, the input should be only samples of the smooth pieces of the boundary, from which we generate node sets that respect a prescribed spacing $h$ and keep the edges and corners. Second, the discretization should be of high order and stable near the boundary, even in nonconvex domains. Third, derivatives of the discrete solution with respect to the parameters should cost little more than the solution itself. Fourth, every stage should run in parallel on the GPU.

Meshless methods suit such problems, since they need only nodes in the domain and on its boundary. Radial basis function-generated finite differences (RBF-FD) approximate differential operators on scattered nodes by local interpolation on stencils~\cite{tolstykh2003,shu2003local,cecil2004numerical,wright2006}, and have matured into a high-order method for PDEs in complex domains~\cite{fasshauer2007meshfree,flyer2015rbffd,fornberg2015solving,fornbergflyer2015primer}. Flyer, Fornberg, Bayona, and Barnett showed that polyharmonic splines~\cite{duchon1977splines} augmented with polynomials converge at the order set by the polynomial degree, with no shape parameter to choose~\cite{flyer2016,bayona2017}. Bayona~\cite{bayona2019insight} explained the roles of the two parts of the basis, and Bayona, Flyer, and Fornberg~\cite{bayona2019polynomials} analyzed the accuracy near domain boundaries. RBF-FD has been applied to nonlinear transport on the sphere~\cite{flyer2012guide} and to the Navier--Stokes equations~\cite{flyer2016enhancing,shahane2021highorder}, has been run on multiple GPUs~\cite{bollig2012}, and its parameters have been studied systematically~\cite{jancic2021monomial,leborne2023guidelines}. Our own work has contributed RBF-FD methods for PDEs on surfaces~\cite{shankar2015rbffd,SHANKAR2017211,shankar2018rbfloi}, overlapped stencils~\cite{shankar2017overlapped}, hyperviscosity stabilization of advection-dominated problems~\cite{shankar2018hyperviscosity,Varun2020RBFFD}, building on~\cite{fornberg2011stabilization}, and methods for domains with moving boundaries~\cite{shankar2021}. In a complex domain, RBF-FD replaces the generation of a mesh~\cite{persson2004simple,geuzaine2009gmsh,si2015tetgen} by two other tasks, the generation of nodes and the closure of the discretization at the boundary. When the PDE contains unknown parameters, a third task is the computation of derivatives of the discrete solution with respect to them.

Node generation for meshless methods has received much attention in the past decade, and Suchde, Jacquemin, and Davydov~\cite{suchde2023point} give a recent survey. Advancing-front methods, which L\"ohner and O\~nate~\cite{lohner1998advancing,lohner2004general} introduced for filling space with points, were made fast and of variable density by Fornberg and collaborators~\cite{fornberg2015fast,vandersande2021}, with related algorithms in~\cite{zamolo2018two,vlasiuk2018fast}. Slak and Kosec~\cite{slak2019} developed a dimension-independent generator of Poisson-disk type, which underlies the Medusa library~\cite{slak2021medusa} and has been parallelized~\cite{depolli2022parallel} and extended to parametric surfaces~\cite{duh2021fast} and to multi-patch NURBS models (including CAD models)~\cite{duh2024discretization}; Jacquemin, Suchde, and Bordas~\cite{jacquemin2023smart} generate and adapt point clouds directly from CAD descriptions. All of these methods take the boundary from an exact or CAD description. In contrast, our earlier node generator~\cite{shankar2018}, on which the present work builds directly, starts from a small number of seed nodes on the boundary. It fits each smooth closed boundary and each embedded object with a parameter-free, high-order parametric model built from spherical RBF interpolants, which supplies the boundary nodes, the normals, and an inside test, and it fills the domain by Poisson-disk sampling in a single pass. The generator has $O(N)$ complexity in two and three dimensions and modifies the node set locally when embedded objects move, which our high-order RBF-FD method for advection--diffusion on time-varying domains relies on~\cite{shankar2021}.

A boundary given only by samples, however, may have edges and corners, which no single smooth interpolant represents~\cite{fleishman2005robust}, and classifying points then requires an implicit model of the whole boundary. Radial basis function implicits~\cite{carr2001reconstruction,turk2002modelling}, including Hermite interpolants of points and normals~\cite{macedo2011hermite}, partition-of-unity implicit surfaces~\cite{ohtake2003}, and in particular the curl-free partition of unity (CFPU) of Drake, Fuselier, and Wright~\cite{drake2022}, reconstruct accurate implicit models from oriented points. Our approach is to combine a parametric partition of unity on each smooth piece with a CFPU representation of the whole boundary that shares its supports. This extends the generator of~\cite{shankar2018} to piecewise-smooth boundaries, including internal obstacles, given by samples of their smooth pieces. Our node generator then samples on the GPU through greedy independent sets, computed by the parallel greedy algorithm of Blelloch, Fineman, and Shun~\cite{blelloch2012,fischer2018}, and thus belongs to the literature on parallel Poisson-disk sampling~\cite{wei2008,gamito2009accurate,bowers2010parallel,ebeida2011efficient,ebeida2012simple}.

The second task is the closure at the boundary, where one-sided stencils reduce the accuracy of RBF-FD and can destabilize it~\cite{bayona2019polynomials}. Ghost nodes outside the domain allow both the PDE and the boundary condition to be enforced at the boundary nodes~\cite{fedoseyev2002improved,fornberg2002observations,fornbergflyer2015primer,shankar2018,shankar2021}. In nonconvex domains, however, the Euclidean nearest neighbors of a node next to a thin wall can lie on the other side of it. We therefore select our stencils by visibility against the implicit model of the boundary, in the spirit of earlier stencil selection for meshless methods~\cite{belytschko1996meshless,organ1996continuous,davydov2011adaptive,davydov2020octant,davydov2023improved}. The closure also determines stability. Although the error and stability of nodal meshless methods have been analyzed extensively~\cite{schaback2017,davydov2016,davydov2019,tominec2021}, the stability of collocation on scattered nodes remains open, so we measure it empirically.

The third task, the computation of derivatives with respect to parameters, arises in inverse problems. Global or local RBF collocation and the method of fundamental solutions have been used for coefficient identification~\cite{dehghan2007rbf}, inverse heat conduction and source problems~\cite{hon2004fundamental,jin2007mfs}, Cauchy problems~\cite{chan2013local}, and PDE-constrained optimal control~\cite{pearson2013rbf}. However, global collocation~\cite{kansa1990multiquadrics} and the method of fundamental solutions lead to dense and often ill-conditioned systems, which limits them to small problems. Physics-informed neural networks~\cite{raissi2019pinn,karniadakis2021piml} and their Kolmogorov--Arnold variants~\cite{shukla2024fair} treat forward and inverse problems alike by training a network on the residual of the PDE and the data, and they do not require a mesh or a node set. In this paper, however, we seek parameters whose error decreases at a prescribed order under refinement and can be bounded by the forward error. We therefore follow the classical approach of PDE-constrained parameter estimation, which differentiates a discretization and solves the resulting nonlinear least-squares problem by Gauss--Newton or Levenberg--Marquardt iterations~\cite{tarantola2005inverse,gunzburger2002perspectives,haber2000optimization,vogel2002computational,kaltenbacher2008iterative,hinze2009optimization}. 

Derivatives with respect to the parameters can come from adjoints~\cite{plessix2006review}, which dolfin-adjoint derives automatically for finite elements~\cite{farrell2013automated,mitusch2019}, or from automatic differentiation. JAX~\cite{jax2018github} has led to differentiable solvers for finite differences and finite volumes~\cite{kochkov2021,holl2020,bezgin2023jaxfluids}, finite elements~\cite{xue2023jaxfem,cattaneo2026jetsci}, and RBF collocation~\cite{nzoyem2023}, and to implicit differentiation~\cite{blondel2022implicit} and differentiable linear solvers~\cite{lineax2023,liu2026jaxamg}. For high-order RBF-FD in complex three-dimensional domains, the cost of the derivatives lies in sparse linear solves and long time loops. Reverse-mode differentiation stores every step of a time loop unless it is checkpointed~\cite{griewank2000revolve,griewank2008evaluating,wang2009minimal}, and automatic differentiation does not pass through a sparse direct factorization. Our framework therefore exposes the factorization of a GPU sparse direct solver to JAX with its own derivative rules~\cite{giles2008collected}, so that derivatives reuse the factors of the forward solve.

To our knowledge, no existing framework goes from samples of the smooth boundary pieces to node sets, high-order meshless operators, sparse solves, and parameter recovery on the GPU, with derivatives through the operators, the solves, and the time loops (Table~\ref{tab:capabilities}). We present KernelPack-JAX, a JAX framework that does this, with the following contributions.

\begin{table}[tbp]
\centering
\caption{Capabilities of related software, as reported in the cited papers and documentation. A bullet (\textbullet) means yes, a circle ($\circ$) partly, and a dash (--) no or not documented. The column Geometry refers to domains given by samples or CAD descriptions of piecewise-smooth boundaries in two and three dimensions, which the finite element and neural network frameworks handle through meshes or constructive solid geometry. The column Inverse marks frameworks that support parameter recovery or optimal control.}
\label{tab:capabilities}
\footnotesize
\setlength{\tabcolsep}{3pt}
\begin{tabular}{@{}lcccccc@{}}
\toprule
 & Geometry & Nodes & \begin{tabular}[b]{@{}c@{}}High-order\\meshless\end{tabular} & GPU & \begin{tabular}[b]{@{}c@{}}Differentiable\\solve\end{tabular} & Inverse\\
\midrule
Medusa~\cite{slak2021medusa} & \textbullet & \textbullet & \textbullet & -- & -- & --\\
RBF~\cite{hines_rbf} & $\circ$ & \textbullet & \textbullet & -- & -- & --\\
RadialBasisFunctions.jl~\cite{beggs2026rbfjl} & $\circ$ & \textbullet & \textbullet & $\circ$ & $\circ$ & --\\
Compadre~\cite{kuberry2019compadre} & -- & -- & \textbullet & \textbullet & -- & --\\
OpenFPM~\cite{incardona2019openfpm} & -- & -- & \textbullet & \textbullet & -- & --\\
Updes~\cite{nzoyem2023} & $\circ$ & -- & -- & \textbullet & \textbullet & \textbullet\\
DeepXDE~\cite{lu2021deepxde} & $\circ$ & -- & -- & \textbullet & $\circ$ & \textbullet\\
PhysicsNeMo~\cite{physicsnemo2023} & \textbullet & -- & -- & \textbullet & $\circ$ & \textbullet\\
JAX-FEM~\cite{xue2023jaxfem} & $\circ$ & -- & -- & \textbullet & \textbullet & \textbullet\\
fea-in-jax~\cite{ballard2025feainjax} & $\circ$ & -- & -- & \textbullet & \textbullet & \textbullet\\
JAX-Fluids~\cite{bezgin2023jaxfluids} & $\circ$ & -- & -- & \textbullet & \textbullet & $\circ$\\
Firedrake, pyadjoint~\cite{ham2023firedrake,mitusch2019} & $\circ$ & -- & -- & $\circ$ & \textbullet & \textbullet\\
$\Phi_{\mathrm{Flow}}$~\cite{holl2024phiflow} & $\circ$ & -- & -- & \textbullet & \textbullet & \textbullet\\
KernelPack-JAX & \textbullet & \textbullet & \textbullet & \textbullet & \textbullet & \textbullet\\
\bottomrule
\end{tabular}
\end{table}
\begin{itemize}
 \item A stacked boundary representation, in which a parametric partition of unity of local polyharmonic interpolants represents each smooth piece (the parametric layer) and a curl-free partition of unity represents the whole boundary implicitly (the implicit layer). Each parametric support induces exactly one implicit ball, so the two layers share one hierarchy of supports and are refined together. Internal obstacles are represented in the same way as the outer wall (Section~\ref{sec:node_gen}), and an extension of the implicit layer into the volume classifies every point of space without a global fit (Section~\ref{sec:node_gen_farfield}).
 \item Parallel node generation, which treats boundary thinning and the dart throwing of interior nodes as greedy independent sets of conflict graphs, computed by the parallel greedy algorithm of~\cite{blelloch2012} extended to sets that must contain given nodes; the result equals the sequential one (Section~\ref{sec:node_gen_gpu}).
 \item Stencil selection for nonconvex domains by the nearest nodes in a distance measured inside the domain, with visibility~\cite{organ1996continuous} decided against the implicit layer. This selection discourages coupling across walls and is Euclidean wherever the nearest nodes are visible (Section~\ref{sec:forward_stencils}).
 \item Differentiable sparse direct solves on the GPU through cuDSS~\cite{cudss}, in forward or reverse mode. Time-dependent problems use the third-order semi-implicit backward differentiation formula (SBDF3)~\cite{ascher1995}, with hyperviscosity for advection~\cite{fornberg2011stabilization,shankar2018hyperviscosity} (Section~\ref{sec:forward}).
 \item Parameter recovery through the discretization, with Jacobians that cost a small multiple of a forward solve and memory independent of the number of time steps (Section~\ref{sec:inverse}).
 \item Error estimates that specialize the stability--consistency framework of~\cite{schaback2017} to the ghost-node closure with Neumann and Dirichlet rows, the standard IMEX argument to SBDF3, and least-squares perturbation theory~\cite{wedin1973,bjorck1996} to the recovered parameters, which bound the parameter error by the forward error divided by the smallest singular value of the Jacobian (Section~\ref{sec:analysis}).
\end{itemize}
We demonstrate our framework in a curved channel in two dimensions (the H-channel) and in a box with a superelliptic cap in three dimensions (the treasure chest), for Darcy flow, reaction--diffusion, advection--diffusion, and the corresponding inverse problems, with RBF-FD of orders $2$, $4$, and $6$. In a final example, we recover reaction laws from solutions carried by a flow around an internal obstacle, a smaller, rotated copy of the box. The forward solutions converge at high order, and the recovered parameters largely follow the forward errors within the bound of our estimates, except for transport with $p=2$ in the treasure chest.

Section~\ref{sec:node_gen} presents the stacked boundary representation and our node generator, Section~\ref{sec:forward} the forward discretization, our stencil selection, and the differentiable solves, and Section~\ref{sec:inverse} the recovery of parameters. Section~\ref{sec:analysis} derives error estimates, Section~\ref{sec:results} presents the numerical results, and Section~\ref{sec:conclusions} concludes. The appendices contain the node generator parameters, code listings, proofs, and experimental details.

\section{Node generation}
\label{sec:node_gen}
\suppressfloats[t]

We now present our node generator. Let $\Omega\subset\R^d$, $d=2,3$, be a bounded domain with a piecewise-smooth boundary $\partial\Omega$, and let $h>0$ be a target spacing. Our generator produces boundary nodes $\chib$ with outward unit normals, interior nodes $\chii$, and ghost nodes $\chig$ outside $\Omega$, which together form a cloud. Algorithm~\ref{alg:node_gen} lists the steps and marks those that are new in this paper.

Our generator extends that of~\cite{shankar2018}, in which one global interpolant maps parameters on the circle or the sphere to a smooth boundary of genus zero, the boundary is supersampled and thinned by sample elimination~\cite{yuksel2015sample}, and the domain is filled by Poisson-disk sampling with an inside test against the nearest boundary normal. However, a single global interpolant cannot represent edges and corners, and both sampling stages visit the samples one at a time. We therefore split the boundary into smooth pieces, each represented by a partition of unity of local interpolants, and build on the same supports an implicit layer that replaces the nearest-normal test. Internal obstacles enter as further pieces (Section~\ref{sec:node_gen_description}), the thinning keeps the features by offering them first, and both sampling stages become greedy independent sets computed on the GPU.

\begin{algorithm}[t]
\caption{Node generation from samples of a piecewise-smooth boundary.}
\label{alg:node_gen}
\begin{algorithmic}[1]
\Require for every smooth piece $\Gamma_\ell$ of every component of $\partial\Omega$ (outer wall and internal obstacles): seed nodes, with parameter values or with normals, and the trimming loops of its parameter region (or the exact map of a planar piece); target spacing $h$
\Ensure boundary nodes $\chib$ with normals, ghost nodes $\chig$, interior nodes $\chii$, implicit layer $\phi$
\State\label{alg:ng_param} assign parameter values to the seed nodes of every piece that has none (Section~\ref{sec:node_gen_description}, Appendix~\ref{sec:app_charts})
\State\label{alg:ng_fit} fit the parametric layer $X_\ell$: solve \eqref{eq:pu_local_system} on every support $S_j$, refine the cells until all validation nodes pass (Section~\ref{sec:node_gen_boundary_math}, Appendix~\ref{sec:app_refine})
\State derive the feature sites, their normals, and their levels from the edges where the pieces meet (Section~\ref{sec:node_gen_sampling})
\State\label{alg:ng_thin} sample the pieces at spacing $h$ and thin the feature sites, then the samples, to $\chib$ by \eqref{eq:boundary_thinning} as a parallel greedy independent set (Section~\ref{sec:node_gen_mis}) \Comment{new}
\State assign the normals and place the ghost nodes $\chig$ by \eqref{eq:ghost_nodes}
\State\label{alg:ng_cover} form one ball $B_j$ per support $S_j$; split the supports of balls that hold too many samples and repeat from step~\ref{alg:ng_fit} on those pieces (Section~\ref{sec:node_gen_levelset_math}, Appendix~\ref{sec:app_cover}) \Comment{new}
\State\label{alg:ng_cfpu} fit the curl-free local functions \eqref{eq:cfpu_system} on all balls in groups of equal size and blend them to $f_{\mathrm{loc}}$ (Sections~\ref{sec:node_gen_levelset_math} and~\ref{sec:node_gen_grouping}) \Comment{new}
\State\label{alg:ng_band} label the free cells of the background grid by min-label propagation, sign the components, and extend $\phi$ by \eqref{eq:implicit_blend} (Sections~\ref{sec:node_gen_farfield} and~\ref{sec:node_gen_labels}) \Comment{new}
\Repeat
 \State\label{alg:ng_interior} draw candidates in $\Omega_\epsilon$ and accept a greedy independent set under \eqref{eq:interior_spacing} with $\chib\cup\chig\cup\chii$ fixed (Sections~\ref{sec:node_gen_interior_math} and~\ref{sec:node_gen_mis}) \Comment{new}
\Until{a pass adds no node, or after four passes}
\end{algorithmic}
\end{algorithm}

\paragraph{A stacked representation of the boundary}
A node generator must know where the boundary and its normals lie, and whether a point lies in the domain. A parametric description answers the second question only by inverting the maps or through the nearest boundary normal, which is ambiguous at edges and corners. We therefore describe $\partial\Omega$ by three layers, each built from the one below.
\begin{enumerate}
 \item The \emph{data layer} consists of the seed nodes of each smooth piece of $\partial\Omega$, with their parameter values, and the \emph{trimming loops} of each piece, the closed curves in parameter space that bound its parameter region (Section~\ref{sec:node_gen_description}).
 \item The \emph{parametric layer} consists of one map $X_\ell$ per smooth piece, a partition of unity of local interpolants of the data over rectangular supports $S_j$ in parameter space (Section~\ref{sec:node_gen_boundary_math}). From its positions and normals we sample the boundary nodes $\chib$ and the ghost nodes $\chig$ (Section~\ref{sec:node_gen_sampling}).
 \item The \emph{implicit layer} consists of a single function $\phi:\R^d\to\R$, a partition of unity of local curl-free fits over balls $B_j$ in $\R^d$ (Sections~\ref{sec:node_gen_levelset_math} and~\ref{sec:node_gen_farfield}). Its zero set approximates $\partial\Omega$, its sign places the interior nodes $\chii$ (Section~\ref{sec:node_gen_interior_math}), and rays traced against it keep the RBF-FD stencils inside the domain (Section~\ref{sec:forward_stencils}).
\end{enumerate}
Figure~\ref{fig:node_stack} shows the three layers for the treasure chest of Section~\ref{sec:nodegen_results}.

\begin{figure}[tbp]
\centering
\begin{tabular}{@{}c@{\hspace{2mm}}c@{}}
\footnotesize (a) Cap parameters & \footnotesize (b) Parametric layer\\
\includegraphics[scale=0.72]{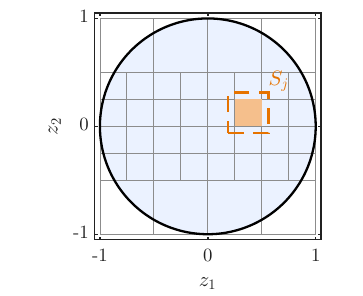} & \includegraphics[scale=0.72]{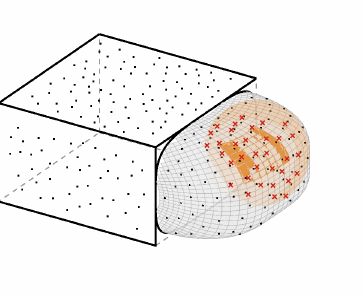}\\[1mm]
\footnotesize (c) Implicit layer, $x_3=0$ & \footnotesize (d) Nodes, $|x_3|<h/2$\\
\includegraphics[scale=0.72]{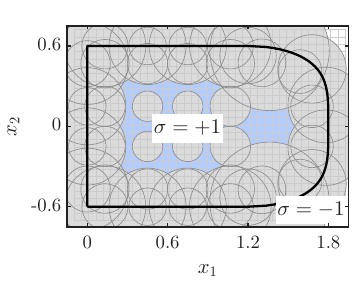} & \includegraphics[scale=0.72]{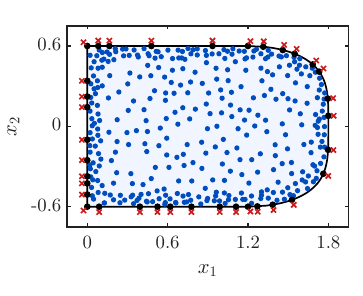}
\end{tabular}
\caption{The stacked representation of the treasure chest \eqref{eq:treasure_chest}. (a) Cap parameter region with one cell (shaded) and its support $S_j$ (dashed). (b) The patch $X(S_j)$ (orange), its ball $B_j$, boundary nodes (dots), and ghost nodes in $B_j$ (crosses). (c) The cover $\bigcup_jB_j$ (gray) and free cells in $\Omega$ (blue), $x_3=0$. (d) Interior (blue), boundary (black), and ghost (red) nodes, $|x_3|<h/2$. (a), (c): $h=0.04$; (b), (d): $h=0.08$.}
\label{fig:node_stack}
\end{figure}

Every layer and both sampling stages run in JAX~\cite{jax2018github} on the GPU (Section~\ref{sec:node_gen_gpu}), and Table~\ref{tab:node_gen_parameters} in Appendix~\ref{sec:app_parameters} lists the parameter values.

\subsection{The data layer}
\label{sec:node_gen_description}
We split $\partial\Omega$ into $N_\Gamma$ smooth pieces, faces of a surface or arcs of a curve, which meet along edges and at corners,
\begin{equation}
 \partial\Omega=\bigcup_{\ell=1}^{N_\Gamma}\Gamma_\ell .
\end{equation}
As in~\cite{shankar2018}, a piece is given by seed nodes with parameter values. However, a piece is not closed, so its parameter values fill a region $U_\ell\subset\R^{d-1}$ on which we assume $\Gamma_\ell$ has a smooth, injective parameterization. The boundary of $U_\ell$ consists of the trimming loops of the piece, which map to its edges. For a face, each loop is a polygon in the parameter plane given by its vertices, one for the outer boundary and one per hole; for an arc, $U_\ell$ is an interval whose endpoints are the vertices. Supersampling fills only the region inside the loops. We denote a parameter by $u\in U_\ell$ and the outward unit normal of $\Gamma_\ell$ at $y\in\Gamma_\ell$ by $n_\ell(y)$.

\paragraph{Internal obstacles}
As in~\cite{shankar2018}, the boundary may include internal obstacles; however, ours may be piecewise smooth and are given, like the outer wall, by the seed nodes of their pieces. The pieces of all components are numbered together, so $\Gamma_1,\dots,\Gamma_{N_\Gamma}$ may belong to different components, and on every piece the normal $n_\ell$ points out of $\Omega$, so on an obstacle it points into the obstacle. With this convention no step of Algorithm~\ref{alg:node_gen} depends on the number or shape of the components. The features of an obstacle come from its own loops, and its ghost nodes lie inside it (Section~\ref{sec:node_gen_sampling}). Its balls join the cover, and the extension of $\phi$ assigns the sign $-1$ to the free cells inside it (Sections~\ref{sec:node_gen_levelset_math} and~\ref{sec:node_gen_farfield}). Interior nodes avoid the obstacle, where $\phi<0$ (Section~\ref{sec:node_gen_interior_math}), and the stencil visibility test detects it (Section~\ref{sec:forward_stencils}). Since the thinning \eqref{eq:boundary_thinning} does not distinguish components, gaps between them must be wider than about $h$, and an obstacle must be thicker than $h/2$ to hold its ghost nodes.

The data for a curved piece $\Gamma_\ell$ are seed nodes for the fit and separate seed nodes for validation,
\begin{align}
 \mathcal{D}_\ell&=\bigl\{(u_{\ell k},\,y_{\ell k})\bigr\}_{k=1}^{N_\ell}, & y_{\ell k}&\in\Gamma_\ell \text{ is the point with parameter } u_{\ell k}\in U_\ell,\\
 \mathcal{V}_\ell&=\bigl\{(u_{\ell v},\,y_{\ell v},\,n_{\ell v})\bigr\}_{v=1}^{N^{\mathrm v}_\ell}, & n_{\ell v}&=n_\ell(y_{\ell v}).
\end{align}
The data for a planar piece are a point $o_\ell\in\R^d$ in its plane, a matrix $E_\ell\in\R^{d\times(d-1)}$ with $E_\ell^{T}E_\ell=I$ whose columns span the plane, and the trimming loops in the coordinates of $E_\ell$. A planar piece may also be given by seed nodes like a curved one.

If the seed nodes of a piece come with normals but without parameter values, our generator assigns them by projecting the nodes onto their principal axis (an arc), onto their best-fit plane (a nearly flat face), or stereographically from a point below the piece (a dome-like face). Appendix~\ref{sec:app_charts} gives the charts and the tests that select one; if none passes, we stop with an error.

\subsection{The parametric layer}
\label{sec:node_gen_boundary_math}
We now construct the maps $X_\ell:U_\ell\to\R^d$. For a planar piece the map is exact,
\begin{equation}
 X_\ell(u)=o_\ell+E_\ell u .
\end{equation}
For a curved piece, whose parameter region is bounded by trimming loops rather than being the circle or the sphere of~\cite{shankar2018}, we construct $X_\ell$ from $\mathcal{D}_\ell$ as a partition of unity~\cite{babuska1997,melenk1996partition,wendland2004} of local interpolants on overlapping rectangles, blended by weights that sum to one~\cite{shepard1968two}. We fix one curved piece, drop the index $\ell$, and write $\mathcal{D}=\{(u_k,y_k)\}_{k=1}^{N}$ and $\mathcal{V}=\{(u_v,y_v,n_v)\}$. With the bounding box $[u^-,u^+]$ of $U$, $u^-,u^+\in\R^{d-1}$, we normalize the parameters to $[-1,1]^{d-1}$ by
\begin{equation}
 z(u)=2\operatorname{diag}(u^+-u^-)^{-1}(u-u^-)-\mathbf{1},
\end{equation}
and write $z_k=z(u_k)$.

\paragraph{Cells and supports}
We cover $[-1,1]^{d-1}$ by axis-aligned cells with disjoint interiors,
\begin{equation}
 C_j=c_j+\operatorname{diag}(\mathbf{s}_j)[-1,1]^{d-1},\qquad j=1,\dots,N_{\mathrm{cell}},
\end{equation}
with centers $c_j\in[-1,1]^{d-1}$ and half-widths $\mathbf{s}_j\in\R^{d-1}$. For an overlap factor $\kappa>1$ we define the support of cell $j$ and its local coordinates by
\begin{equation}
 S_j=c_j+\kappa\operatorname{diag}(\mathbf{s}_j)(-1,1)^{d-1},\qquad
 q_j(z)=\bigl(\kappa\operatorname{diag}(\mathbf{s}_j)\bigr)^{-1}(z-c_j),
\end{equation}
so that $z\in S_j$ if and only if $\norm{q_j(z)}_\infty<1$. We write $q_{ja}$, $a=1,\dots,d-1$, for the components of $q_j$. The parameter region of a planar piece is covered in the same way, but its supports only define balls of the implicit layer.

\paragraph{Local interpolants}
On the support $S_j$ we use the seed nodes with indices
\begin{equation}
 K_j=\{k:\ z_k\in S_j\},
\end{equation}
their mean $\bar y_j=|K_j|^{-1}\sum_{k\in K_j}y_k$, and their local coordinates $q_{jk}=q_j(z_k)$. The local interpolant is a cubic polyharmonic spline~\cite{duchon1977splines} with affine augmentation,
\begin{equation}
 X_j(z)=\bar y_j+\sum_{k\in K_j}\alpha_{jk}\norm{q_j(z)-q_{jk}}^3+\Lambda_j\begin{pmatrix}1\\ q_j(z)\end{pmatrix},
 \qquad \alpha_{jk}\in\R^d,\ \Lambda_j\in\R^{d\times d}.
 \label{eq:coordinates}
\end{equation}
Let $\mathbf{A}_j\in\R^{|K_j|\times|K_j|}$ and $\mathbf{P}_j\in\R^{|K_j|\times d}$ have the entries and rows
\begin{equation}
 (\mathbf{A}_j)_{kk'}=\norm{q_{jk}-q_{jk'}}^3,\qquad
 (\mathbf{P}_j)_{k,:}=\bigl(1,\ q_{jk}^{T}\bigr),\qquad k,k'\in K_j,
\end{equation}
and let $\boldsymbol{\alpha}_j$ and $\mathbf{Y}_j$ be the $|K_j|\times d$ matrices with rows $\alpha_{jk}^{T}$ and $(y_k-\bar y_j)^{T}$. The interpolation conditions $X_j(z_k)=y_k$, $k\in K_j$, and the moment conditions $\mathbf{P}_j^{T}\boldsymbol{\alpha}_j=0$ form the system
\begin{equation}
 \begin{pmatrix}\mathbf{A}_j&\mathbf{P}_j\\ \mathbf{P}_j^{T}&0\end{pmatrix}
 \begin{pmatrix}\boldsymbol{\alpha}_j\\ \Lambda_j^{T}\end{pmatrix}
 =\begin{pmatrix}\mathbf{Y}_j\\ 0\end{pmatrix}.
 \label{eq:pu_local_system}
\end{equation}
Recall that this system is nonsingular if $\mathbf{P}_j$ has full column rank $d$~\cite{micchelli1986interpolation,wendland2004}, which holds if and only if the points $q_{jk}$ do not all lie in one hyperplane of $\R^{d-1}$. On every support that holds seed nodes we require this and $|K_j|\geq m_{\min}$ for a prescribed integer $m_{\min}$, stop with an error otherwise, and solve \eqref{eq:pu_local_system} with a small diagonal shift of $\mathbf{A}_j$ (Table~\ref{tab:node_gen_parameters}).

\paragraph{The blended map}
With the tensor-product Wendland weights~\cite{wendland1995}
\begin{equation}
 \psi_j(z)=\prod_{a=1}^{d-1}\bigl(1-|q_{ja}(z)|\bigr)_+^4\bigl(4|q_{ja}(z)|+1\bigr),
\end{equation}
which are positive on $S_j$ and vanish outside it, we define
\begin{equation}
 X(u)=\frac{\sum_{j=1}^{N_{\mathrm{cell}}}\psi_j(z(u))\,X_j(z(u))}{\sum_{j=1}^{N_{\mathrm{cell}}}\psi_j(z(u))}.
\end{equation}
Supports without seed nodes lie outside $z(U)$ and are dropped. The denominator is positive at every $z$ in a cell whose support holds seed nodes, in particular on $z(U)$, because $z$ lies in some closed cell $C_j$ and $C_j\subset S_j$. The Jacobian
\begin{equation}
 DX(u)=\bigl(\partial_{u_1}X(u),\dots,\partial_{u_{d-1}}X(u)\bigr)\in\R^{d\times(d-1)},
\end{equation}
including the derivatives of the weights, is computed by forward-mode automatic differentiation~\cite{griewank2008evaluating}. The unit normal, with its sign chosen once per piece so that it points out of $\Omega$, is
\begin{equation}
 n(u)=\begin{cases}
 \pm\dfrac{\partial_{u_1}X(u)\times\partial_{u_2}X(u)}{\norm{\partial_{u_1}X(u)\times\partial_{u_2}X(u)}}, & d=3,\\[2ex]
 \pm\dfrac{1}{\norm{DX(u)}}\begin{pmatrix}0&1\\-1&0\end{pmatrix}DX(u), & d=2.
 \end{cases}
\end{equation}

\paragraph{Adaptive refinement}
Since fixed cells may be too coarse for strongly curved pieces, we accept the map only if, at every validation node $u_v$, the position error $\norm{X(u_v)-y_v}$, the normal error $\norm{n(u_v)-n_v}$, and the inverse condition number $\sigma_{\min}(DX(u_v))/\sigma_{\max}(DX(u_v))$ meet prescribed tolerances. Starting from $2^{d-1}$ cells, we split every cell whose support contains a failing validation node into $2^{d-1}$ children and refit. After three unsuccessful refinements we stop with an error (Appendix~\ref{sec:app_refine}).

\subsection{Boundary and ghost nodes from the parametric layer}
\label{sec:node_gen_sampling}
We now sample the boundary nodes from the parametric layer. As in~\cite{shankar2018}, we supersample each piece in parameter space and thin the mapped samples to spacing $h$. However, sample elimination~\cite{yuksel2015sample} visits the samples in an arbitrary order, so a sample of a smooth piece may remove the node that would sit on an edge or a corner. We therefore offer the edges and corners first, much as Delaunay refinement of piecewise-smooth complexes protects vertices, edges, and faces in that order~\cite{cheng2010delaunay}.

\paragraph{Parameter samples}
For a curved piece $\Gamma_\ell$ we bound the stretching of $X_\ell$ by
\begin{equation}
 \bar\sigma_\ell=\max_{u\in U^{\mathrm g}_\ell}\sigma_{\max}\bigl(DX_\ell(u)\bigr),
\end{equation}
where $U^{\mathrm g}_\ell$ is a fine grid in $U_\ell$. Thus $\norm{X_\ell(u)-X_\ell(u')}\lesssim\bar\sigma_\ell\norm{u-u'}$. For a surface piece, $\mathcal{P}_\ell\subset U_\ell$ consists of points on the trimming loops at spacing $r_\ell$, followed by an approximately maximal Poisson-disk sample~\cite{ebeida2011efficient} with separation $r_\ell$ (i.e., every point of $U_\ell$ lies within about $r_\ell$ of a sample), where
\begin{equation}
 r_\ell=\frac{h}{1.1\sqrt2\,\bar\sigma_\ell}\ \text{ for a curved piece},\qquad
 r_\ell=\frac{h}{\sqrt2}\ \text{ for a planar piece}.
\end{equation}
For a curve, $\mathcal{P}_\ell$ consists of equispaced parameters, including both endpoints, with spacing at most $h/(2.2\,\bar\sigma_\ell)$. Every point of the piece therefore lies within roughly $h/\sqrt2$ of $X_\ell(\mathcal{P}_\ell)$, and the thinning sets the final spacing.

\paragraph{Features}
We sample every mapped loop at its vertices and at arc-length spacing at most $h/4$ along each side. For a sample $y$ let $\mathcal{I}(y)$ be the set of pieces whose mapped loops pass within $\delta_{\mathcal F}=0.02h$ of $y$. The samples with $|\mathcal{I}(y)|\geq2$ at which the normals of the pieces in $\mathcal{I}(y)$ differ by more than $\theta_{\mathcal F}=2^\circ$ are the feature sites $x^{\mathcal F}_1,\dots,x^{\mathcal F}_{N_{\mathcal F}}$; a smooth seam thus carries none. The site $x^{\mathcal F}_i$ receives the normalized sum of the normals of its pieces and a level,
\begin{equation}
 n^{\mathcal F}_i=\frac{\sum_{\ell\in \mathcal{I}(x^{\mathcal F}_i)}n_\ell(x^{\mathcal F}_i)}{\norm{\sum_{\ell\in \mathcal{I}(x^{\mathcal F}_i)}n_\ell(x^{\mathcal F}_i)}},\qquad
 \pi_i=\min\bigl(|\mathcal{I}(x^{\mathcal F}_i)|-1,\ d-1\bigr),
\end{equation}
where $n_\ell(x^{\mathcal F}_i)$ is the normal of $\Gamma_\ell$ at the nearest point of its loop. For generic arrangements $\pi_i$ is the codimension of the feature in $\partial\Omega$, with $\pi_i=1$ at a corner of a curve or on a crease where two surface pieces meet, and $\pi_i=2$ at a vertex where three or more meet. Where two loops are tangent, they stay within $\delta_{\mathcal F}$ of each other over a short stretch, and the sites along it count all pieces involved.

\paragraph{Ordering and thinning}
We order the feature sites by decreasing level, $\pi_{i_1}\geq \pi_{i_2}\geq\dots\geq \pi_{i_{N_{\mathcal F}}}$, with ties in the order of the pieces and of their loops, and form the sequence
\begin{equation}
 (x_1,x_2,\dots,x_n)=\bigl(x^{\mathcal F}_{i_1},\dots,x^{\mathcal F}_{i_{N_{\mathcal F}}},\ X_1(\mathcal{P}_1),\ \dots,\ X_{N_\Gamma}(\mathcal{P}_{N_\Gamma})\bigr).
\end{equation}
We then keep $x_1$ and, for $i=2,\dots,n$, keep $x_i$ if
\begin{equation}
 \norm{x_i-x_k}\geq h\quad\text{for every kept }x_k\text{ with }k<i .
 \label{eq:boundary_thinning}
\end{equation}
The kept points form $\chib$. To each $x_b\in\chib$ we assign the unit normal
\begin{equation}
 n_b=\begin{cases}n^{\mathcal F}_i, & x_b=x^{\mathcal F}_i \text{ is a feature site},\\ n_\ell(u), & x_b=X_\ell(u),\ u\in\mathcal{P}_\ell,\end{cases}
\end{equation}
where $n_\ell$ is the normal of the parametric layer, and, as in the earlier method, place one ghost node along each normal~\cite{fornberg2002observations},
\begin{equation}
 \chig=\bigl\{x_b+\tfrac12 h\,n_b:\ x_b\in\chib\bigr\}.
 \label{eq:ghost_nodes}
\end{equation}

\subsection{The implicit layer}
\label{sec:node_gen_levelset_math}
To decide whether a point $x\in\R^d$ lies in $\Omega$ without inverting the maps $X_\ell$, we construct a function $\phi:\R^d\to\R$ that is positive in $\Omega$ and negative outside $\overline\Omega$, except within a distance of order $h$ from $\partial\Omega$, and classify $x$ by the sign of $\phi(x)$, as for radial basis function and partition-of-unity implicit surfaces~\cite{carr2001reconstruction,turk2002modelling,ohtake2003}. Near the boundary, $\phi$ is a partition of unity of local functions that vanish at the boundary samples in their supports and increase into $\Omega$, as in the curl-free partition of unity (CFPU) method of Drake, Fuselier, and Wright~\cite{drake2022}, but with supports taken from the parametric layer. We fit $\phi$ to points $\xi_k\in\partial\Omega$ with outward unit normals $\hat n_k$, $k=1,\dots,N_\xi$, namely the boundary nodes $x_b$ with their normals $n_b$ (Section~\ref{sec:node_gen_sampling}) and a denser sample of the pieces near the feature sites (Appendix~\ref{sec:app_cover}).

\paragraph{One ball per parametric support}
Each support $S_j$, on a piece with map $X$, is mapped to a patch $X(S_j)\subset\partial\Omega$, which is covered by the ball $B_j$ whose center $\bar x_j$ is the mean of the images of the Gauss--Legendre points in $S_j$ and whose radius is
\begin{equation}
 R_j=\kappa_B\max\Bigl(\max_{y\in X(S_j)}\norm{y-\bar x_j},\,2\bar s\Bigr),
\end{equation}
with an enlargement factor $\kappa_B>1$ and the median nearest-neighbor distance $\bar s$ of the samples $\xi_k$; the maximum is evaluated at the same Gauss--Legendre points. Since the supports cover the parameter regions, the balls cover $\partial\Omega$ up to the resolution of the quadrature, and we check that every sample lies in some ball (Appendix~\ref{sec:app_cover}). Let $I_j=\{k:\ \xi_k\in B_j\}$ index the samples in $B_j$.

\paragraph{Local functions}
In $B_j$ we use the scaled coordinates $\hat x=(x-\bar x_j)/R_j$ and $\hat x_k=(\xi_k-\bar x_j)/R_j$. We first seek a scalar potential $s_j$ whose gradient interpolates the normals~\cite{macedo2011hermite}, $\nabla s_j(\hat x_i)=\hat n_i$ for $i\in I_j$. With the polyharmonic spline~\cite{duchon1977splines} $\varphi(\hat x)=\norm{\hat x}^3$ we take
\begin{equation}
 s_j(\hat x)=-\sum_{k\in I_j}\lambda_{jk}\cdot\nabla\varphi(\hat x-\hat x_k)+b_j\cdot\hat x,\qquad \lambda_{jk},\,b_j\in\R^d,
 \label{eq:cfpu_potential}
\end{equation}
whose gradient is $\nabla s_j(\hat x)=\sum_k\Phi(\hat x-\hat x_k)\lambda_{jk}+b_j$ with the curl-free kernel~\cite{narcowich1994generalized,fuselier2008sobolev}
\begin{equation}
 \Phi(\hat x)=-\nabla\nabla^{T}\varphi(\hat x)=-3\Bigl(\norm{\hat x}I+\frac{\hat x\hat x^{T}}{\norm{\hat x}}\Bigr)\in\R^{d\times d}.
\end{equation}
The interpolation and moment conditions form the saddle-point system
\begin{equation}
 \sum_{k\in I_j}\Phi(\hat x_i-\hat x_k)\lambda_{jk}+b_j=\hat n_i,\quad i\in I_j,\qquad \sum_{k\in I_j}\lambda_{jk}=0,
 \label{eq:cfpu_system}
\end{equation}
of size $d(|I_j|+1)$, which we also solve with a small diagonal shift (Table~\ref{tab:node_gen_parameters}). Since $s_j$ is constant along $\partial\Omega$ only up to the error of this fit, we subtract the interpolant $\tilde s_j(\hat x)=\sum_{k\in I_j}\tilde\lambda_{jk}\norm{\hat x-\hat x_k}+\tilde\lambda_{j0}$ of its values $s_j(\hat x_k)$, $k\in I_j$, with $\sum_k\tilde\lambda_{jk}=0$, and define the local function
\begin{equation}
 f_j(x)=R_j\bigl(\tilde s_j(\hat x)-s_j(\hat x)\bigr).
\end{equation}
It vanishes at every sample in $B_j$, and since $\tilde s_j$ only corrects the small variation of $s_j$ along the boundary, $\nabla f_j\approx-\hat n$ near the samples, so $f_j$ increases into $\Omega$. Finally, with the Wendland weights~\cite{wendland1995} $\omega_j(x)=(1-r_j)_+^4(4r_j+1)$, $r_j=\norm{x-\bar x_j}/R_j$, we blend the local functions on the union of the balls,
\begin{equation}
 f_{\mathrm{loc}}(x)=\frac{\sum_j\omega_j(x)f_j(x)}{\sum_j\omega_j(x)},\qquad x\in\bigcup_jB_j .
\end{equation}

\subsection{The implicit layer away from the boundary}
\label{sec:node_gen_farfield}
We now extend the implicit layer to all of space. Let
\begin{equation}
 \mathcal{U}=\R^d\setminus\bigcup_jB_j
\end{equation}
be the uncovered region, where every weight $\omega_j$ vanishes and $f_{\mathrm{loc}}$ is undefined. Since the balls cover $\partial\Omega$ (Section~\ref{sec:node_gen_levelset_math}), $\mathcal{U}\cap\partial\Omega=\emptyset$ up to the resolution of the quadrature, so a connected subset of $\mathcal{U}$ lies either in $\Omega$ or in $\R^d\setminus\overline\Omega$. We assign each such region the sign of $f_{\mathrm{loc}}$ on the adjacent covered region, using a background grid.

\paragraph{Background grid}
Let $\mathbb{B}'$ be the bounding box of the samples $\{\xi_k\}$, enlarged by $5\%$ of its side lengths on each side, and let $h_{\mathrm g}=\tfrac12\min_jR_j$. We divide $\mathbb{B}'$ into cubic cells $V_\iota$ of width $h_{\mathrm g}$ with centers $c_\iota$. Since every radius satisfies $R_j\geq2h_{\mathrm g}>\tfrac12\sqrt d\,h_{\mathrm g}$, every ball contains a cell center. A cell is \emph{covered} if $\bar\omega(c_\iota)>0$, where
\begin{equation}
 \bar\omega(x)=\sum_j\omega_j(x),
\end{equation}
and \emph{free} otherwise. Two cells are adjacent if they share a face.

\paragraph{Component signs}
We assign a sign $\varsigma_\iota\in\{-1,1\}$ to every cell in three steps.
\begin{enumerate}
 \item For a covered cell, $\varsigma_\iota=\operatorname{sign}f_{\mathrm{loc}}(c_\iota)$.
 \item We label the connected components of the free cells by min-label propagation.
 \item For each component $\mathcal{C}$, let $\Sigma_{\mathcal{C}}$ be the set of signs $\varsigma_\iota$ of the covered cells adjacent to $\mathcal{C}$. If $\Sigma_{\mathcal{C}}$ contains exactly one sign, we set $\varsigma_\iota$ to that sign for every cell of $\mathcal{C}$.
\end{enumerate}
If some $\Sigma_{\mathcal{C}}$ contains both signs (i.e., the cover does not separate $\Omega$ from its exterior) or is empty, we use instead the generalized winding number~\cite{jacobson2013,barill2018fast} of the boundary, evaluated with the quadrature points of the pieces and their area-weighted normals.

\paragraph{The extended function}
Let $\phi_\infty$ be the diameter of the bounding box of $\{\xi_k\}$, and let $\iota(x)$ be the index of the cell that contains $x\in\mathbb{B}'$. The implicit layer is
\begin{equation}
 \phi(x)=\begin{cases}
 f_{\mathrm{loc}}(x), & \bar\omega(x)>0,\\
 \phi_\infty\,\varsigma_{\iota(x)}, & \bar\omega(x)=0,\ x\in\mathbb{B}',\\
 -\phi_\infty, & \bar\omega(x)=0,\ x\notin\mathbb{B}'.
 \end{cases}
 \label{eq:implicit_blend}
\end{equation}
Note that $\phi$ is discontinuous at the boundary of the cover. However, away from $\partial\Omega$ only its sign is used, and near $\partial\Omega$, $\phi=f_{\mathrm{loc}}$. Finally, we verify that $|\phi(\xi_k)|\leq10^{-8}\phi_\infty$ for every $k$.

\subsection{Interior nodes from the implicit layer}
\label{sec:node_gen_interior_math}
We now place the interior nodes. In the earlier method, a bounding box is filled with nodes at spacing $h$, the nodes outside are removed, and a refined layer is obtained by copying the boundary nodes inward. We keep this idea, but test inside by the sign of the implicit layer and sample the refined layer at a smaller spacing, as in variable-density node generators~\cite{fornberg2015fast,slak2019}. For parameters $\rho\in(0,1)$ and $\beta>0$ the spacing is
\begin{equation}
 r(x)=\begin{cases}\rho h, & \operatorname{dist}(x,\chib)\leq\beta h,\\ h, & \text{otherwise}.\end{cases}
\end{equation}
Since the ghost nodes lie at the distance $h/2$ from their boundary nodes, for $\beta\geq\tfrac12$ every boundary and ghost node lies within $\beta h$ of $\chib$, and $r=\rho h$ there.

Let $\mathbb{B}$ be the bounding box of $\partial\Omega$ and $\epsilon>0$. With \eqref{eq:implicit_blend}, interior nodes are admissible in
\begin{equation}
 \Omega_\epsilon=\bigl\{x\in\mathbb{B}:\ \phi(x)\geq\epsilon h\bigr\}.
\end{equation}
Holding $\chib$ and $\chig$ fixed, we seek an approximately maximal set $\chii\subset\Omega_\epsilon$ such that
\begin{equation}
 \norm{x-y}\geq\max\bigl(r(x),r(y)\bigr)\quad\text{for distinct } x,y\in\chii\cup\chib\cup\chig,\ \{x,y\}\not\subset\chib\cup\chig .
 \label{eq:interior_spacing}
\end{equation}
Condition \eqref{eq:interior_spacing} sets the clearance from boundary and ghost nodes, while the threshold $\epsilon h$ removes candidates near $\partial\Omega$, where the sign of $\phi$ is least reliable. We do not impose the clearance through $\phi$, since near convex edges and corners $\phi(x)$ can be smaller than $\operatorname{dist}(x,\partial\Omega)$ by more than $h$, so a condition $\phi(x)\geq\rho h$ would remove candidates several $h$ inside the domain.

Instead of Bridson's algorithm~\cite{bridson2007}, which visits the candidates one at a time, our generator constructs $\chii$ by dart throwing~\cite{dippe1985antialiasing,cook1986,mitchell1987generating}, evaluating all candidates of a pass in parallel (Section~\ref{sec:node_gen_mis}). Let $c_{\mathrm d}>0$ be a candidate density and $|\mathbb{B}|$ the volume of $\mathbb{B}$, and start with the fixed set $\mathcal{X}=\chib\cup\chig$. Each pass proceeds as follows.
\begin{enumerate}
 \item Draw $c_{\mathrm d}\,|\mathbb{B}|/(\rho h)^d$ candidates uniformly in $\mathbb{B}$, and discard those outside $\Omega_\epsilon$.
 \item Visit the remaining candidates in random order, and accept a candidate $y$ if $\norm{y-x}\geq\max(r(x),r(y))$ for every $x\in\mathcal{X}$ and every candidate accepted before it in this pass.
 \item Add the accepted candidates to $\chii$ and to $\mathcal{X}$.
\end{enumerate}

\subsection{Parallel algorithms}
\label{sec:node_gen_gpu}
We now describe the three parallel components of our generator: greedy independent sets for the boundary thinning and dart throwing, grouped local fits and evaluations, and the labeling of the free grid cells. Appendix~\ref{sec:app_listings_nodegen} lists the JAX kernels.

\subsubsection{Greedy independent sets}
\label{sec:node_gen_mis}
Let $x_1,\dots,x_n\in\R^d$ have radii $r_1,\dots,r_n>0$ and distinct priorities $\operatorname{rank}(1),\dots,\operatorname{rank}(n)$, where a smaller value means a higher priority, and let $\mathsf{G}$ be the conflict graph in which $i$ and $j$ are adjacent if $\norm{x_i-x_j}<\max(r_i,r_j)$. The sequential greedy algorithm first accepts a prescribed set of fixed points and then visits the remaining points in order of priority, accepting a point if none of its neighbors has been accepted. Since fixed points may be closer than their radii, the edges between them are removed from $\mathsf{G}$. With constant radius $h$ and the order of the sequence of Section~\ref{sec:node_gen_sampling} as priority, this is the thinning \eqref{eq:boundary_thinning}; with random priorities, and with the boundary, ghost, and previously accepted interior nodes as fixed points, it is one pass of dart throwing.

In our parallel algorithm, every point is labeled undecided ($\mathsf{U}$), in ($\mathsf{I}$), or out ($\mathsf{O}$), starting with the fixed points in $\mathsf{I}$ and all others in $\mathsf{U}$. In each round all undecided points are updated at once; a point with a neighbor in $\mathsf{I}$ moves to $\mathsf{O}$, and otherwise a point none of whose higher-priority neighbors is undecided moves to $\mathsf{I}$. Neighbors are found in the $3^d$ surrounding cells of a uniform grid with cell width $\max_ir_i$, and rounds run until no point is undecided (Appendix~\ref{sec:app_listings_nodegen}, Listings~\ref{lst:cell_table}, \ref{lst:greedy}, and~\ref{lst:dart}).

\begin{proposition}
\label{prop:mis}
The parallel algorithm terminates, and the points that end in $\mathsf{I}$ are exactly the points accepted by the sequential greedy algorithm.
\end{proposition}
\begin{proof}[Proof sketch]
By induction along the priority order, every point is decided no later than the first round in which none of its higher-priority neighbors is undecided, and it then enters the state that the sequential algorithm gives it, since these neighbors already have their sequential states (Appendix~\ref{sec:app_mis}).
\end{proof}

The number of rounds, the dependence length of the greedy algorithm, is at most the number of vertices on the longest priority-increasing path in $\mathsf{G}$. For random priorities it is $O(\log\deg_{\max}\,\log n)$ with high probability, where $\deg_{\max}$ is the largest degree in $\mathsf{G}$~\cite{blelloch2012}, and $O(\log n)$ with high probability for every graph~\cite{fischer2018}, so dart throwing takes few rounds. These bounds do not apply to the boundary thinning, whose priorities follow the sequence order, so its number of rounds can grow with the number of samples on a piece; by the proposition, however, it reproduces \eqref{eq:boundary_thinning} exactly. Our algorithm is the parallel greedy algorithm of~\cite{blelloch2012}, which keeps the priorities fixed and so returns the sequential greedy set; Luby's algorithm~\cite{luby1986} instead draws new priorities in every round. Wei~\cite{wei2008} parallelizes Poisson-disk sampling on the GPU through a phase decomposition of the background grid, and Bowers et al.~\cite{bowers2010parallel} extend it to surfaces.

\subsubsection{Grouped evaluation and fitting}
\label{sec:node_gen_grouping}
Both layers are weighted sums of local functions, with weights $\psi_j/\sum_i\psi_i$ and maps $X_j$ in the parametric layer (where $x$ stands for $z$), and $\omega_j/\sum_i\omega_i$ and $f_j$ in the implicit layer. Let $m(x)$ be the number of supports whose weight is positive at $x$. It equals one inside a single support and is largest near the corners of the cover, so padding every evaluation to the largest $m$ would waste much of the work. We therefore partition the query points into the groups $\mathcal{G}_e=\{x:\ 2^{e-1}<m(x)\leq2^e\}$, $e=0,1,2,\dots$, and, for $x\in\mathcal{G}_e$, sum over the $2^e$ supports with the largest weights at $x$, selected by a partial sort. Since all omitted weights are zero, the grouped sum is exact, and the number of local evaluations, $\sum_e2^e|\mathcal{G}_e|$, is at most $2\sum_xm(x)$ (Listing~\ref{lst:grouped} in Appendix~\ref{sec:app_listings_nodegen}).

The local fits are grouped in the same way, since the dense systems $\mathbf{M}_j\mathbf{c}_j=\mathbf{g}_j$ of a support, \eqref{eq:pu_local_system} or the two systems of Section~\ref{sec:node_gen_levelset_math}, have sizes affine in the number of points of the support. We place support $j$ in the group $e$ with $2^{e-1}<|K_j|\leq2^e$, and embed its system in one for $2^e$ points by the block $\operatorname{diag}(\mathbf{M}_j,I)$, which leaves $\mathbf{c}_j$ unchanged, and solve each group with one batched dense solve (Listing~\ref{lst:cfpu_fit} in Appendix~\ref{sec:app_listings_nodegen}).

\subsubsection{Connected components}
\label{sec:node_gen_labels}
We label the components of Section~\ref{sec:node_gen_farfield} by min-label propagation~\cite{hawick2010parallel}, which needs no sequential graph search. Every free cell starts with its own index as label, and in each round takes the smallest label among itself and its free neighbors. After as many rounds as the largest diameter of a component, measured in cells along face-adjacent paths, every cell carries the smallest index of its component. Each round is one parallel update, implemented by shifting the label array along the grid axes (Listing~\ref{lst:labels} in Appendix~\ref{sec:app_listings_nodegen}).

\section{Forward problems}
\label{sec:forward}
Having generated nodes in the domain and on its boundary, we now describe how our framework solves PDEs on them. We consider elliptic problems such as $-\nabla\cdot(K\nabla u)=f$ in $\Omega$ and semilinear parabolic problems $u_t=\mathcal{A}u+\mathcal{N}(t,u)$ in $\Omega$, each with Dirichlet or Neumann conditions $\mathcal{B}u=g$ on $\partial\Omega$. The discretization is standard RBF-FD on the node sets of Section~\ref{sec:node_gen}, with ghost nodes for the boundary conditions, SBDF3 in time (Section~\ref{sec:forward_ghosts}), and hyperviscosity for advection (Section~\ref{sec:forward_hyperviscosity}). We add stencil selection inside the domain (Sections~\ref{sec:forward_operators} and~\ref{sec:forward_stencils}), the batched GPU implementation in JAX, and a GPU direct solver that JAX can differentiate (Section~\ref{sec:forward_linear_algebra}). Algorithm~\ref{alg:forward} lists the steps and marks the new ones.

\begin{algorithm}[!b]
\caption{Forward solve on a node set.}
\label{alg:forward}
\begin{algorithmic}[1]
\Require nodes $\chii,\chib,\chig$ and implicit layer $\phi$ from Algorithm~\ref{alg:node_gen}; target order $p$; problem data
\Ensure discrete solution $U$ (elliptic) or $U^{n_t}$ at $t=T$ (parabolic)
\State\label{alg:fw_stencil} select the stencil of every node: the $n_s$ nearest nodes in the distance \eqref{eq:stencil_distance}, with visibility decided against $\phi$ (Section~\ref{sec:forward_stencils}, Appendix~\ref{sec:app_raytrace}) \Comment{new}
\State compute the weights of $L$ and $D_a$ in batched dense solves and store them in COO form (Section~\ref{sec:forward_operators})
\State assemble $A$ and $B$, e.g.\ \eqref{eq:variable_coefficient}, and the ghost-node closure \eqref{eq:ghost_closure} or the SBDF3 matrix \eqref{eq:sbdf3_matrix} (Section~\ref{sec:forward_ghosts})
\If{the problem has advection}
 \State compute $\mu_a$ by Arnoldi iteration and the hyperviscosity coefficients \eqref{eq:gamma} (Section~\ref{sec:forward_hyperviscosity})
\EndIf
\State\label{alg:fw_check} check that the rightmost eigenvalue of the reduced operator \eqref{eq:reduced_operator} has nonpositive real part; otherwise stop (Section~\ref{sec:forward_linear_algebra})
\State\label{alg:fw_factor} factor the matrix once with cuDSS, called as a differentiable JAX linear solve (Section~\ref{sec:forward_linear_algebra}) \Comment{new}
\If{the problem is elliptic}
 \State solve \eqref{eq:ghost_closure}
\Else
 \State\label{alg:fw_sbdf} choose $\Delta t$ by \eqref{eq:dt_rule} and run the SBDF3 loop \eqref{eq:sbdf3}, one solve with the factors per step (Section~\ref{sec:forward_time})
\EndIf
\end{algorithmic}
\end{algorithm}

Let $\chii$, $\chib$, and $\chig$ be the interior, boundary, and ghost nodes, with $N_i$, $N_b$, and $N_g=N_b$ elements. The $N=N_i+N_b$ interior and boundary nodes are the physical nodes; all $N_f=N+N_g$ nodes are ordered as interior, boundary, ghost, and indexed by $\mathrm{I}$, $\mathrm{B}$, $\mathrm{P}$ (physical), and $\mathrm{G}$. A discrete function is a vector
\begin{equation}
 U=(U_{\mathrm I},U_{\mathrm B},U_{\mathrm G})\in\R^{N_f},
\end{equation}
and $\bar U=(U_{\mathrm I},U_{\mathrm B})\in\R^{N}$ denotes its values at the physical nodes.

\subsection{Differential operators}
\label{sec:forward_operators}
Our differential operators are RBF-FD approximations~\cite{tolstykh2003,wright2006,fornberg2015solving} with polyharmonic splines and polynomial augmentation~\cite{flyer2016}. For an even target order $p$ and a second-order operator, each stencil uses polynomials of total degree $p+1$, the polyharmonic spline $r^m$ with $m=\max(p+1,5)$, and the
\begin{equation}
 n_s=2\binom{p+1+d}{d}+1
\end{equation}
nodes~\cite{flyer2016,bayona2017} closest to the center in the distance of Section~\ref{sec:forward_stencils}, ghost nodes included as in~\cite{shankar2018,shankar2021}. With the same stencils for the gradient, this gives matrices
\begin{equation}
 L\in\R^{N\times N_f},\qquad D_a\in\R^{N\times N_f},\quad a=1,\dots,d,
\end{equation}
that approximate $\Delta$ and $\partial_{x_a}$ at the physical nodes from values at all nodes. On quasi-uniform nodes with spacing $h$, $L$ is accurate to $O(h^{p})$ and $D_a$ to $O(h^{p+1})$ on smooth functions. Variable-coefficient operators are assembled from these; for example, we discretize $-\nabla\cdot(K\nabla u)=-K\Delta u-\nabla K\cdot\nabla u$ by
\begin{equation}
 -\operatorname{diag}(\bar K)\,L-\sum_{a=1}^{d}\operatorname{diag}(D_aK)\,D_a,
 \label{eq:variable_coefficient}
\end{equation}
where $K$ holds the coefficient at all nodes and $\bar K$ its values at the physical nodes.

\paragraph{Parallel assembly}
The weights of each row of $L$ solve a dense saddle-point system of size $n_s+\binom{p+1+d}{d}$ on the $n_s$ nodes of its stencil~\cite{flyer2016,bayona2017}. These systems are small and independent, so our implementation gathers the stencils of $n_{\mathrm{batch}}$ rows into an array of shape $(n_{\mathrm{batch}},n_s,d)$, forms the local systems with array operations along the leading axis, and calls one batched dense solve on the GPU (Listing~\ref{lst:weights} in Appendix~\ref{sec:app_listings_solvers}). The polynomial basis is evaluated in coordinates centered and scaled per stencil. A local system whose polynomial block has reciprocal condition number below $\sqrt{\epsilon_{\mathrm{mach}}}$, with $\epsilon_{\mathrm{mach}}$ the machine precision, is solved instead by column-pivoted QR~\cite{businger1965linear} of the block, and we stop if the block is numerically rank deficient.
The weights and stencil indices define the matrix in coordinate (COO) form, so a product with it is a gather followed by a scatter-add, differentiable in the weights and the vector (Listing~\ref{lst:matvec} in Appendix~\ref{sec:app_listings_solvers}).

\subsection{Stencil selection}
\label{sec:forward_stencils}
RBF-FD stencils are usually the Euclidean nearest nodes of their center $x_c$. However, next to a thin part of the domain, such as a wall of the H-channel of Section~\ref{sec:nodegen_results}, some of these nodes can lie on the other side of the wall, where the solution need not be related to that at $x_c$. Davydov and coworkers select stencils by geometric criteria aimed at accuracy rather than at the shape of the domain~\cite{davydov2011adaptive,oanh2017adaptive,davydov2023improved}. Instead, we take the $n_s$ nodes nearest to $x_c$ in a distance $d_c$ measured inside the domain. We call a segment $[x_c,y]$ \emph{visible} if $\phi\geq-\delta_\phi$ on it, where $\phi$ is the implicit layer of Section~\ref{sec:node_gen_levelset_math} and $\delta_\phi=0.05h$. For a physical node $y$ let $y^{\partial}=y$, and for a ghost node $y$ let $y^{\partial}$ be the boundary node along whose normal $y$ was placed. A node $y$ is visible from $x_c$ if $[x_c,y^{\partial}]$ is visible. Among the $4n_s$ Euclidean nearest nodes $x_1,\dots,x_{4n_s}$ of $x_c$ we set
\begin{equation}
 d_c(x_l)=\begin{cases}\norm{x_l-x_c}, & x_l\text{ visible from }x_c,\\ \min_j\bigl(\norm{x_j-x_c}+\norm{x^{\partial}_l-x_j}\bigr)+\norm{x_l-x^{\partial}_l}, & \text{otherwise},\end{cases}
 \label{eq:stencil_distance}
\end{equation}
where $j$ runs over the candidates $x_j$ with $\phi(x_j)<2h$ such that $x_j$ is visible from $x_c$ and $x_l$ is visible from $x_j$. Thus the minimum is taken over paths inside the domain that bend once near the boundary, a discrete analogue of the visibility and diffraction criteria for meshless methods in nonconvex bodies~\cite{belytschko1995crack,organ1996continuous,belytschko1996meshless}. To bound the cost, each hidden $x_l$ is tested only from the $8$ candidates $x_j$ with the smallest values of the bracket. Candidates without such a path rank after all others, in Euclidean order.

By the triangle inequality, $d_c(x)\geq\norm{x-x_c}$. Hence, if the $n_s$ Euclidean nearest nodes of $x_c$ are all visible, the stencil is the Euclidean one. Our implementation exploits this in two steps. First, we test the segments from every node to its $n_s$ Euclidean nearest nodes. Segments whose ends lie deep in the domain are visible by the certificate~\eqref{eq:visibility_certificate} of Appendix~\ref{sec:app_raytrace}, which uses only $\phi$ at the two nodes; the rest are decided by sphere tracing~\cite{hart1996} against $\phi$. Second, only for the nodes with a hidden segment among the first $n_s$ (i.e., those that see across the boundary), we evaluate \eqref{eq:stencil_distance} for all $4n_s$ candidates and take the $n_s$ smallest.

\subsection{Ghost-node closure and time stepping}
\label{sec:forward_ghosts}
\label{sec:forward_time}
Without ghost nodes, a boundary node carries only the boundary condition, and the stencils near the boundary are one-sided. Ghost nodes~\cite{fedoseyev2002improved,fornberg2002observations} add one unknown per boundary node, so that both the PDE and the boundary condition can be imposed there. For a problem $\mathcal{A}u=f$ in $\Omega$ with a boundary condition $\mathcal{B}u=g$ on $\partial\Omega$, let $A\in\R^{N\times N_f}$ approximate $\mathcal{A}$ at the physical nodes and $B\in\R^{N_b\times N_f}$ approximate $\mathcal{B}$ at the boundary nodes. We collocate the equation at every physical node and the boundary condition at the boundary nodes,
\begin{equation}
 AU=f|_{\chii\cup\chib},\qquad BU=g|_{\chib}.
 \label{eq:ghost_closure}
\end{equation}
These are $N_f$ equations for $N_f$ unknowns. For the Dirichlet condition $u=g$, $B$ selects $U_{\mathrm B}$. For the Neumann condition $\partial_nu=g$, the row of $B$ at $x_b$ holds the RBF-FD weights of $n_b\cdot\nabla$ on the stencil of $x_b$, with the normal $n_b$ of Section~\ref{sec:node_gen_sampling}. A flux condition $K\partial_nu=g$ is imposed as $\partial_nu=g/K$. For two of the index sets $X$ and $Y$, $A_{XY}$ has the rows of $A$ at $X$ and the columns at $Y$, and $A_Y$ and $B_Y$ are the columns of $A$ and $B$ at $Y$. The ghost values are determined by the boundary rows of $A$ (Dirichlet) or by $B$ (Neumann) whenever the square block $A_{\mathrm{BG}}$, respectively $B_{\mathrm G}$, is nonsingular, which we assume throughout. Eliminating them gives the reduced operator
\begin{equation}
 M=A_{\mathrm{II}}-A_{\mathrm{IG}}A_{\mathrm{BG}}^{-1}A_{\mathrm{BI}}\ \ \text{(Dirichlet)},\qquad
 M=A_{\mathrm P}-A_{\mathrm G}B_{\mathrm G}^{-1}B_{\mathrm P}\ \ \text{(Neumann)},
 \label{eq:reduced_operator}
\end{equation}
of size $N_i$ and $N$, whose spectrum governs the stability of the semi-discretization and which enters the error estimates of Section~\ref{sec:analysis}. Note that with Neumann conditions $M\mathbf 1=0$, since the stencils reproduce constants; a pure Neumann elliptic problem thus determines the solution only up to a constant, and its data satisfy the discrete compatibility condition only approximately. Following~\cite{bochev2005finite}, we add a multiplier $\lambda$ to every equation of \eqref{eq:ghost_closure} and append the constraint $\mathbf 1^TU=0$ over all $N_f$ nodes. Under the assumptions of Theorem~\ref{thm:elliptic}, the resulting bordered system of size $N_f+1$ is nonsingular. Finally, we shift $U$ by the constant that makes its mean over the physical nodes equal to that of a reference (the manufactured solution in Section~\ref{sec:forward_results}, the observations in Section~\ref{sec:inverse}); this changes neither \eqref{eq:ghost_closure} nor $\lambda$.

We now turn to time-dependent problems of the form
\begin{equation}
 u_t=\mathcal{A}u+\mathcal{N}(t,u)\ \text{in }\Omega,\qquad \mathcal{B}u=g\ \text{on }\partial\Omega,
 \label{eq:semilinear}
\end{equation}
where the linear operator $\mathcal{A}$ (e.g., $\nu\Delta$) is treated implicitly and $\mathcal{N}$ (reaction, advection, forcing) explicitly. Let $A$ approximate $\mathcal{A}$ as above, let $\mathcal{N}_h(t,U)\in\R^N$ approximate $\mathcal{N}$ at the physical nodes, and let $t_n=n\,\Delta t$. The third-order semi-implicit backward differentiation formula (SBDF3)~\cite{ascher1995} with the closure \eqref{eq:ghost_closure} at the new time level reads
\begin{equation}
 \begin{aligned}
 \tfrac{11}{6}\bar U^{n+1}-\Delta t\,AU^{n+1}
 &=3\bar U^{n}-\tfrac32\bar U^{n-1}+\tfrac13\bar U^{n-2}
 +\Delta t\bigl[3\mathcal{N}_h^{n}-3\mathcal{N}_h^{n-1}+\mathcal{N}_h^{n-2}\bigr],\\
 BU^{n+1}&=g(t_{n+1}),
 \end{aligned}
 \label{eq:sbdf3}
\end{equation}
with $\mathcal{N}_h^{j}=\mathcal{N}_h(t_j,U^j)$. Each step is a linear system of the form \eqref{eq:ghost_closure} whose matrix,
\begin{equation}
 \mathcal{K}=\begin{pmatrix}\tfrac{11}{6}\,[\,I\ 0\,]-\Delta t\,A\\ B\end{pmatrix}\in\R^{N_f\times N_f},
 \label{eq:sbdf3_matrix}
\end{equation}
does not change between steps; here $[\,I\ 0\,]$ selects the physical values of $U$, and the ghost values are recomputed at every step. The starting states at $t_0$, $t_1$, and $t_2$ are the exact solution in our tests (Section~\ref{sec:forward_results}). The time step is
\begin{equation}
 \Delta t=\frac{T}{n_t},\qquad n_t=\Bigl\lceil\frac{T}{\min\bigl(c_1h,\,c_2h^{p/3}\bigr)}\Bigr\rceil,
 \label{eq:dt_rule}
\end{equation}
with constants $c_1$ and $c_2$ given in Section~\ref{sec:forward_results}. The term $c_1h$ is the CFL restriction of the explicit terms, and $c_2h^{p/3}$ makes the $O(\Delta t^3)$ error of SBDF3 match the spatial error $O(h^p)$. For $p>3$, the second term is smaller only when $h<(c_1/c_2)^{3/(p-3)}$; for larger $h$, $\Delta t=c_1h$ and the time error is formally $O(h^3)$.

The integration is a single compiled JAX function whose time loop evaluates $\mathcal{N}_h$ with sparse products and solves with the factors of $\mathcal{K}$ (Section~\ref{sec:forward_linear_algebra}, Listing~\ref{lst:sbdf3} in Appendix~\ref{sec:app_listings_solvers}); JAX differentiates the loop in the parameters of $\mathcal{N}$.
\subsection{Hyperviscosity}
\label{sec:forward_hyperviscosity}
For advection, $\mathcal{N}$ contains $-v\cdot\nabla u$ with a velocity $v$ of maximum speed $v_{\max}$. Let $\tilde D_a\in\R^{N_f\times N_f}$ apply the stencils of $D_a$ at all nodes, ghost nodes included. On scattered nodes these square matrices can have eigenvalues with small positive real parts that destabilize the discretized advection over long times~\cite{platte2006eigenvalue,fornberg2011stabilization,shankar2018hyperviscosity}. Following~\cite{shankar2018hyperviscosity}, we damp these modes by adding $\gamma H U$ to the implicit operator, $A=\nu L+\gamma H$, with the power $k=p/2+1$, where $H\in\R^{N\times N_f}$ is a direct RBF-FD approximation of $\Delta^k$ whose stencils use polynomials of degree $2k-1$, the polyharmonic spline $r^{2k+1}$, and $2\binom{2k-1+d}{d}+1$ nodes. The rows of $H$ are set to zero at the boundary nodes, which keeps the physical equation in the rows that determine the ghost values, and at the interior nodes within distance $2h$ of $\chib$. There the advection is weak (our velocities vanish on the boundary), and the direct approximation of $\Delta^k$ can move eigenvalues of the reduced operator into the right half-plane. We set $\gamma=\sum_{a=1}^d\gamma_a$, where $\gamma_a$ balances the damping $|\gamma_a|(d\hat k^2)^k$ of the test mode $e_{\hat k}$ defined below against its growth $v_{\max}|\mu_a|\hat k^{q_a}$,
\begin{equation}
 \gamma_a=(-1)^{k+1}\,d^{-k}\,2^{\,q_a-2k}\,h^{\,2k-q_a}\,v_{\max}\,|\mu_a|,
 \label{eq:gamma}
\end{equation}
and the sign $(-1)^{k+1}$ makes $\gamma\Delta^k$ dissipative. The constant $\mu_a$ is the real part of the rightmost eigenvalue of $\tilde D_a$, and $q_a$ compares the error of $D_a$ on the oscillatory function $e_{\hat k}(x)=\exp\bigl(\mathrm{i}\hat k\,\mathbf 1\cdot x\bigr)$, $\hat k=2/h$, with $\mu_a$,
\begin{equation}
 q_a=\frac{\log\norm{\mathrm{i}\hat k\,\bar e_{\hat k}-D_ae_{\hat k}}-\log\bigl(|\mu_a|\,\norm{\bar e_{\hat k}}\bigr)}{\log\hat k}.
\end{equation}
We compute $\mu_a$ once per cloud and order by Arnoldi iteration on the GPU; $\gamma_a$ is fixed in time. However, since its stencils reproduce polynomials of degree $2k-1$ only, $H$ is not a consistent approximation of $\Delta^k$; $HU$ merely stays bounded as $h\to0$ for smooth $u$. If the $|\mu_a|$ stay bounded, then $\gamma=O(h^{2k-q})$ with $q=\max_aq_a$, and the added term is $O(h^{p+2-q})$, which for $q<2$ is of higher order than the truncation error $O(h^p)$.

\subsection{Linear algebra and stability checks}
\label{sec:forward_linear_algebra}
Every time step of SBDF3, and every elliptic solve, requires a sparse solve with the same matrix, so the factorization should be computed once and reused for thousands of substitutions. On the GPU, sparse direct factorizations are provided by NVIDIA's cuDSS~\cite{cudss}. However, JAX offers little support for this task. Its sparse matrix formats are experimental, and it has no sparse direct solver that keeps a factorization between calls. Moreover, a call into an external library is opaque to JAX; it can be compiled into a \texttt{jit}-transformed program, batched by \texttt{vmap}, or differentiated by \texttt{jvp} and \texttt{vjp} only if each of these behaviors is supplied explicitly. A naive wrapper, such as a host callback, would copy every right-hand side to the host, serialize the time loop, and break differentiation. We therefore built the following three components.

\paragraph{Sparse storage} Our operators remain in the coordinate (COO) form in which they are assembled (Section~\ref{sec:forward_operators}), and their products are gathers followed by scatter-adds that are differentiable in the weights and the vector (Listing~\ref{lst:matvec} in Appendix~\ref{sec:app_listings_solvers}). For cuDSS, we convert the matrix of \eqref{eq:ghost_closure}, or $\mathcal{K}$ in \eqref{eq:sbdf3_matrix}, to compressed sparse row (CSR) form on the GPU once per sparsity pattern, and for a nonsymmetric matrix we also form the CSR form of its transpose.

\paragraph{A native cuDSS primitive} We wrote a C++ extension that registers cuDSS solves as foreign-function (FFI) targets of XLA, the compiler behind JAX~\cite{jax2018github}. At setup, once per cloud and order, the extension runs the analysis and factorization phases of cuDSS and keeps the factors in GPU memory, in a registry indexed by an integer handle; for a nonsymmetric matrix it factors $\mathcal{K}^T$ as well. Inside a compiled program, a solve is an FFI call that receives the right-hand sides as a device buffer and the handle as an attribute, orders itself after the preceding work on JAX's CUDA stream through an event, and runs the cuDSS solve phase, so every time step costs one forward and one backward substitution and no data leave the GPU. Since the call reads shared native state, it is marked as having a side effect, so that the compiler neither removes nor reorders it. A matrix that depends on parameters, as \eqref{eq:variable_coefficient} does through $K$, keeps its sparsity pattern; a second target then copies the new values into the factor and runs the refactorization phase of cuDSS, which reuses the symbolic analysis~\cite{davis2006direct,davis2016survey}. A solution $U$ of $\mathcal{K}U=b$ is accepted only if its normwise backward error~\cite{rigal1967compatibility,higham2002accuracy} satisfies $\norm{\mathcal{K}U-b}\leq10^{-8}\bigl(\norm{\mathcal{K}}_F\norm{U}+\norm{b}\bigr)$, and we stop otherwise.

\paragraph{Differentiation and batching} The FFI call itself has neither a derivative nor a batching rule, and differentiating through the factorization would be both expensive and unnecessary. Two wrappers supply both rules (Listing~\ref{lst:cudss}). First, \texttt{jax.lax.custom\_linear\_solve}~\cite{jax2018github} declares the call a solution of $\mathcal{K}U=b$, with the COO product as the matrix action and the cuDSS solve with $\mathcal{K}^{T}$ as its transpose. JAX then differentiates $U$ implicitly, solving $\mathcal{K}\dot U=\dot b-\dot{\mathcal K}U$ with the same factors in forward mode and with $\mathcal{K}^{T}$ in reverse mode~\cite{giles2008collected}; the factorization is never differentiated. Second, without a batching rule, \texttt{vmap} would call cuDSS once per tangent direction; our rule instead passes a batch of right-hand sides to cuDSS as the columns of a single solve. A forward-mode derivative in $n_\theta$ directions thus costs, per step, one substitution for the state and one with $n_\theta$ right-hand sides for the tangents.
\begin{lstlisting}[float=tbp,language=Python,basicstyle=\ttfamily\lstsize,caption={A factored cuDSS system as a differentiable JAX solve with batched right-hand sides; \texttt{handle} refers to the cuDSS factors of $\mathcal{K}$ and \texttt{matvec} to Listing~\ref{lst:matvec}.},label={lst:cudss}]
@custom_vmap
def direct(b):                            # b: (N_f,) or (N_f, n_rhs)
    return cudss_solve(handle, b)         # jax.ffi call into cuDSS

@direct.def_vmap
def direct_batched(axis_size, in_batched, b):
    cols = jnp.moveaxis(b, 0, -1)         # batch axis -> extra columns
    return jnp.moveaxis(cudss_solve(handle, cols), -1, 0), True

# direct_T: the same two definitions with cudss_solve_transposed

def solve(b):
    return jax.lax.custom_linear_solve(
        lambda u: matvec(vals, u), b,
        solve=lambda _, rhs: direct(rhs),
        transpose_solve=lambda _, rhs: direct_T(rhs))
\end{lstlisting}

Finally, we detect unstable discretizations before solving by checking the reduced operator \eqref{eq:reduced_operator} of every implicit operator $A$ of Section~\ref{sec:forward_results}; with Neumann conditions we use $M-\tfrac1N\mathbf 1\mathbf 1^T$, whose eigenvalues are those of $M$ with the zero eigenvalue of the constants replaced by $-1$. The matrix $M$ is never formed; a product with it needs one solve with the factored block $A_{\mathrm{BG}}$ or $B_{\mathrm G}$. We compute its rightmost eigenvalues by a restarted Arnoldi iteration~\cite{arnoldi1951principle,sorensen1992implicit,stewart2002krylov} on the GPU and stop if a converged one has a positive real part or the iteration does not converge. However, this check is necessary but not sufficient for stability; it ignores the explicit terms, and since BDF3 is $A(\alpha)$-stable~\cite{widlund1967note} with $\alpha\approx86^\circ$~\cite{hairer1996} but not $A$-stable, eigenvalues with negative real part close to the imaginary axis can still cause instability. Section~\ref{sec:analysis} states the stability assumption of our error estimates.

\section{Inverse problems}
\label{sec:inverse}
\suppressfloats[t]
We now use our differentiable forward solvers to recover parameters of a PDE from values of its solution. Let $\theta\in\R^{n_\theta}$ be parameters of the problems of Section~\ref{sec:forward}, such as the coefficients of a permeability, of a reaction law, or of a velocity field. We recover $\theta$ by minimizing the misfit of the discrete solution (Section~\ref{sec:inverse_formulation}). Since the discrete solution is a JAX function of $\theta$, we obtain its Jacobian by differentiating through the operators, sparse solves, and time loop of Section~\ref{sec:forward}. We solve the least-squares problem by Gauss--Newton and Levenberg--Marquardt iterations, warm-started from one node set to the next (Section~\ref{sec:inverse_optimization}). Algorithm~\ref{alg:inverse} lists the steps and marks the new ones.

\begin{algorithm}[!b]
\caption{Parameter recovery across a sequence of node sets.}
\label{alg:inverse}
\begin{algorithmic}[1]
\Require node sets from coarsest to finest; observations $u_{\mathrm{obs}}$ on each
\Ensure recovered parameters $\theta_h$ on each node set
\State $\theta\gets0$
\For{each node set, from the coarsest to the finest}
 \State set up the discretization once: stencils, weights, assembly, stability check, and factorization of Algorithm~\ref{alg:forward}; when $\theta$ enters the operator, each evaluation of $r$ refactors numerically
 \State\label{alg:inv_residual} define $r(\theta)$ by \eqref{eq:inverse_residual}, with $J$ formed in one forward-mode pass that solves the $n_\theta$ tangents together with the cuDSS factors and stores only $J$ (Section~\ref{sec:inverse_optimization}) \Comment{new}
 \If{this is the coarsest node set}
  \State minimize $\tfrac12\norm{r}^2$ by Levenberg--Marquardt from $\theta$
 \Else
  \Repeat\label{alg:inv_gn} \Comment{Gauss--Newton from the previous $\theta_h$}
   \State $\theta\gets\theta-2^{-j}J^{+}r$ with the smallest $j\leq6$ that does not increase $\norm{r}$
  \Until{the step is below $10^{-10}(1+\norm{\theta})$ or no $j$ decreases $\norm{r}$; fall back to Levenberg--Marquardt if needed}
 \EndIf
 \State $\theta_h\gets\theta$
\EndFor
\end{algorithmic}
\end{algorithm}

\subsection{Formulation}
\label{sec:inverse_formulation}
Let $U(\theta)\in\R^{N_f}$ be the discrete solution, which for an elliptic problem solves \eqref{eq:ghost_closure} and for a parabolic problem is the SBDF3 state \eqref{eq:sbdf3} at the final time $T$. Given observations $u_{\mathrm{obs}}\in\R^{N}$ at the physical nodes, we minimize $\tfrac12\norm{r(\theta)}^2$ with the residual
\begin{equation}
 r(\theta)=\frac{\bar U(\theta)-u_{\mathrm{obs}}}{\norm{u_{\mathrm{obs}}}}\in\R^{N}.
 \label{eq:inverse_residual}
\end{equation}
The forcing $f$ and the boundary data $g$ are fixed. Note that for the pure Neumann Darcy problem, $\theta$ also enters the right-hand side $g/K(\theta)$ of the boundary rows, and the constant of Section~\ref{sec:forward_ghosts} is chosen so that $\bar U(\theta)$ and $u_{\mathrm{obs}}$ have the same mean. The parameters may enter the implicit operator, as a coefficient $K(x;\theta)$ in \eqref{eq:variable_coefficient} does, or the explicit terms $\mathcal{N}$ of \eqref{eq:semilinear}, as a reaction law or a velocity
\begin{equation}
 v(x;\theta)=\sum_{i=1}^{n_\theta}\theta_iw_i(x)
 \label{eq:velocity_basis}
\end{equation}
with fixed divergence-free basis fields $w_i$ that vanish on $\partial\Omega$. For a velocity, the hyperviscosity coefficients \eqref{eq:gamma} are computed once, with the prescribed maximum speed $v_{\max}$ of the true velocity, and kept fixed during the optimization, so that the implicit operator does not depend on $\theta$.

\subsection{Optimization and derivatives}
\label{sec:inverse_optimization}
On the coarsest cloud, and as a fallback, we minimize with the method of Levenberg and Marquardt~\cite{levenberg1944,marquardt1963} in Optimistix~\cite{optimistix2024}, with the linear least-squares subproblems solved by Lineax~\cite{lineax2023}; on the finer clouds we use the warm-started Gauss--Newton iteration below. Both need the Jacobian $J(\theta)=\partial r/\partial\theta\in\R^{N\times n_\theta}$. Since $n_\theta$ is small and $N$ is large, $J$ is computed in forward mode, with $n_\theta$ directional derivatives propagated together with the state~\cite{griewank2008evaluating}.

Recall that each linear solve in the forward map is the primitive of Listing~\ref{lst:cudss}, whose tangents are solved with the factors of $\mathcal{K}$ (Section~\ref{sec:forward_linear_algebra}) for all $n_\theta$ tangent right-hand sides together. When $\theta$ enters only $\mathcal{N}$, the matrix $\mathcal{K}$ does not depend on $\theta$, $\dot{\mathcal{K}}=0$, and the derivative of the loop of Listing~\ref{lst:sbdf3} in Appendix~\ref{sec:app_listings_solvers} is the same loop with $n_\theta$ tangent states carried alongside the state.

In reverse mode, the same primitive uses the transpose solve with the factors of $\mathcal{K}$. The adjoint of $U=\mathcal{K}^{-1}b$ is $\check b=\mathcal{K}^{-T}\check U$ and $\check{\mathcal{K}}_{ij}=-\check b_iU_j$ on the nonzero pattern of $\mathcal{K}$. Reverse mode gives the gradient of $\tfrac12\norm{r}^2$ with one adjoint solve per solve of the forward map, independent of $n_\theta$, and thus pays off for gradient methods with many parameters. However, for a time loop it stores the trajectory unless it is checkpointed~\cite{griewank2000revolve}, and since our problems have $n_\theta\leq6$, we use forward mode throughout.

Our warm start exploits the sequence of clouds. The iteration starts from $\theta=0$ on the coarsest cloud and, on every finer cloud, from the parameters recovered on the next coarser one. Their distance from the minimizer on the finer cloud is at most the sum of the parameter errors on the two clouds. We then take damped Gauss--Newton steps~\cite{dennis1996numerical,nocedal2006numerical} $\theta\leftarrow\theta-2^{-j}J^{+}r$, where $J^{+}$ is the Moore--Penrose pseudoinverse~\cite{penrose1955generalized} of $J$, with the smallest $j\in\{0,\dots,6\}$ for which $\norm{r}$ decreases. The iteration stops when $\norm{2^{-j}J^{+}r}\leq10^{-10}(1+\norm{\theta})$, or when no such $j$ exists. Since the residual at the minimizer $\theta_h$ is of the size of the discretization error, the Gauss--Newton iteration contracts by a factor of about $L_J\norm{r(\theta_h)}/\sigma_h^2$ per step~\cite{bjorck1996}, where $L_J$ is a Lipschitz constant of $J$ and $\sigma_h=\sigma_{\min}(J(\theta_h))$. If the residual or $J$ is not finite, or after $30$ steps, we continue with Levenberg--Marquardt.

Optimistix linearizes $r$ once per step and stores the intermediate values of the forward map for later products with $J$, which for a time loop with $n_t$ steps takes memory proportional to $n_tN$. Our derivative rule avoids this by forming $J$ in a single forward pass (Listing~\ref{lst:dense_jvp} in Appendix~\ref{sec:app_listings_solvers}). Since the rule depends only on the primal value of $\theta$, the linearization stores the $N\times n_\theta$ matrix $J$ alone and applies it to tangent vectors, while the iterates are unchanged.

\section{Error estimates}
\label{sec:analysis}
We now derive error estimates for our forward and inverse problems. Each bounds an error by a defect (the truncation error, or the forward error at the true parameters) times an amplification factor (a stability constant, or the inverse of the smallest singular value of the Jacobian). Polynomial reproduction bounds the truncation error~\cite{davydov2016,davydov2019}. However, uniform stability of RBF-FD collocation on scattered nodes is open (Tominec, Larsson, and Heryudono~\cite{tominec2021} analyzed an oversampled least-squares variant), so we measure the stability constants as in~\cite{schaback2017,larsson2022} (Section~\ref{sec:forward_results}). Proofs are in Appendix~\ref{sec:proofs}.

We write $\Pi v\in\R^{N_f}$ and $\bar\Pi v\in\R^N$ for the values of a function $v$ at all nodes and at the physical nodes; since our manufactured solutions are defined on all of $\R^d$, $\Pi u$ includes ghost values. Norms are $2$-norms unless marked otherwise.

\subsection{Forward problems}
\label{sec:analysis_forward}

\paragraph{Truncation error}
Suppose that a stencil $x_{c1},\dots,x_{cn_s}$ at $x_c$ with weights $w_{ck}$ reproduces an operator $\mathcal{L}$ of order $s$ on polynomials of degree $\varpi\geq s$, and that $u$ is smooth on the ball of radius $\varrho_c=\max_k\norm{x_{ck}-x_c}$ about $x_c$. The stencil then errs by at most the Taylor remainder of degree $\varpi+1$,
\begin{equation}
 \Bigl|\sum_{k=1}^{n_s}w_{ck}\,u(x_{ck})-(\mathcal{L}u)(x_c)\Bigr|\leq\frac{\varrho_c^{\varpi+1}}{(\varpi+1)!}\,|u|_{C^{\varpi+1}}\sum_{k=1}^{n_s}|w_{ck}|
 \label{eq:consistency}
\end{equation}
(Lemma~\ref{lem:consistency} in Appendix~\ref{sec:proof_consistency}). Thus, if $\varrho_c\leq Ch$ and the scaled stencil sums
\begin{equation}
 W_L=\max_c h^2\sum_k|L_{ck}|,\qquad W_D=\max_{a,c}h\sum_k|(D_a)_{ck}|
 \label{eq:stencil_sums}
\end{equation}
stay bounded, then with $\varpi=p+1$ the right-hand side is $O(h^p)$ for $L$ and $O(h^{p+1})$ for $D_a$ and the normal derivative. Every row of the closure \eqref{eq:ghost_closure}, including the variable-coefficient \eqref{eq:variable_coefficient} and boundary rows, therefore has a truncation error of order $h^p$ or higher.

\paragraph{Elliptic problems}
Let $\tau=(\tau_A,\tau_B)=(A\Pi u-\bar\Pi f,\ B\Pi u-\tilde g)$ be the truncation error, where $\tilde g$ is the right-hand side of the boundary rows ($g|_{\chib}$, or $(g/K)|_{\chib}$ for the flux condition), and $E=\Pi u-U$ the error. For our Darcy problem, with Neumann rows and the multiplier $\lambda$ of Section~\ref{sec:forward_ghosts}, the error solves
\begin{equation*}
 AE-\lambda\mathbf 1=\tau_A,\qquad BE-\lambda\mathbf 1=\tau_B .
\end{equation*}
The boundary rows give the ghost errors, $E_{\mathrm G}=B_{\mathrm G}^{-1}(\tau_B+\lambda\mathbf 1-B_{\mathrm P}\bar E)$, and eliminating them leaves
\begin{equation}
 M\bar E=\tau_A-G\tau_B+\lambda\upsilon,\qquad G=A_{\mathrm G}B_{\mathrm G}^{-1},\quad \upsilon=\mathbf 1-G\mathbf 1,
 \label{eq:error_equation}
\end{equation}
with the reduced operator $M$ of \eqref{eq:reduced_operator}. The entries of $G$ are of order $h^{-1}$ (Laplacian weights over normal-derivative weights), but $\tau_B=O(h^{p+1})$ compensates. Since $M\mathbf 1=0$, the equation is solvable only if its right-hand side lies in the range of $M$~\cite{bochev2005finite}; this fixes $\lambda$, the defect of the discrete compatibility condition, and our choice of the mean removes the remaining constant.

\begin{theorem}[error of the elliptic problem]
\label{thm:elliptic}
Suppose that the boundary rows determine the ghost values ($B_{\mathrm G}$ nonsingular), that the constants are the only null vectors of $M$, and that the compatibility defect can be removed ($\zeta^T\upsilon\neq0$, where $\zeta$ spans the null space of $M^T$). Let $M_Q$ be $M$ acting on vectors of mean zero, which it maps one-to-one onto $\zeta^\perp$. Then the multiplier is $\lambda=-\zeta^T(\tau_A-G\tau_B)/\zeta^T\upsilon$, and the error at the physical nodes satisfies
\begin{equation}
 \norm{\bar E}\leq S_h\bigl(\norm{\tau_A}+\norm{G}\norm{\tau_B}\bigr),\qquad S_h=\norm{M_Q^{-1}}\,\frac{\norm{\upsilon}\norm{\zeta}}{|\zeta^T\upsilon|}.
 \label{eq:elliptic_bound}
\end{equation}
\end{theorem}
\begin{proof}[Proof sketch]
Multiplying \eqref{eq:error_equation} by $\zeta^T$ gives $\lambda$. The remaining right-hand side is the projection of $\tau_A-G\tau_B$ onto $\zeta^\perp$ along $\upsilon$, whose norm is the middle factor, and $M_Q^{-1}$ maps it to the error, which has mean zero (Appendix~\ref{sec:proof_elliptic}).
\end{proof}

The factor $\norm{M_Q^{-1}}$ is the discrete counterpart of the inverse of the smallest nonzero eigenvalue of the elliptic operator, and the second factor of $S_h$ equals $1$ when $\upsilon$ is parallel to $\zeta$. If $S_h$ is bounded uniformly in $h$, the error converges with order $p$, since $\norm{\tau_A}$ is of order $h^p$ and $\norm{G}\norm{\tau_B}$ of order $h^{p}$ or higher. For the Dirichlet closure, $BE=0$ gives $E_{\mathrm B}=0$, the boundary rows of $A$ determine $E_{\mathrm G}$ through $A_{\mathrm{BG}}$, and the error solves $M\bar E_{\mathrm I}=\tau_{A,\mathrm I}-A_{\mathrm{IG}}A_{\mathrm{BG}}^{-1}\tau_{A,\mathrm B}$, where $\tau_{A,\mathrm I}$ and $\tau_{A,\mathrm B}$ are the interior and boundary rows of $\tau_A$; this is~\cite[Thm.~1]{schaback2017} specialized to the ghost-node closure, with an elimination factor of order one. Note that our relative errors divide by $\norm{\bar\Pi u}\sim\sqrt N$, which cancels the factor $\sqrt N$ between $\norm{\tau}$ and \eqref{eq:consistency}. Since $S_h$ cannot be bounded a priori, Section~\ref{sec:forward_results} reports the effective stability constant
\begin{equation}
 \hat S_h=\frac{\norm{\bar E}}{\norm{\tau_A}}\leq S_h\Bigl(1+\frac{\norm{G}\norm{\tau_B}}{\norm{\tau_A}}\Bigr),
 \label{eq:effective_stability}
\end{equation}
and the error converges with the order of $\tau_A$ as long as $\hat S_h$ stays bounded (Appendix~\ref{sec:app_estimates}).

\paragraph{Time-dependent problems}
The same elimination applies in every SBDF3 step, whose truncation error $\tau^n$ is $O(h^p+\Delta t^3)$ (Lemma~\ref{lem:parabolic_consistency} in Appendix~\ref{sec:proof_sbdf3}); the explicit terms $\mathcal{N}_h$ add the usual exponential factor of the discrete Gronwall inequality~\cite{clark1987short,akrivis1999}.

\begin{theorem}[error of the time-dependent problem]
\label{thm:parabolic}
Consider the closure with Neumann rows, as for our reaction--diffusion problem, and let $M$ and $G$ be formed from the implicit operator $A$ of \eqref{eq:semilinear}. Let $\tau^n=(\tau^n_A,\tau^n_B)$ be the truncation error of step $n$ (Appendix~\ref{sec:proof_sbdf3}). Suppose that the SBDF3 recursion for $y'=My$ with forcing, $(\tfrac{11}{6}I-\Delta t M)y^{n+1}=3y^n-\tfrac32y^{n-1}+\tfrac13y^{n-2}+\Delta t\,w^{n+1}$, is stable, i.e., $\norm{y^n}\leq C_S(\max_{j\leq2}\norm{y^j}+\Delta t\sum_{j=3}^{n}\norm{w^j})$ for $n\Delta t\leq T$, with a constant $C_S\geq1$ independent of $h$ and $\Delta t$. Suppose further that the explicit terms $\mathcal{N}_h(t,V)$ depend only on the physical values $\bar V$, with Lipschitz constant $L_{\mathcal N}$. Then the error $E^n=\Pi u(t_n)-U^n$ satisfies
\begin{equation}
 \max_{0\leq n\leq T/\Delta t}\norm{\bar E^n}\leq C_S\,e^{7C_SL_{\mathcal N}T}\Bigl(\max_{j\leq2}\norm{\bar E^j}+T\max_{3\leq n\leq T/\Delta t}\bigl(\norm{\tau^n_A}+\norm{G}\norm{\tau^n_B}\bigr)\Bigr).
 \label{eq:parabolic_bound}
\end{equation}
\end{theorem}
\begin{proof}[Proof sketch]
Eliminating the ghost errors gives SBDF3 for $\bar E^\prime=M\bar E$, forced by the truncation error (its boundary part passed through $G$) and by differences of the explicit terms, at most $7L_{\mathcal N}$ times earlier errors ($7=3+3+1$, the sum of the absolute extrapolation coefficients). Stability and the discrete Gronwall inequality give \eqref{eq:parabolic_bound} (Appendix~\ref{sec:proof_parabolic}).
\end{proof}

The stability assumption holds, e.g., if $M$ is diagonalizable with eigenvector condition number bounded uniformly in $h$ and its spectrum lies in a sector~\cite{lubich1991resolvent} inside the $A(\alpha)$-stability region of BDF3, $\alpha\approx86^\circ$~\cite{hairer1996}; the spectral check of Section~\ref{sec:forward_linear_algebra} tests only a necessary condition. For our locally Lipschitz reaction terms, the theorem holds with the Lipschitz constant on a max-norm neighborhood of $\bar\Pi u$, which the discrete solution does not leave for small $h$ and $\Delta t$, since its max-norm error is of order $\sqrt N(h^p+\Delta t^3)$ and $p>d/2$. In contrast, the discretized advection depends on the ghost values with Lipschitz constant of order $h^{-1}$, so the theorem gives no uniform bound for advection--diffusion. However, for the Dirichlet closure an analogous bound without the exponential factor holds if the full recursion including advection is stable (with $C_S$ its stability constant), as the hyperviscosity is designed to ensure~\cite{fornberg2011stabilization,shankar2018hyperviscosity}.

\subsection{Inverse problems}
\label{sec:analysis_inverse}
For the inverse problems, the forward error is the defect. With exact observations $u_{\mathrm{obs}}=\bar\Pi u(\theta_\star)$, the residual \eqref{eq:inverse_residual} at the true parameters is the relative forward error, $r_h(\theta_\star)=-\bar E(\theta_\star)/\norm{u_{\mathrm{obs}}}$, where the subscript $h$ marks the dependence of $r$ and the recovered parameters $\theta_h$ on the cloud. Let $J_h=\partial r_h/\partial\theta$ and $\delta=\theta_h-\theta_\star$. Linearizing, $r_h(\theta_\star)\approx r_h(\theta_h)-J_h\delta$, and since $J_h^Tr_h(\theta_h)=0$ at a minimizer, applying the pseudoinverse $J_h^+$ gives
\begin{equation*}
 \delta\approx-J_h^{+}r_h(\theta_\star).
\end{equation*}
To first order, only the component of the forward error in the range of $J_h$ (i.e., of the sensitivities) enters, amplified by at most $1/\sigma_{\min}(J_h)$. Adding the linearization remainder makes this precise~\cite{wedin1973,bjorck1996}.

\begin{theorem}[parameter error, a posteriori]
\label{thm:inverse_aposteriori}
Let $\theta_h$ be a stationary point of the least-squares problem, $J_h(\theta_h)^Tr_h(\theta_h)=0$, at which the Jacobian has full rank, $\sigma_h=\sigma_{\min}(J_h(\theta_h))>0$, and let $L_J$ be a Lipschitz constant of $J_h$ on the segment between $\theta_\star$ and $\theta_h$. Then
\begin{equation}
 \norm{\delta}\leq\norm{J_h(\theta_h)^{+}r_h(\theta_\star)}+\frac{L_J}{2\sigma_h}\norm{\delta}^2\leq\frac{\norm{r_h(\theta_\star)}}{\sigma_h}+\frac{L_J}{2\sigma_h}\norm{\delta}^2.
 \label{eq:inverse_bound}
\end{equation}
\end{theorem}
\begin{proof}[Proof sketch]
With $\hat J=J_h(\theta_h)$ and the mean $\bar J$ of $J_h$ on the segment from $\theta_\star$ to $\theta_h$, $r_h(\theta_h)=r_h(\theta_\star)+\bar J\delta$, and the stationarity condition gives $\delta=-\hat J^+r_h(\theta_\star)+\hat J^+(\hat J-\bar J)\delta$, where $\norm{\hat J-\bar J}\leq\tfrac12L_J\norm{\delta}$ (Appendix~\ref{sec:proof_aposteriori}).
\end{proof}

In the spirit of a priori analyses of parameter identification with finite elements~\cite{rannacher2005}, our next theorem shows that the stationary point assumed in Theorem~\ref{thm:inverse_aposteriori} exists near $\theta_\star$ when the Jacobian at the truth has full rank and the forward error is small.

\begin{theorem}[parameter error, a priori]
\label{thm:inverse_apriori}
Suppose that the Jacobian at the true parameters has full rank, $\sigma_0=\sigma_{\min}(J_h(\theta_\star))>0$. Let $L_0$ be a Lipschitz constant of $J_h$ on the ball of radius $\sigma_0/(3L_0)$ about $\theta_\star$. If $\norm{r_h(\theta_\star)}<5\sigma_0^2/(36L_0)$, then the least-squares problem has a local minimizer in this ball, and every stationary point $\theta_h$ in the ball satisfies $\norm{\theta_h-\theta_\star}\leq2\norm{r_h(\theta_\star)}/\sigma_0$.
\end{theorem}
\begin{proof}[Proof sketch]
On the boundary of the ball the residual exceeds $\norm{r_h(\theta_\star)}$, so the minimum of $\norm{r_h}$ over the ball is an interior stationary point. At any stationary point in the ball, $\sigma_h\geq\sigma_0-L_0\norm{\delta}$ by Weyl's inequality~\cite{horn2012matrix}, and \eqref{eq:inverse_bound} gives the bound (Appendix~\ref{sec:proof_apriori}).
\end{proof}

For these stationary points, the parameter error is therefore $O(\norm{r_h(\theta_\star)})$ if $\sigma_0$ is bounded away from zero and $L_0$ is bounded, uniformly in $h$. By Theorems~\ref{thm:elliptic} and~\ref{thm:parabolic}, this rate is $O(h^p)$ for the elliptic and reaction--diffusion problems if $S_h$ and $C_S$ are bounded uniformly; for transport, we rely on the measured forward errors. Since forward-mode differentiation discretizes the tangent equations in the same way, we expect $J_h(\theta_\star)$ to converge to the nodal values of the continuous sensitivities, and $\sigma_0$ to a positive limit if the continuous parameter-to-observation map is injective to first order. We do not prove stability of the tangent scheme; however, $\sigma_{\min}(J_h)$ changes little over our finest clouds (Section~\ref{sec:inverse_results}).

\section{Numerical results}
\label{sec:results}
We first describe the node sets, then the convergence of the forward problems and of the recovered parameters under refinement. All computations run on one NVIDIA RTX 5090 in double precision.
\subsection{Node generation}
\label{sec:nodegen_results}
We now present results for our node generator on two domains. The first is the H-channel
\begin{equation}
 \begin{aligned}\Omega_2={}&\bigl\{(x,y):\ 0<x<2,\ \bigl|y-\tfrac12\bigr|<\tfrac12+0.06\sin(\pi x/2)\bigr\}\\&\setminus\bigl([0,0.9]\cup[1.1,2]\bigr)\times[0.4,0.6],\end{aligned}
\end{equation}
which consists of two channels with curved outer walls joined by a slot of width $0.2$. Its boundary consists of two curved and ten straight pieces and has four reentrant corners at the ends of the slot. The second, the treasure chest, is a box with a superelliptic cap,
\begin{equation}
 \Omega_3=\Omega_{\mathrm{box}}\cup\Bigl\{(x,y,z):\ x\geq1.2,\ \Bigl(\frac{x-1.2}{0.6}\Bigr)^2+\Bigl|\frac{y}{0.6}\Bigr|^4+\Bigl|\frac{z}{0.35}\Bigr|^4<1\Bigr\},
 \label{eq:treasure_chest}
\end{equation}
with the box $\Omega_{\mathrm{box}}=(0,1.2)\times(-0.6,0.6)\times(-0.35,0.35)$. Its boundary consists of five faces of the box, four planar pieces in the plane $x=1.2$ around the base of the cap, and the cap. The cap is given to the node generator only as seed nodes with normals, together with its rim. For each $h\in\{0.16,0.12,0.08,0.06,0.04\}$ we generate one cloud with the parameters of Table~\ref{tab:node_gen_parameters} in Appendix~\ref{sec:app_parameters}, giving $N=141$ to $1160$ physical nodes in the H-channel and $N=805$ to $25636$ in the treasure chest. Figure~\ref{fig:node_clouds} shows examples, including the refined boundary layer and the ghost nodes at the slot corners. On every cloud, all interior nodes lie inside the exact domain and all ghost nodes outside it. Distinct nodes are at least $h/2$ apart, except in the H-channel at $h=0.16$, where the walls between the two channels are $0.2=1.25h$ thick and the ghost nodes of their upper and lower faces are $h/4$ apart. Figure~\ref{fig:timings} of Section~\ref{sec:forward_results} gives the generation times.

\begin{figure}[tbp]
\centering
\includegraphics[width=0.34\textwidth]{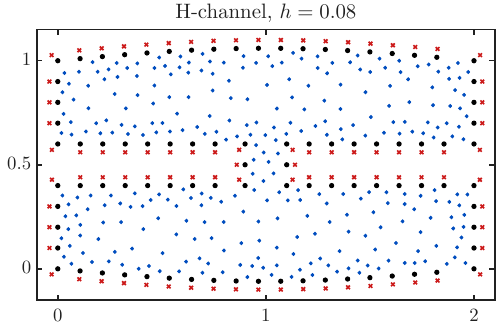}\hspace{3mm}
\includegraphics[width=0.21\textwidth]{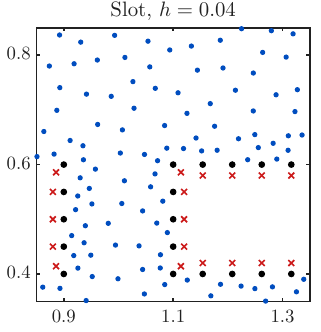}

\medskip
\includegraphics[width=0.27\textwidth]{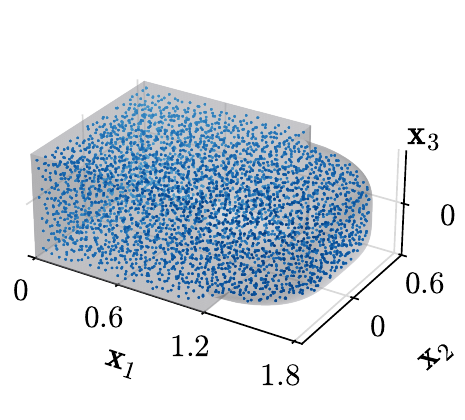}\hspace{3mm}
\includegraphics[width=0.27\textwidth]{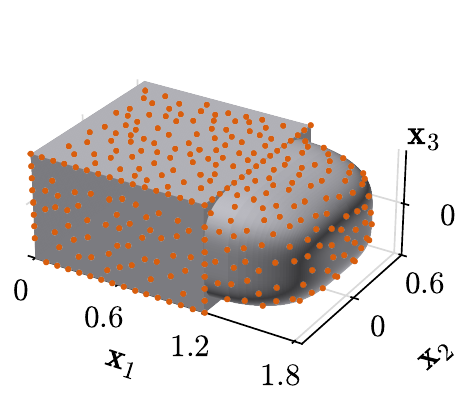}
\caption{Node clouds. Top: interior nodes (blue), boundary nodes (black), and ghost nodes (red crosses) of the H-channel at $h=0.08$ (left) and of its slot at $h=0.04$ (right). Bottom: the treasure chest at $h=0.08$, with the interior nodes in front of the translucent boundary (left), shaded darker with distance from the viewer, and the boundary nodes on the visible part of the boundary (right).}
\label{fig:node_clouds}
\end{figure}

\begin{figure}[tbp]
\centering
\begin{tabular}{@{}c@{\hspace{1mm}}c@{\hspace{1mm}}c@{}}
\includegraphics[width=0.325\textwidth]{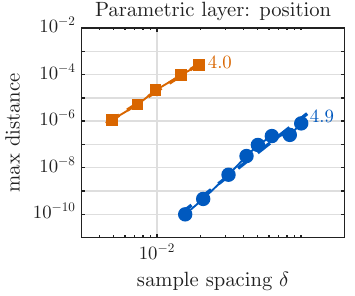} & \includegraphics[width=0.325\textwidth]{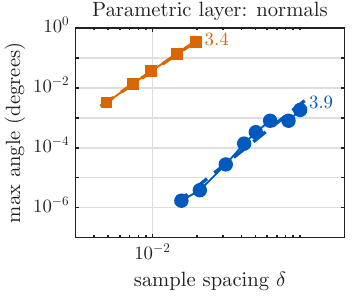} & \includegraphics[width=0.325\textwidth]{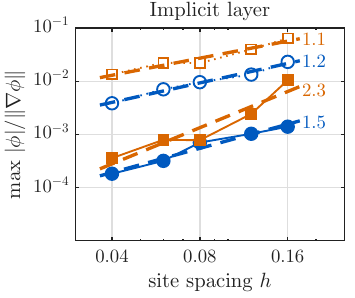}\\
\multicolumn{3}{c}{\includegraphics{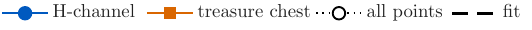}}
\end{tabular}
\caption{Errors of the geometric layers against the exact boundaries. Left and middle: largest distance and normal angle of the parametric layer on the curved pieces (the walls of the H-channel and the cap of the treasure chest, with normals measured at least $0.02$ from the rim of the cap) against the spacing $\delta$ of the seed nodes. Right: largest distance $|\phi|/\norm{\nabla\phi}$ of the implicit layer at exact boundary points against the spacing $h$ of its sites, at least $2h$ from edges and corners (filled markers) and over all points (open markers). The dashed lines are least-squares fits, labeled with their slopes.}
\label{fig:geometry}
\end{figure}

\paragraph{Geometric layers}
Figure~\ref{fig:geometry} shows how the two layers of Section~\ref{sec:node_gen} converge to the exact boundary. The errors are measured at $2\times10^4$ random points per curved piece, independent of the input data. The parametric layer is fit to the seed nodes, so its error depends on their spacing $\delta$, and on $h$ only through refinements of the cover, which refit the cap of the treasure chest on smaller supports from the same seed nodes and thus increase its distance error (Appendix~\ref{sec:app_cover_refinement}). The error in the parametric layer decreases roughly like $\delta^{4.9}$ on the walls of the H-channel and like $\delta^{4.0}$ on the cap; its normals converge like $\delta^{3.9}$ and $\delta^{3.4}$. Our clouds use $65$ seed nodes per wall ($\delta=0.031$) and $41$ seed nodes per parameter direction on the cap ($\delta=0.0073$).

In contrast, the implicit layer is fit to boundary sites at spacing $h$, and away from edges and corners its distance error decreases like $h^{1.5}$ in the H-channel and $h^{2.3}$ in the treasure chest. At edges and corners, however, $\phi$ rounds the crease over a width proportional to $h$, and the largest error over all points decreases only like $h$. Note that this rounding does not move the boundary nodes, which lie on the parametric layer; it affects only points within a fraction of $h$ of an edge, where the clearance condition~\eqref{eq:interior_spacing} excludes interior nodes anyway.

\paragraph{Parallel algorithms}
Finally, we test the two parallel building blocks of our generator in isolation. On $n=10^3$ to $10^6$ uniformly random candidates in the unit square and the unit cube, with the common radius $r=(0.6/n)^{1/d}$, the parallel greedy independent set of Section~\ref{sec:node_gen_mis} equals the sequential greedy set after $5$ to $9$ rounds, a number that grows roughly like $\log n$. Grouped evaluation of the implicit layer (Section~\ref{sec:node_gen_grouping}) at the interior nodes costs $1.3$ to $1.4$ times the work of an exact sum over the balls containing each point, compared with $1.6$ to $3.0$ times with padding to the largest overlap count.

\subsection{Forward problems}
\label{sec:forward_results}

\subsubsection{Test problems}
\label{sec:forward_tests}
We now test our forward solvers on three problems on the clouds of Section~\ref{sec:nodegen_results}, with orders $p\in\{2,4,6\}$. Each problem has a manufactured solution $u_\star$ that is smooth on all of $\R^d$ (in order for us to be able to test high-order convergence rates). The forcing $f$ and the boundary data $g$ are computed from $u_\star$ and the true parameters $\theta_\star$, with the normals $n_b$ of the boundary nodes, so that the difference between the discrete solution and $u_\star$ is the discretization error, which should decrease at order $p$ (Section~\ref{sec:analysis}).

\paragraph{Darcy flow}
The pressure $u$ solves
\begin{equation}
 -\nabla\cdot\bigl(K\nabla u\bigr)=f\ \ \text{in }\Omega,\qquad K\,\partial_nu=g\ \ \text{on }\partial\Omega,
 \label{eq:darcy_k}
\end{equation}
with the permeability
\begin{equation}
 K(x;\theta)=\exp\bigl(F(x)\cdot\theta\bigr),
 \label{eq:permeability}
\end{equation}
where $F(x)$ is a vector of features and $\theta$ its coefficients, for the first five families of Table~\ref{tab:families}. Since the flux condition holds on all of $\partial\Omega$, $u$ is determined only up to a constant, which we fix by matching the mean of $u_\star$ over the physical nodes (Section~\ref{sec:forward_ghosts}). The manufactured solution is
\begin{equation}
 u_\star(x)=\begin{cases}
 \sin(1.1x-0.8y)+0.2x^2+0.1xy-0.05y^2, & d=2,\\
 \sin(0.9x-0.7y+0.5z)+0.12x^2-0.08yz+0.07z^2, & d=3.
 \end{cases}
 \label{eq:ustar}
\end{equation}

\begin{table}[tbp]
\centering
\caption{Parameter families of the permeability $K$ and of the reaction law $R$: feature vectors $F(x)$ and true parameters $\theta_\star$. The cubic law is used only for $R$.}
\label{tab:families}
\small
\begin{tabular}{@{}llll@{}}
\toprule
Family & $d$ & $F(x)$ & $\theta_\star$\\
\midrule
constant & 2 & $(1)$ & $(0.12)$\\
 & 3 & $(1)$ & $(0.09)$\\\addlinespace
affine & 2 & $(1,x,y)$ & $(0.08,0.22,-0.16)$\\
 & 3 & $(1,x,y,z)$ & $(0.06,0.18,-0.11,0.09)$\\\addlinespace
quadratic & 2 & $(1,x,y,x^2,y^2)$ & $(0.05,0.18,-0.12,0.08,-0.04)$\\
 & 3 & $(1,x,y,z,x^2,z^2)$ & $(0.04,0.15,-0.08,0.07,0.09,-0.05)$\\\addlinespace
cross term & 2 & $(1,x,y,xy)$ & $(0.10,0.30,-0.22,0.12)$\\
 & 3 & $(1,x,y,z,xz)$ & $(0.07,0.20,-0.15,0.10,0.14)$\\\addlinespace
sine product & 2 & $(1,\sin\pi x\sin\pi y)$ & $(0.08,0.22)$\\
 & 3 & $(1,\sin\pi x\sin\pi y\sin\pi z)$ & $(0.06,0.20)$\\\addlinespace
cubic & 2, 3 & $R(u;\theta)=\theta_1u+\theta_2u^3$ & $(0.16789,0.02706)$\\
\bottomrule
\end{tabular}
\end{table}

\paragraph{Reaction--diffusion}
The concentration $u$ solves
\begin{equation}
 u_t=\nu\Delta u+R(u,x;\theta)+f\ \ \text{in }\Omega\times(0,T],\qquad \partial_nu=g\ \ \text{on }\partial\Omega,
 \label{eq:rd}
\end{equation}
with $T=\tfrac12$, $\nu=0.08$ for $d=2$ and $\nu=0.06$ for $d=3$. The reaction law is either quadratic in $u$ with a spatially varying rate,
\begin{equation}
 R(u,x;\theta)=\exp\bigl(F(x)\cdot\theta\bigr)\,u^2,
\end{equation}
with the first five families of Table~\ref{tab:families}, or the cubic law $R(u;\theta)=\theta_1u+\theta_2u^3$, the sixth family of Table~\ref{tab:families}. The manufactured solution is $u_\star(x,t)=e^{-0.35t}\,u_0(x)$, with $u_0$ given by \eqref{eq:ustar}.

\begin{figure}[tbp]
\centering
\begin{tabular}{@{}c@{\hspace{1mm}}c@{}}
\textbf{H-channel} & \textbf{Treasure chest}\\
\includegraphics[scale=0.8]{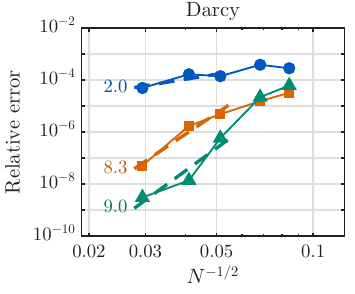} & \includegraphics[scale=0.8]{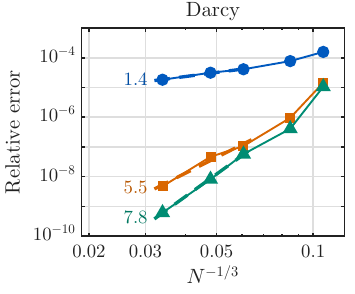}\\
\includegraphics[scale=0.8]{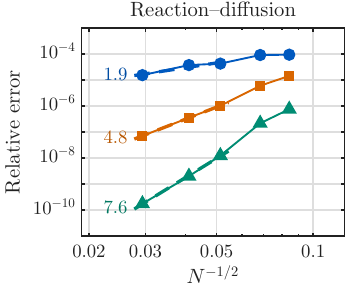} & \includegraphics[scale=0.8]{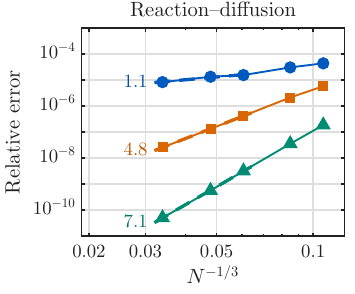}\\
\includegraphics[scale=0.8]{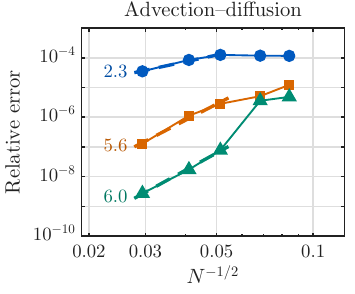} & \includegraphics[scale=0.8]{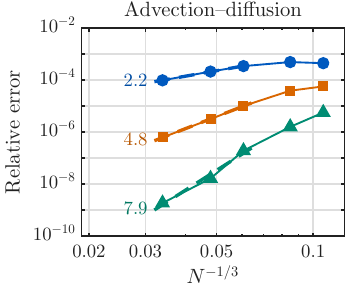}\\
\multicolumn{2}{c}{\includegraphics[scale=0.8]{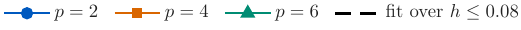}}
\end{tabular}
\caption{Relative errors of the forward problems against $N^{-1/d}$, where $N$ is the number of physical nodes. Each curve is the largest error over the parameter families. The dashed lines are least-squares fits over the three finest clouds, labeled with their slopes.}
\label{fig:forward_convergence}
\end{figure}

\begin{figure}[tbp]
\centering
\begin{tabular}{@{}c@{\hspace{2mm}}c@{}}
\small H-channel & \small Treasure chest, section $y=0$\\
\includegraphics[width=0.40\textwidth]{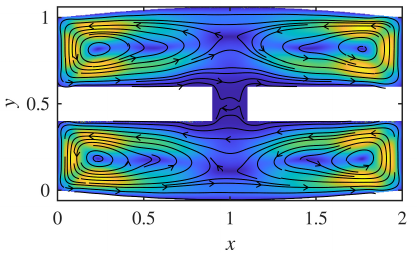} & \includegraphics[width=0.40\textwidth]{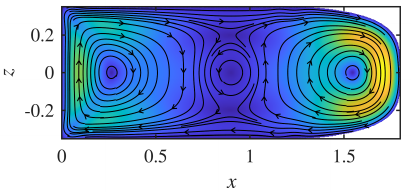}\\
\multicolumn{2}{c}{\includegraphics[width=0.32\textwidth]{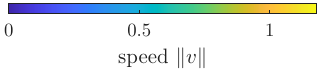}}
\end{tabular}
\caption{The advection--diffusion velocity $1.18\,w$ of \eqref{eq:velocity_field} at $t=T$, colored by speed, with streamlines. In the treasure chest the panel shows the section $y=0$ and the components in it.}
\label{fig:velocity_ad}
\end{figure}

\paragraph{Advection--diffusion}
The concentration $u$ solves
\begin{equation}
 u_t+v\cdot\nabla u=\nu\Delta u+f\ \ \text{in }\Omega\times(0,T],\qquad u=g\ \ \text{on }\partial\Omega,
 \label{eq:ad}
\end{equation}
with $T=\tfrac12$ and $\nu=0.05$. The velocity $v(x,t)=1.18\sin(\pi t)\,w(x)$ grows from zero at $t=0$ to its full strength at $t=T$. The field $w$ is the normalized curl of a stream function ($d=2$) or a vector potential ($d=3$) that vanishes quadratically on the boundary,
\begin{equation}
 w=\frac{\tilde w}{\max_{x\in X_0}\norm{\tilde w(x)}},\qquad
 \tilde w=\begin{cases}(\partial_y\Psi,\,-\partial_x\Psi), & d=2,\\ \nabla\times(\Psi e_y), & d=3,\end{cases}\qquad \Psi=\phi_{0,+}^2\,m_0 .
 \label{eq:velocity_field}
\end{equation}
Here $\phi_{0,+}=\max(\phi_0,0)$ is the positive part of the implicit layer $\phi_0$ of the cloud with $h=0.16$, $m_0$ is a quadratic about the center of the domain, and $X_0$ is a fixed uniform grid in $\Omega$ (Appendix~\ref{sec:app_velocity}). The field $w$, shown in Figure~\ref{fig:velocity_ad}, is divergence-free, has maximum speed one, vanishes on the zero set of $\phi_0$ (an approximation of $\partial\Omega$), and is the same on every cloud. The manufactured solution is
\begin{equation}
 u_\star(x,t)=1+0.15\cos(0.8t)\prod_{a=1}^{d}\cos(\pi x_a).
 \label{eq:ustar_ad}
\end{equation}
With a P\'eclet number of about $24$ (based on the maximum speed and unit length), the problem is moderately advection dominated, and we add the hyperviscosity of Section~\ref{sec:forward_hyperviscosity}.

\begin{figure}[tbp]
\centering
\begin{tabular}{@{}c@{\hspace{1mm}}c@{}}
\includegraphics[scale=0.8]{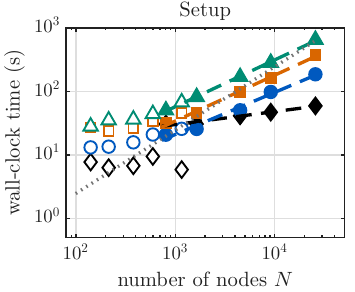} & \includegraphics[scale=0.8]{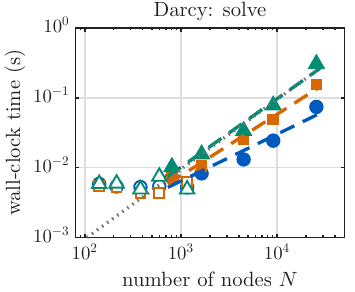}\\
\includegraphics[scale=0.8]{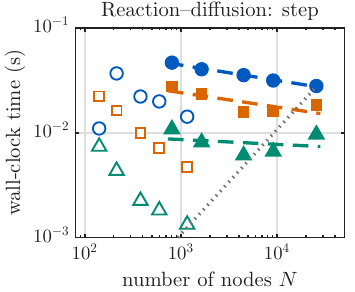} & \includegraphics[scale=0.8]{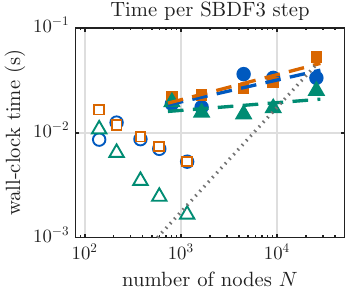}\\
\multicolumn{2}{c}{\includegraphics[scale=0.8]{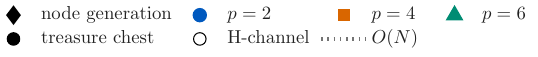}}
\end{tabular}
\caption{Wall-clock times of the forward problems on one NVIDIA RTX 5090 against the number $N$ of physical nodes (filled markers: treasure chest; open markers: H-channel). Top left: node generation, and the operator setup for one cloud and order (stencils, weights, assembly, factorization). Top right: one Darcy solve. Bottom: one SBDF3 step of reaction--diffusion and of advection--diffusion. Dashed lines: least-squares fits $CN^a$ for the treasure chest; dotted lines: $O(N)$. All times include just-in-time compilation; those of the solves and steps are medians over the parameter families.}
\label{fig:timings}
\end{figure}

\paragraph{Time stepping and error}
The time-dependent problems are integrated by SBDF3, which treats $\nu\Delta$ and the hyperviscosity implicitly and $R$, $v\cdot\nabla u$, and $f$ explicitly, with starting values at $t_0$, $t_1$, and $t_2$ taken from $u_\star$. The constants of the step rule \eqref{eq:dt_rule} are $c_1=c_2=\tfrac14$ for reaction--diffusion and $c_1=\tfrac1{16}$, $c_2=\tfrac14$ for advection--diffusion. For advection--diffusion with $p=4$, $c_2h^{4/3}<c_1h$ only for $h<1/64$, so all our clouds use $\Delta t=h/16$ and the time error is formally $O(h^3)$. However, such an error would decrease with a slope of only about $3.6$ against $N^{-1/d}$, below the observed slopes; thus, the spatial error dominates on our clouds. We report the relative error over the physical nodes at $t=T$,
\begin{equation}
 \mathrm{err}=\frac{\norm{\bar U-u_\star|_{\chii\cup\chib}}_2}{\norm{u_\star|_{\chii\cup\chib}}_2},
\end{equation}
and the least-squares slope of $\log\mathrm{err}$ against $\log N^{-1/d}$ over the three finest clouds, $h\leq0.08$.

\subsubsection{Results}
\paragraph{Convergence}
Figure~\ref{fig:forward_convergence} shows the forward errors (slopes of the individual families in Table~\ref{tab:forward_slopes} in Appendix~\ref{sec:app_slopes}). We expect the errors to decrease as $O(h^p)$. For $p=4$ the slopes against $N^{-1/d}$ lie between $4.5$ and $8.3$, and for $p=6$ between $6.0$ and $9.0$. These slopes exceed $p$ mainly because the boundary layer makes $N^{-1/d}$ decrease more slowly than $h$. From $h=0.08$ to $h=0.04$, $N^{-1/d}/h$ grows from $0.64$ to $0.73$ in the H-channel and from $0.76$ to $0.85$ in the treasure chest; an error of $O(h^p)$ thus has a slope of about $1.2p$ against $N^{-1/d}$. Against $h$, most problems converge at $O(h^p)$ or faster; the exceptions are advection--diffusion in the H-channel with $p=6$ (roughly $h^{4.9}$) and reaction--diffusion in the treasure chest with $p=4$ (roughly $h^{3.8}$). For $p=2$ the slopes are close to two, except for Darcy and reaction--diffusion in the treasure chest, where they are only $1.1$ to $1.6$.

Recall that Theorems~\ref{thm:elliptic} and~\ref{thm:parabolic} predict these rates when the stability constants stay bounded; Table~\ref{tab:forward_estimate} in Appendix~\ref{sec:app_estimates} lists both factors of the Darcy bound. The truncation error converges at least at $O(h^p)$, and the effective stability constant $\hat S_h$ of \eqref{eq:effective_stability} shows no systematic growth in $h$, except in the treasure chest with $p=2$, where it grows by a factor of $1.6$ from $h=0.08$ to $h=0.04$ and lowers the slope to $1.4$. Reaction--diffusion shows the same reduction in this case (Table~\ref{tab:forward_slopes}). In addition, advection--diffusion, which Theorem~\ref{thm:parabolic} does not cover, converges at the same rates.

\paragraph{Stencil selection}
We compare our stencil selection (step~\ref{alg:fw_stencil} of Algorithm~\ref{alg:forward}) with the obvious alternative, the $n_s$ Euclidean nearest nodes, on the same clouds. At $h=0.04$, $3\%$, $9\%$, and $26\%$ of the Euclidean stencils in the H-channel contain a node hidden behind a wall, for $p=2$, $4$, and $6$; in the treasure chest the fractions are $0.8\%$, $2.2\%$, and $4.0\%$, near the rim of the cap. With Euclidean stencils, the reduced operators of all three problems in the H-channel at $h=0.08$ with $p=2$ have an eigenvalue with positive real part, and so does the advection--diffusion generator at $h=0.16$ with $p=6$. These cases therefore fail the stability check (step~\ref{alg:fw_check} of Algorithm~\ref{alg:forward}), which every problem passes with our stencils (Section~\ref{sec:forward_stencils}). Where both pass, the errors in the treasure chest are nearly the same. On the coarse H-channel clouds the Euclidean errors are smaller, likely because the manufactured solutions continue smoothly across the walls, so that values on the other side remain accurate data; on the finest cloud with $p=6$, however, they are larger. Note that for a solution that is unrelated across a wall (e.g., in a channel with different inflows), Euclidean stencils would be inconsistent.

\paragraph{Cost}
In Figure~\ref{fig:timings}, the time of node generation grows only like $N^{0.2}$. In the treasure chest, most of it is spent in the fit of the implicit layer (steps~\ref{alg:ng_cover}--\ref{alg:ng_band} of Algorithm~\ref{alg:node_gen}), whose size is set by the boundary data more than by $h$. The setup of the operators (steps~\ref{alg:fw_stencil}--\ref{alg:fw_factor} of Algorithm~\ref{alg:forward}) grows like $N^{0.66}$ to $N^{0.74}$ and is the most expensive stage of our pipeline. A Darcy solve costs a numerical refactorization and one substitution; its time grows like $N^{0.7}$ to $N^{1.0}$. The time of an SBDF3 step (step~\ref{alg:fw_sbdf}), one forward and one backward substitution with the cuDSS factors and a few sparse products, grows weakly with $N$, likely because fixed costs per step (e.g., launching the GPU kernels) dominate on these clouds.


\subsection{Inverse problems}
\label{sec:inverse_results}

\subsubsection{Test problems}
\label{sec:inverse_tests}
We now test our inverse solver on the clouds, orders, and time steps of Section~\ref{sec:forward_tests}. For the time-dependent inverse tests, every forward solve in the iteration starts from the exact states at $t=0$, $\Delta t$, and $2\Delta t$. In each problem, $\theta$ is recovered from noise-free observations of the manufactured solution at the physical nodes, $u_{\mathrm{obs}}=u_\star|_{\chii\cup\chib}$ (at $t=T$ for the time-dependent problems). Thus, the residual $r(\theta_\star)$ of \eqref{eq:inverse_residual} is the relative forward error, and the error of the recovered parameters $\theta_h$ shows how the discretization error propagates to the parameters.

\paragraph{Darcy flow}
Here we recover the coefficients $\theta$ of the permeability \eqref{eq:permeability} in \eqref{eq:darcy_k}, for the first five families of Table~\ref{tab:families}, from the pressure $u$. The forcing $f$ and the flux $g$ are fixed at their true values; note that $\theta$ then also enters the boundary rows (Section~\ref{sec:inverse_formulation}). The manufactured solution is \eqref{eq:ustar}.

\paragraph{Reaction--diffusion}
For this problem we recover the coefficients of the reaction law in \eqref{eq:rd}, for the six families of Table~\ref{tab:families}, from $u(\cdot,T)$. Again, $f$ and the Neumann data $g$ are fixed at their true values, and the manufactured solution is $u_\star(x,t)=e^{-0.35t}u_0(x)$.

\begin{figure}[tbp]
\centering
\begin{tabular}{@{}c@{\hspace{1mm}}c@{\hspace{1mm}}c@{\hspace{1mm}}c@{}}
\multicolumn{2}{c}{\footnotesize Smooth modes} & \multicolumn{2}{c}{\footnotesize $C^1$ modes}\\
\includegraphics[width=0.24\textwidth]{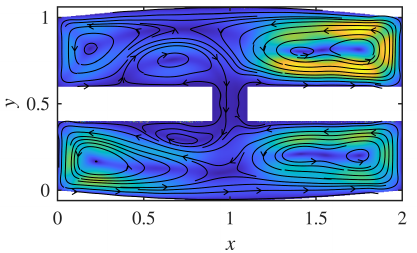} & \includegraphics[width=0.24\textwidth]{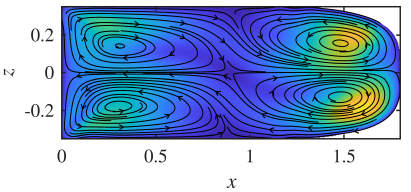} & \includegraphics[width=0.24\textwidth]{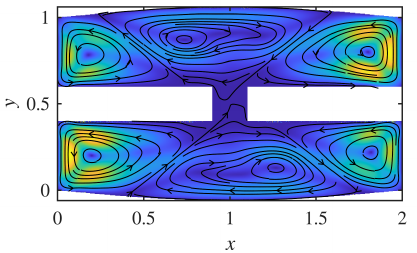} & \includegraphics[width=0.24\textwidth]{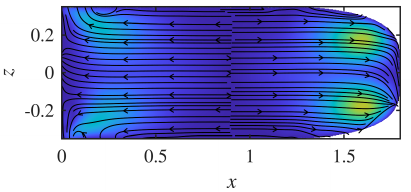}\\
\multicolumn{4}{c}{\includegraphics[width=0.28\textwidth]{Images/plots/legend_velocity}}
\end{tabular}
\caption{The true transport velocities $v(x;\theta_\star)$, colored by speed, with streamlines, for the smooth (left pair) and the $C^1$ (right pair) set of modes, each in the H-channel and the treasure chest. In the treasure chest the panels show the section $y=0$ and the components in it; the speed includes the component normal to the section.}
\label{fig:velocity_transport}
\end{figure}

\paragraph{Transport}
In the transport problem, the unknowns are the coefficients of a steady velocity
\begin{equation}
 v(x;\theta)=\sum_{i=1}^{3}\theta_iw_i(x)
 \label{eq:transport_velocity}
\end{equation}
and the observations are $u(\cdot,T)$, where $u$ solves \eqref{eq:ad} with $v=v(x;\theta)$, $\nu=0.05$, and $T=\tfrac12$. The basis fields $w_i$ are built as in \eqref{eq:velocity_field}, with $m_0$ replaced by a mode $m_i$ and, for $d=3$, $e_y$ by $e_i$. Two sets of three modes are used, one smooth and one whose fields are only $C^1$ across the coordinate planes through the center of the domain (Table~\ref{tab:velocity_modes} in Appendix~\ref{sec:app_velocity}). The true coefficients are multiples of $(1.45,-0.90,0.60)$ for $d=2$ and $(1.20,-0.80,0.50)$ for $d=3$, scaled to a maximum speed of $1.18$ as in the forward problem (Figure~\ref{fig:velocity_transport}). The manufactured solution is \eqref{eq:ustar_ad}, and $f$ and the Dirichlet data $g$ are computed with $v(x;\theta_\star)$ and fixed, as is the hyperviscosity (Section~\ref{sec:inverse_formulation}).

\begin{figure}[tbp]
\centering
\begin{tabular}{@{}c@{\hspace{1mm}}c@{}}
\textbf{H-channel} & \textbf{Treasure chest}\\
\includegraphics[scale=0.8]{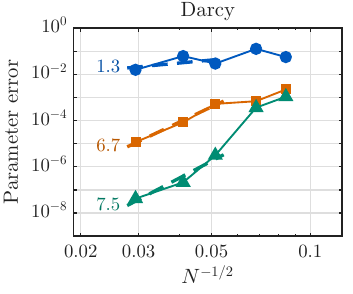} & \includegraphics[scale=0.8]{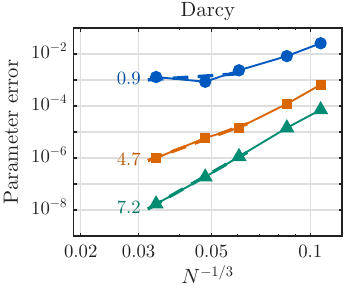}\\
\includegraphics[scale=0.8]{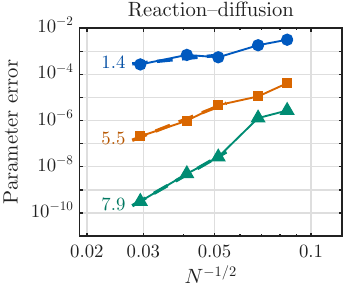} & \includegraphics[scale=0.8]{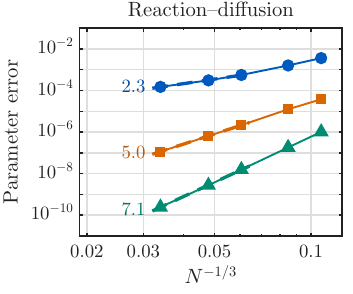}\\
\includegraphics[scale=0.8]{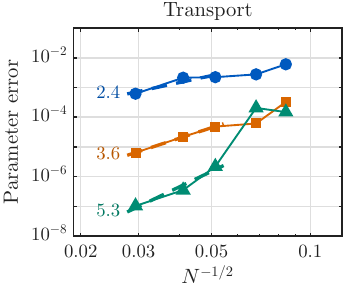} & \includegraphics[scale=0.8]{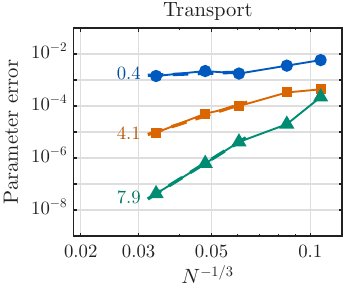}\\
\multicolumn{2}{c}{\includegraphics[scale=0.8]{Images/plots/legend_convergence}}
\end{tabular}
\caption{Relative errors of the recovered parameters against $N^{-1/d}$. Each curve is the largest error over the parameter families (for transport, over the two velocity bases). The dashed lines are least-squares fits over the three finest clouds, labeled with their slopes.}
\label{fig:inverse_convergence}
\end{figure}

\subsubsection{Results}
We report the relative parameter error $\norm{\theta_h-\theta_\star}/\norm{\theta_\star}$. The iteration starts from $\theta=0$ on the coarsest cloud and from the parameters of the next coarser cloud on the others (Section~\ref{sec:inverse_optimization}). From $\theta=0$, Levenberg--Marquardt typically needs a few tens of iterations; from the warm starts, Gauss--Newton (step~\ref{alg:inv_gn} of Algorithm~\ref{alg:inverse}) needs at most seven (Appendix~\ref{sec:app_iterations}).

\begin{figure}[tbp]
\centering
\begin{tabular}{@{}c@{\hspace{1mm}}c@{\hspace{1mm}}c@{}}
\includegraphics[width=0.325\textwidth]{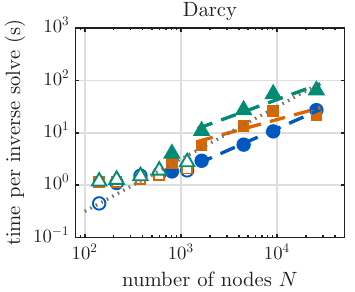} & \includegraphics[width=0.325\textwidth]{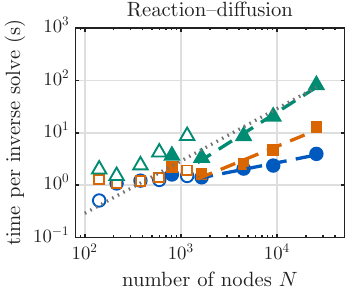} & \includegraphics[width=0.325\textwidth]{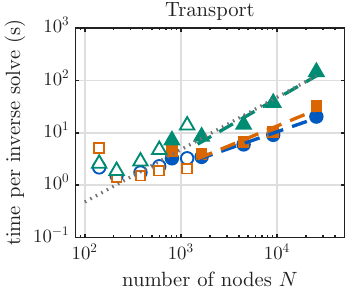}\\
\multicolumn{3}{c}{\includegraphics[scale=0.8]{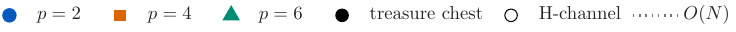}}
\end{tabular}
\caption{Wall-clock time of one inverse solve against the number $N$ of physical nodes, median over the parameter families (filled markers: treasure chest; open markers: H-channel). The iteration on the coarsest cloud starts from $\theta=0$, the others from the parameters of the next coarser cloud. Dashed lines: least-squares fits $CN^a$ for the treasure chest over the warm-started inverse solves; dotted lines: $O(N)$. The setup of the operators is shared with the forward problems (Figure~\ref{fig:timings}).}
\label{fig:inverse_timings}
\end{figure}

\paragraph{Convergence}
Figure~\ref{fig:inverse_convergence} shows the parameter errors (slopes of the individual families in Table~\ref{tab:inverse_slopes} in Appendix~\ref{sec:app_slopes}). Since $r(\theta_\star)$ is the forward error, we expect the parameter errors to follow the forward errors of Figure~\ref{fig:forward_convergence}. For $p=6$ this holds roughly, with slopes between $6.5$ and $9.3$ in the treasure chest and between $5.2$ and $11.0$ in the H-channel. For $p\leq4$, however, the slopes scatter widely, and some families fall below $p$, namely transport with $p=4$ in the H-channel ($2.2$ to $3.3$), Darcy with $p=4$ in the treasure chest (down to $2.8$), and Darcy and reaction--diffusion with $p=2$ in the treasure chest (down to $0.9$ and $1.0$). Transport with $p=2$ in the treasure chest is the one exception to the trend; its parameter error hardly decreases over the three finest clouds ($0.4$ to $0.9$), although the forward error does.

Theorem~\ref{thm:inverse_aposteriori} explains both the agreement with the forward errors and the scatter for $p\leq4$. To first order, the parameter error is $-J^+r(\theta_\star)$, the component of the forward error in the range of $J$, amplified by at most $1/\sigma_{\min}(J)$. This amplification hardly changes with the cloud; for each problem, order, and family, $\sigma_{\min}(J)$ varies by at most $19\%$ over the three finest clouds. However, the component in the range of $J$ changes from cloud to cloud and need not decrease with the forward error as a whole. Cloud-to-cloud variation produces scatter when it is comparable to the decrease in the total forward error. For $p=2$, that decrease is only about a factor of three over the three finest clouds. The stalled parameter error for transport with $p=2$ in the treasure chest is likely due to this variation dominating the trend. Finally, the efficiency index
\begin{equation}
 \eta=\frac{\norm{\theta_h-\theta_\star}\,\sigma_{\min}\bigl(J(\theta_h)\bigr)}{\norm{r(\theta_\star)}},
\end{equation}
which the theorem predicts to be at most one to first order, is below one in every inverse solve and typically well below it (Table~\ref{tab:inverse_efficiency} in Appendix~\ref{sec:app_estimates}).

\paragraph{Cost}
Each iteration evaluates the residual and the Jacobian in one forward-mode pass (step~\ref{alg:inv_residual} of Algorithm~\ref{alg:inverse}), in which the $n_\theta\leq6$ tangent directions are solved together with the cuDSS factors. For the time-dependent problems, an iteration thus costs a small multiple of one time integration, and since the warm starts need few iterations, an inverse solve costs a few forward solves (Figure~\ref{fig:inverse_timings}). Darcy flow is the exception, since a Darcy forward solve is cheap and fixed costs (e.g., the compilation of the differentiated forward map) make a Darcy inverse solve much longer than a few forward solves (Table~\ref{tab:end_to_end}). The longest inverse solves are those of the time-dependent problems with $p=6$, which take the most time steps. In the treasure chest, the time of an inverse solve grows roughly like $N^{0.4}$ to $N^{1.2}$.

\subsection{End-to-end cost}
\label{sec:end_to_end}
We now consider the cost of the whole computation, from the seed nodes to the recovered parameters. Table~\ref{tab:end_to_end} lists the time of every stage of Algorithms~\ref{alg:node_gen}--\ref{alg:inverse} on the finest clouds. The setup of the operators dominates; stencil selection, weights, assembly, the stability check, and the factorization grow with the stencil size (and hence with $p$) and together cost more than all other stages. However, the setup is done once per cloud and order; further parameter families, observations, or right-hand sides on the same cloud cost only forward and inverse solves.

\begin{table}[tbp]
\centering
\caption{Wall-clock seconds of the stages of Algorithms~\ref{alg:node_gen}--\ref{alg:inverse} on the finest clouds, $h=0.04$, on one NVIDIA RTX 5090, including just-in-time compilation. Setup: stencils, weights, assembly, stability check, and factorization for one cloud and order. RD: reaction--diffusion, AD: advection--diffusion. Forward and inverse solves: medians over the parameter families. The forward solves of RD and AD integrate to $t=T$. The inverse solves are warm-started from the next coarser cloud and share the setup.}
\label{tab:end_to_end}
\footnotesize
\setlength{\tabcolsep}{4pt}
\begin{tabular}{@{}llrrrrrrrrr@{}}
\toprule
& & & & & \multicolumn{3}{c}{Forward solve} & \multicolumn{3}{c}{Inverse solve}\\
\cmidrule(lr){6-8}\cmidrule(l){9-11}
Geometry & $p$ & $N$ & Nodes & Setup & Darcy & RD & AD & Darcy & RD & Transport\\
\midrule
H-channel & 2 & 1160 & 5.89 & 26 & 0.01 & 0.71 & 1.06 & 1.92 & 1.50 & 3.26\\
 & 4 & 1160 &  & 46 & 0.01 & 0.70 & 1.07 & 2.07 & 1.91 & 2.07\\
 & 6 & 1160 &  & 68 & $<0.01$ & 1.66 & 2.07 & 2.76 & 8.76 & 14\\
\midrule
Treasure chest & 2 & 25636 & 59 & 187 & 0.07 & 1.40 & 6.70 & 27 & 3.96 & 20\\
 & 4 & 25636 &  & 377 & 0.16 & 2.70 & 11 & 21 & 13 & 32\\
 & 6 & 25636 &  & 658 & 0.31 & 12 & 32 & 65 & 82 & 146\\
\bottomrule
\end{tabular}
\end{table}


\subsection{A domain with an internal obstacle}
\label{sec:obstacle_results}
The boundaries of the H-channel and the treasure chest are connected. We now present a final example with a two-component boundary, in which an obstacle is placed inside the treasure chest and a reaction law is recovered from a solution carried around the obstacle by a flow. The method of~\cite{shankar2018} already handles smooth embedded objects. However, the obstacle here has curved and planar pieces, edges, vertices, and reentrant creases of the fluid domain. Our framework treats it with the same algorithms as the outer wall (Section~\ref{sec:node_gen_description}); nothing in our node generator or solvers is specific to this obstacle, and further obstacles add pieces in the same way.

\paragraph{Domain}
The obstacle is a smaller, rotated copy of the treasure chest,
\begin{equation}
 \Omega_{\mathrm{obs}}=\bigl\{c_{\mathrm{obs}}+0.3\,Q(x-x_{\mathrm c}):\ x\in\Omega_3\bigr\},\qquad
 \Omega_4=\Omega_3\setminus\overline\Omega_{\mathrm{obs}},
\end{equation}
where $x_{\mathrm c}=(0.9,0,0)$ is the center of the bounding box of $\Omega_3$, $c_{\mathrm{obs}}=(0.6,0,0)$, and $Q$ is the rotation by $25^\circ$ about the $x$-axis followed by the rotation by $35^\circ$ about the $z$-axis. None of the faces of the obstacle is parallel to a face of the outer chest. The boundary of $\Omega_4$ consists of the ten pieces of the outer chest and the ten corresponding pieces of the obstacle. Every piece of the obstacle, including the planar ones, is given only by seed nodes with normals and the trimming loops of its parameter region, with the normals pointing into $\Omega_{\mathrm{obs}}$ (i.e., out of the fluid domain). For each $h\in\{0.16,0.12,0.08,0.06,0.04\}$ we generate one cloud with the parameters of the treasure chest, giving $N=872$ to $26586$ physical nodes. All interior nodes lie inside the exact domain $\Omega_4$, and all ghost nodes lie outside it, those of the obstacle inside $\Omega_{\mathrm{obs}}$ (Figure~\ref{fig:obstacle_nodes}).

\begin{figure}[tbp]
\centering
\includegraphics[width=0.30\textwidth]{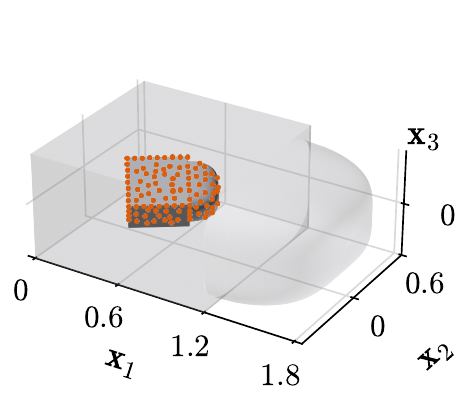}\hspace{3mm}
\includegraphics[width=0.36\textwidth]{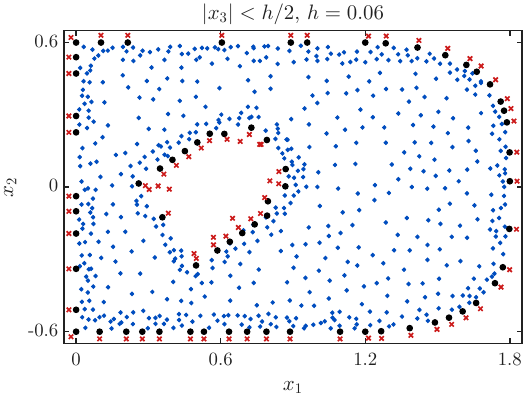}
\caption{The treasure chest with an internal obstacle. Left: the obstacle inside the translucent outer boundary, with its boundary nodes on the visible part of its surface, $h=0.04$. Right: interior nodes (blue), boundary nodes (black), and ghost nodes (red crosses) with $|z|<h/2$, $h=0.06$.}
\label{fig:obstacle_nodes}
\end{figure}

\begin{figure}[tbp]
\centering
\begin{tabular}{@{}c@{\hspace{3mm}}c@{}}
\includegraphics[width=0.42\textwidth]{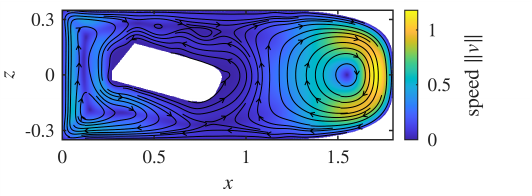} & \includegraphics[width=0.42\textwidth]{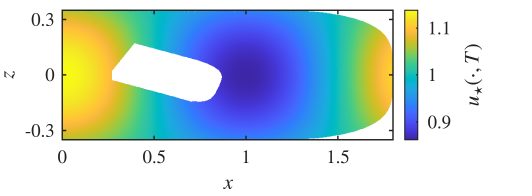}\\[1mm]
\includegraphics[width=0.42\textwidth]{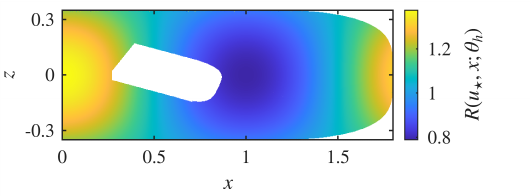} & \includegraphics[width=0.42\textwidth]{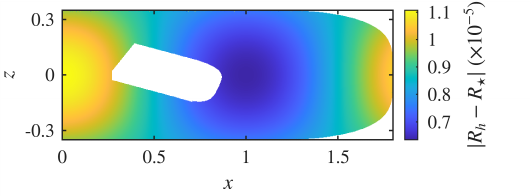}
\end{tabular}
\caption{The problem with an internal obstacle on the section $y=0$. Top: the velocity $v=1.18\,w$, colored by speed, with streamlines, and the solution $u_\star(\cdot,T)$. Bottom: the reaction rate $R(u_\star,x;\theta_h)$ of the sine-product law recovered on the cloud with $h=0.06$ and $p=4$, and its pointwise error.}
\label{fig:obstacle_solution}
\end{figure}

\paragraph{Problem}
The concentration $u$ solves
\begin{equation}
 u_t+v\cdot\nabla u=\nu\Delta u+R(u,x;\theta)+f\ \ \text{in }\Omega_4\times(0,T],\qquad u=g\ \ \text{on }\partial\Omega_4,
 \label{eq:adr}
\end{equation}
with $T=\tfrac12$, $\nu=0.05$, the six reaction laws and true parameters of the reaction--diffusion problem \eqref{eq:rd}, and the manufactured solution \eqref{eq:ustar_ad}. The velocity $v=1.18\,w$ is steady, with $w$ given by \eqref{eq:velocity_field} for the implicit layer $\phi_0$ of the cloud of $\Omega_4$ with $h=0.16$. Note that the zero set of $\phi_0$ approximates both boundary components and $\phi_0$ is negative inside the obstacle. Thus, the divergence-free field $w$ vanishes near the outer wall and on the obstacle surface, and the flow passes around the obstacle without slip (Figure~\ref{fig:obstacle_solution}). The hyperviscosity and the time steps are those of the advection--diffusion and transport problems, and, as in Section~\ref{sec:inverse_tests}, we recover $\theta$ from $u(\cdot,T)$ at the physical nodes with $f$ and $g$ fixed at their true values. The forward error is the relative error $\norm{r(\theta_\star)}$ of the solution with the true parameters.

\begin{figure}[tbp]
\centering
\begin{tabular}{@{}c@{\hspace{1mm}}c@{}}
\includegraphics[scale=0.8]{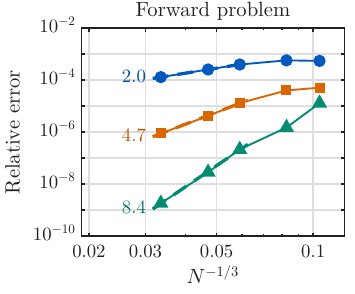} & \includegraphics[scale=0.8]{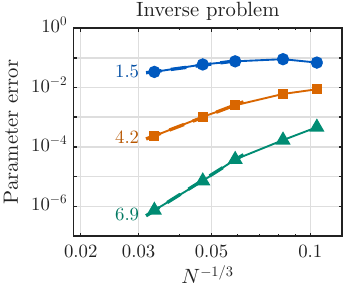}\\
\multicolumn{2}{c}{\includegraphics[scale=0.8]{Images/plots/legend_convergence}}
\end{tabular}
\caption{Advection--diffusion--reaction in the treasure chest with an internal obstacle. Left: relative forward error with the true parameters; right: relative error of the recovered parameters, both against $N^{-1/3}$. Each curve is the largest error over the six reaction laws. The dashed lines are least-squares fits over the three finest clouds, labeled with their slopes.}
\label{fig:obstacle_convergence}
\end{figure}

\paragraph{Results}
Figure~\ref{fig:obstacle_convergence} shows the errors, and Figure~\ref{fig:obstacle_solution} the recovered reaction rate of the sine-product law, whose error is small and nearly uniform across the section. Every cloud and order passes the stability check, and the warm-started inverse solves need at most six Gauss--Newton iterations. The forward error decreases with slopes $2.0$, $4.7$, and $8.4$ for $p=2$, $4$, and $6$, roughly as for advection--diffusion in the treasure chest in Figure~\ref{fig:forward_convergence}. The parameter error decreases with slopes $1.5$, $4.2$, and $6.9$; these are somewhat below the forward slopes, as for the problems of Section~\ref{sec:inverse_results} (slopes of the individual laws in Tables~\ref{tab:forward_slopes} and~\ref{tab:inverse_slopes} in Appendix~\ref{sec:app_slopes}). Again, the efficiency index $\eta$ is below one in every inverse solve (Table~\ref{tab:inverse_efficiency} in Appendix~\ref{sec:app_estimates}). These results show similar forward rates and a similar propagation of forward error into the recovered parameters with and without the obstacle.

\section{Conclusions}
\label{sec:conclusions}
We have presented KernelPack-JAX, a framework that takes samples of smooth boundary pieces through node generation, high-order RBF-FD solves, and parameter recovery on the GPU, with derivatives through the operators, the solves, and the time loops. Our framework rests on the stacked boundary representation, in which the parametric and implicit layers share one hierarchy of supports. The implicit layer, extended into the volume through the components of the uncovered grid cells, classifies every point of space; it serves as the inside test of the interior sampling and guides stencil selection through a visibility heuristic. Internal obstacles are represented like the outer wall. Node selection by greedy independent sets makes the thinning and sampling stages parallel computations with a deterministic result. Finally, the differentiable sparse direct solves make derivatives through long time loops cost a small multiple of the loops themselves, with memory independent of the number of time steps.

In both test domains and in the domain with an internal obstacle, the forward errors converge at high order, and the recovered parameters largely follow them. The exception is transport with $p=2$ in the treasure chest, where the part of the forward error that enters the parameters does not decrease over the finest clouds. In our implementation, the setup of the operators dominates the total cost.

Our error estimates in Section~\ref{sec:analysis} separate consistency, of order $p$ by polynomial reproduction, from stability. However, the stability of the ghost-node closure is assumed; we test a necessary spectral condition and measure the stability constants. The first-order bound of the parameter error by the forward error and the smallest singular value of the Jacobian held in all our experiments (Appendix~\ref{sec:app_estimates}).

The main open question is a proof of uniform stability for RBF-FD collocation with a ghost-node closure on scattered nodes, including the advection-dominated case. We used forward-mode derivatives throughout, since our problems have at most six parameters. For problems with many parameters, we plan to add reverse mode through the same transpose solves, with checkpointing of the time loop. In addition, since a CAD model consists of trimmed parametric faces, it provides the input of our node generator directly, and we plan to apply our framework to CAD models with many faces, fillets, and features smaller than the node spacing.

\section*{CRediT authorship contribution statement}
{\noindent
\textbf{Alberto Cattaneo:} conceptualization, methodology, writing, software.\\
\textbf{M. Keith Ballard:} funding acquisition, resources, supervision.\\
\textbf{Robert M. Kirby:} conceptualization, funding acquisition, project administration.\\
\textbf{Grady B. Wright:} conceptualization, methodology.\\
\textbf{Varun Shankar:} conceptualization, funding acquisition, methodology, project administration, resources, software, supervision, validation, visualization, writing.
}

\section*{Declaration of competing interest}
The authors declare that they have no known competing financial interests or personal relationships that could have appeared to influence the work reported in this paper.

\section*{Data availability}
The source code of KernelPack-JAX, including the scripts that generate every node set, result, table, and figure of this paper (\texttt{scripts/paper/reproduce.sh}), is publicly available at \url{https://github.com/VarShankar/kernelpack-jax}.

\section*{Acknowledgments}
MKB was supported by the United States Air Force Research Laboratory (AFRL). AC, RMK, and VS were supported by the U.S. Air Force through ARCTOS Technology Solutions, LLC, under Task Order FA2394-25-F-B004, Prime Contract FA2394-25-D-B005, through Subcontract 25B005.01.00.0001.00.05-C1. RMK and VS were also supported by Air Force Office of Scientific Research (AFOSR) LRIR grant FA9550-25-1-0042. VS was also supported by the U.S. National Science Foundation (NSF), Division of Mathematical Sciences (DMS), under award 2505986. GBW was supported by NSF grants DMS-2309712 and DMS-2505987.

\appendix
\allowdisplaybreaks
\section{Node generation}
\label{sec:app_nodegen}

\subsection{Parameters}
\label{sec:app_parameters}
Table~\ref{tab:node_gen_parameters} lists the parameters of Section~\ref{sec:node_gen}.

\begin{table}[htbp]
\centering
\caption{Parameters of the node generator, fixed across geometries, resolutions, and random seeds.}
\label{tab:node_gen_parameters}
\footnotesize
\begin{tabular}{llc}
\toprule
Stage & Parameter & Value\\
\midrule
Charts & compression bounds $\mu_c$ & $\tfrac14,\ \tfrac1{20}$\\
Parametric layer & overlap factor $\kappa$ & $1.5$\\
 & minimum seed nodes per support $m_{\min}$ & $12$\\
 & diagonal shift of $\mathbf{A}_j$ & $10^{-12}\max(1,\max|\mathbf{A}_j|)$\\
 & tolerances $\varepsilon_x,\ \varepsilon_n,\ \varepsilon_J$ & $10^{-3},\ 0.05,\ 0.05$\\
 & maximum refinements, maximum cells & $3,\ 128$\\
Boundary nodes & feature site spacing, incidence distance $\delta_{\mathcal F}$ & $h/4,\ 0.02h$\\
 & smooth-seam angle $\theta_{\mathcal F}$ & $2^\circ$\\
 & ghost offset & $h/2$\\
Implicit layer & feature enrichment $g_{\mathcal E},\ \beta_{\mathcal E}$ & $2,\ 2$\\
 & ball enlargement $\kappa_B$ & $1.25$\\
 & points per ball $m_{\mathrm c},\ M_{\mathrm c}$ & $8,\ 512$\\
 & maximum cover refinements & $3$\\
 & quadrature points per axis $N_Q$ & $48$\\
 & relative diagonal shift in \eqref{eq:cfpu_system} & $10^{-12}$\\
Interior nodes & boundary layer $\rho,\ \beta$ & $0.6,\ 1.5$\\
 & surface threshold $\epsilon$ & $0.2$\\
 & candidate density $c_{\mathrm d}$ & $6$\\
 & maximum passes & $4$\\
\bottomrule
\end{tabular}
\end{table}

\subsection{Parameters from points}
\label{sec:app_charts}
A curved piece is often given only by seed nodes $y_{\ell k}\in\Gamma_\ell$, unit normals, and the trimming loops in physical space. We then compute the parameters, for both the seed nodes and the loops, with one of the following charts, and stop with an error if none applies. We fix a piece, drop the index $\ell$, and let $\bar y$ be the mean of the seed nodes.
\begin{itemize}
 \item For a curve, $u=(y-\bar y)\cdot e_1$, where $e_1$ is the principal axis~\cite{pearson1901lines} of the seed nodes. We require $u$ to increase along the curve.
 \item For a surface, we first try the best-fit plane of the seed nodes, with orthonormal axes $e_1,e_2$ and normal $e_3$, and set $u=\bigl((y-\bar y)\cdot e_1,\,(y-\bar y)\cdot e_2\bigr)$. We orient $e_3$ so that the mean of $n_k\cdot e_3$ is positive, and accept this chart if $n_k\cdot e_3\geq\tfrac14$ for every normal $n_k$ of a seed node.
 \item Otherwise we use a stereographic chart. Let $o$ be the centroid of the trimming loops, $e_0$ the unit normal of their best-fit plane, oriented towards the seed nodes, and $e_1,e_2$ its in-plane axes. With the scaled coordinates $\hat y_a=(y-o)\cdot e_a/\max_k|(y_k-o)\cdot e_a|$, $a=0,1,2$, we set
 \begin{equation}
  u=\frac{(\hat y_1,\hat y_2)}{\norm{\hat y}+\hat y_0},
 \end{equation}
 the stereographic projection of the direction $\hat y/\norm{\hat y}$ from the antipode of $e_0$. This chart requires that the piece lies in one hemisphere about $o$, $\hat y_0\geq0$ for every seed node, and that every normal, expressed in the scaled coordinates, makes an angle with the direction $\hat y$ whose cosine is at least $\tfrac14$.
\end{itemize}
An accepted chart must also not compress the seed nodes. With $\tilde y=y$ for the first two charts and $\tilde y=\hat y$ for the stereographic chart, we require
\begin{equation}
 \min_{k'\neq k}\frac{\norm{u_k-u_{k'}}}{\norm{\tilde y_k-\tilde y_{k'}}}\geq\mu_c\quad\text{for every }k,
\end{equation}
with $\mu_c=\tfrac14$ for the first two charts and $\mu_c=\tfrac1{20}$ for the stereographic chart. The cap of the treasure chest in Section~\ref{sec:nodegen_results} receives the stereographic chart. Its parameter region in Figure~\ref{fig:node_stack}(a) is therefore a disk.

\subsection{Adaptive refinement of the parametric layer}
\label{sec:app_refine}
Let $\bar y$ be the mean of all seed nodes and $\delta_{\mathcal D}=2\max_k\norm{y_k-\bar y}$, let $\varepsilon_x$, $\varepsilon_n$, $\varepsilon_J$ be tolerances, and let $\sigma_{\min}$ and $\sigma_{\max}$ denote the smallest and largest singular values of a matrix. We accept the map if every validation node satisfies
\begin{equation}
 \norm{X(u_v)-y_v}\leq\varepsilon_x\delta_{\mathcal D},\qquad \norm{n(u_v)-n_v}\leq\varepsilon_n,\qquad
 \frac{\sigma_{\min}(DX(u_v))}{\sigma_{\max}(DX(u_v))}>\varepsilon_J .
 \label{eq:pu_acceptance}
\end{equation}
We measure the worst violation of \eqref{eq:pu_acceptance} by
\begin{equation}
 E_{\mathrm{fit}}=\max_{v}\max\Bigl(\frac{\norm{X(u_v)-y_v}}{\varepsilon_x\delta_{\mathcal D}},\ \frac{\norm{n(u_v)-n_v}}{\varepsilon_n},\ \frac{\varepsilon_J\,\sigma_{\max}(DX(u_v))}{\sigma_{\min}(DX(u_v))}\Bigr),
\end{equation}
and build the cover as follows.
\begin{enumerate}
 \item Start with the $2^{d-1}$ cells $c_j\in\{-\tfrac12,\tfrac12\}^{d-1}$, $\mathbf{s}_j=(\tfrac12,\dots,\tfrac12)$.
 \item Solve \eqref{eq:pu_local_system} on every support and evaluate \eqref{eq:pu_acceptance} at every validation node. If all validation nodes satisfy it, stop.
 \item Split every cell whose support contains $z(u_v)$ for a failing validation node $u_v$ into its $2^{d-1}$ children, with centers $c_j+\tfrac12\operatorname{diag}(\mathbf{s}_j)\mathbf{t}$, $\mathbf{t}\in\{-1,1\}^{d-1}$, and half-widths $\tfrac12\mathbf{s}_j$. Return to step 2.
\end{enumerate}
We stop with an error if the map is not accepted after three refinements, if a refinement does not reduce $E_{\mathrm{fit}}$, or if the cover exceeds $128$ cells.

\subsection{Oriented samples and the cover of balls}
\label{sec:app_cover}
The oriented samples of Section~\ref{sec:node_gen_levelset_math} are
\begin{equation}
 \{(\xi_k,\hat n_k)\}_{k=1}^{N_\xi}=\{(x_b,n_b):\ x_b\in\chib\}\ \cup\ \mathcal{E}.
\end{equation}
The set $\mathcal{E}$ adds points near sharp features, where the normal field changes abruptly. For parameters $g_{\mathcal E}>1$ and $\beta_{\mathcal E}>0$ we obtain it as follows.
\begin{enumerate}
 \item Sample every piece as in Section~\ref{sec:node_gen_sampling}, with $h$ replaced by $h/g_{\mathcal E}$.
 \item Keep the samples $y$ that lie off the trimming loops and within the distance $\beta_{\mathcal E}h$ of some feature site, $\min_i\norm{y-x^{\mathcal F}_i}<\beta_{\mathcal E}h$.
 \item Thin the kept samples as in \eqref{eq:boundary_thinning} with spacing $h/g_{\mathcal E}$, placing the boundary nodes first so that all of them are kept. Each remaining sample $y=X_\ell(u)$ carries the normal $n_\ell(u)$.
\end{enumerate}

Let $Q$ be the tensor Gauss--Legendre rule with $N_Q$ points per axis on the bounding box of each parameter region $U_\ell$, restricted to $U_\ell$. Let $\kappa_B>1$ be an enlargement factor, and let $\bar s$ be the median over $k$ of the distance from $\xi_k$ to its nearest neighbor among the other samples. Each support $S_j$ of the parametric layer, on a piece with map $X$, defines the mapped set, its mean, and a radius,
\begin{gather}
 \mathcal{Y}_j=\bigl\{X(u):\ u\in Q,\ z(u)\in S_j\bigr\},\qquad
 \bar x_j=\frac{1}{|\mathcal{Y}_j|}\sum_{y\in \mathcal{Y}_j}y,\\
 R_j=\kappa_B\max\Bigl(\max_{y\in \mathcal{Y}_j}\norm{y-\bar x_j},\,2\bar s\Bigr),
\end{gather}
and thereby the ball
\begin{equation}
 B_j=\bigl\{x\in\R^d:\ \norm{x-\bar x_j}<R_j\bigr\}.
\end{equation}
The samples in $B_j$ have the indices $I_j=\{k:\ \xi_k\in B_j\}$. For prescribed bounds $m_{\mathrm c}<M_{\mathrm c}$ we adjust the cover by two rules.
\begin{itemize}
 \item If $|I_j|<m_{\mathrm c}$, let $\delta_j$ be the distance from $\bar x_j$ to its $m_{\mathrm c}$-th nearest sample. We set $R_j\leftarrow\max(R_j,1.01\,\delta_j)$, and stop with an error if this exceeds twice the original radius.
 \item If $|I_j|>M_{\mathrm c}$, we split the cell $C_j$ into its $2^{d-1}$ children, refit the parametric layer on that piece, rebuild $\chib$, and rebuild the balls.
\end{itemize}
We allow at most three refinements by the second rule. Finally, we check that every sample lies in at least one ball. To test that the boundary is closed and consistently oriented, we also check that the quadrature of $\int_{\partial\Omega}n\,dS=0$ with the rule of $2N_Q$ points per axis on all pieces has norm at most $0.1$ times the quadrature of the area.

\subsection{Implementation}
\label{sec:app_listings_nodegen}
We pad every array that enters a compiled kernel to a power of two and mask the padding, so that all kernels have fixed shapes. To find the neighbors of the points, we bin them into a uniform grid with cell width $\max_ir_i$. Sorting the points by the index of their cell gives each cell a contiguous range of points, and we copy these ranges into a table with one row per cell and a fixed number of columns (Listing~\ref{lst:cell_table}). For each point we then gather the rows of the $3^d$ cells around it, test the distances, and store the indices of the neighbors in an array of fixed width, padded with $-1$.
\begin{lstlisting}[language=Python,basicstyle=\ttfamily\lstsize,caption={Binning points into a table of grid cells.},label={lst:cell_table}]
@partial(jax.jit, static_argnames=("shape", "capacity"))
def cell_table(x, lo, width, shape, capacity):   # shape: cells per axis
    ijk = ((x - lo) // width).astype(int).T
    cell = jnp.ravel_multi_index(ijk, shape, mode="clip")    # cell of each point
    order = jnp.argsort(cell, stable=True)          # points sorted by cell
    count = jnp.bincount(cell, length=math.prod(shape))
    start = jnp.cumsum(count) - count               # first slot of each cell
    slot = jnp.arange(x.shape[0]) - start[cell[order]]   # position in cell
    table = jnp.full((math.prod(shape), capacity), -1)
    return table.at[cell[order], slot].set(order, mode="drop")
\end{lstlisting}

Listing~\ref{lst:greedy} shows the parallel greedy independent set of Section~\ref{sec:node_gen_gpu}. Each round gathers the states of the neighbors and takes two reductions over the columns of the neighbor array, and the rounds continue until no point is undecided.
\begin{lstlisting}[language=Python,basicstyle=\ttfamily\lstsize,caption={The parallel greedy independent set.},label={lst:greedy}]
@jax.jit
def greedy_set(nbr, rank, state):        # state: 0 = U, 1 = I, 2 = O
    pad = nbr < 0
    j = jnp.where(pad, 0, nbr)
    higher = ~pad & (rank[j] < rank[:, None])
    def round_(s):
        sj = s[j]
        has_in = jnp.any((sj == 1) & ~pad, axis=1)
        blocked = jnp.any((sj == 0) & higher, axis=1)
        s = jnp.where((s == 0) & has_in, 2, s)
        return jnp.where((s == 0) & ~blocked, 1, s)
    return lax.while_loop(lambda s: jnp.any(s == 0), round_, state)
\end{lstlisting}

One pass of the interior sampling of Section~\ref{sec:node_gen_interior_math} is a greedy independent set over the fixed points followed by the candidates (Listing~\ref{lst:dart}). The fixed points come first in the priority order and start in $\mathsf{I}$, and a random permutation orders the candidates. The accepted candidates are then those of sequential dart throwing in that random order.
\begin{lstlisting}[language=Python,basicstyle=\ttfamily\lstsize,caption={One pass of dart throwing as a greedy independent set.},label={lst:dart}]
def dart_pass(fixed, r_fixed, cand, r_cand, rng):
    n0, n1 = len(fixed), len(cand)
    x = jnp.concatenate([fixed, cand])
    r = jnp.concatenate([r_fixed, r_cand])
    nbr = neighbor_lists(x, r)           # |x_i - x_j| < max(r_i, r_j)
    rank = jnp.concatenate([jnp.arange(n0), n0 + rng.permutation(n1)])
    state = jnp.concatenate([jnp.ones(n0, int), jnp.zeros(n1, int)])
    state = greedy_set(nbr, rank, state)
    return cand[state[n0:] == 1]
\end{lstlisting}

Listing~\ref{lst:grouped} evaluates $f_{\mathrm{loc}}$ group by group, as described in Section~\ref{sec:node_gen_grouping}. The groups are formed on the host, and each width $2^e$ compiles one kernel.
\begin{lstlisting}[language=Python,basicstyle=\ttfamily\lstsize,caption={Evaluation of $f_{\mathrm{loc}}$ grouped by the number of overlapping balls.},label={lst:grouped}]
@partial(jax.jit, static_argnames="width")
def f_loc_group(x, width, centers, radii, fits):   # x: points of one group
    def one(xi):
        r = jnp.linalg.norm(xi - centers, axis=1) / radii
        w = jnp.maximum(1 - r, 0)**4 * (4*r + 1)          # omega_j(xi)
        w_top, j = jax.lax.top_k(w, width)                # contains all w_j > 0
        f = local_values(xi, j, fits)                     # f_j(xi) for the selected j
        return jnp.sum(w_top * f) / jnp.sum(w_top)
    return jax.lax.map(one, x, batch_size=256)

m = overlap_counts(x, centers, radii)                     # m(x)
e = np.ceil(np.log2(np.maximum(m, 1))).astype(int)
for k in np.unique(e):                                    # one kernel per width
    phi[e == k] = f_loc_group(x[e == k], 2**k, centers, radii, fits)
\end{lstlisting}

Listing~\ref{lst:cfpu_fit} shows the curl-free system of Section~\ref{sec:node_gen_levelset_math} for a group of balls with at most $n$ samples each. A mask marks the samples that are present, and the rows of the missing ones are replaced by rows of the identity.
\begin{lstlisting}[language=Python,basicstyle=\ttfamily\lstsize,caption={The curl-free local fits for a group of balls.},label={lst:cfpu_fit}]
@jax.jit
def curl_free_fits(q, nrm, mask):        # q, nrm: (B, n, d) scaled points, normals
    B, n, d = q.shape
    D = q[:, :, None] - q[:, None, :]                          # (B, n, n, d)
    r = jnp.linalg.norm(D, axis=-1)
    rs = jnp.where(r > 0, r, 1.0)
    Phi = -3 * (r[..., None, None] * jnp.eye(d)
                + D[..., :, None] * D[..., None, :] / rs[..., None, None])
    Phi = Phi * (mask[:, :, None] * mask[:, None, :])[..., None, None]
    G = Phi.transpose(0, 1, 3, 2, 4).reshape(B, n * d, n * d)
    md = jnp.repeat(mask, d, axis=1)                           # (B, n d)
    G = G + jnp.eye(n * d) * (1 - md)[:, None, :]              # padded rows
    E = jnp.tile(jnp.eye(d), (n, 1)) * md[..., None]           # 1 (x) I_d
    K = jnp.block([[G, E], [E.swapaxes(1, 2), jnp.zeros((B, d, d))]])
    rhs = jnp.concatenate([(nrm * mask[..., None]).reshape(B, n * d),
                           jnp.zeros((B, d))], axis=1)
    sol = jnp.linalg.solve(K, rhs[..., None])[..., 0]
    return sol[:, :n * d].reshape(B, n, d), sol[:, n * d:]     # lambda_jk, b_j
\end{lstlisting}

The min-label propagation of Section~\ref{sec:node_gen_farfield} acts on the whole background grid at once. We obtain the labels of the $2d$ face neighbors of all cells by shifting the array of labels by one cell along each axis (Listing~\ref{lst:labels}). Covered cells and the cells beyond the edges of the grid carry the largest integer and never lower a label. The sign of each component then follows from one segment reduction over the labels.
\begin{lstlisting}[language=Python,basicstyle=\ttfamily\lstsize,caption={Labels of the connected components of the free cells.},label={lst:labels}]
@jax.jit
def label_components(free):              # free: boolean array of grid cells
    big = jnp.iinfo(jnp.int32).max
    lab = jnp.where(free, jnp.arange(free.size).reshape(free.shape), big)
    def shift(l, axis, s):               # neighbor label, big beyond the edge
        n = l.shape[axis]
        edge = jnp.arange(n) == (0 if s == 1 else n - 1)
        edge = jnp.expand_dims(edge, [a for a in range(l.ndim) if a != axis])
        return jnp.where(edge, big, jnp.roll(l, s, axis=axis))
    def body(carry):
        l, _ = carry
        new = l
        for axis in range(l.ndim):
            for s in (1, -1):
                new = jnp.minimum(new, shift(l, axis, s))
        new = jnp.where(free, new, big)
        return new, jnp.any(new != l)
    return lax.while_loop(lambda c: c[1], body, (lab, True))[0]
\end{lstlisting}

\section{Stencils and solvers}
\label{sec:app_solvers}

\subsection{Visibility by ray tracing}
\label{sec:app_raytrace}
The stencil selection of Section~\ref{sec:forward_stencils} needs to decide whether $\phi\geq-\delta_\phi$ along a segment $[x_c,y]$. If $\phi$ is Lipschitz continuous with constant $L_\phi$, then
\begin{equation}
 \phi(x_c)+\phi(y)>L_\phi\norm{y-x_c}
 \label{eq:visibility_certificate}
\end{equation}
implies that $[x_c,y]$ is visible, since $\phi$ cannot fall below zero between the two cones of slope $L_\phi$ at the ends. The certificate settles most pairs from the values of $\phi$ at the nodes alone. For each remaining segment we start at $z=x_c$ and move $z$ toward $y$ by $\max(\phi(z)/L_\phi,\varepsilon_{\mathrm{ray}})$ per step. We stop with the segment hidden when $\phi(z)<-\delta_\phi$, and with the segment visible when $z$ reaches $y$ or when \eqref{eq:visibility_certificate} holds with $x_c$ replaced by $z$. We take $L_\phi=1.5$, above the largest measured $\norm{\nabla\phi}$. Since $L_\phi$ is a measured bound and $\phi$ jumps at the edge of the cover, the test is a heuristic. We take $\varepsilon_{\mathrm{ray}}=0.1h$ and count a segment that is still undecided after $400$ steps as hidden. Off the cover of the implicit layer, $\phi$ is the constant $\pm\phi_\infty$ of \eqref{eq:implicit_blend}, which we replace by $\pm0.2\min_jR_j$. This is a lower bound of the distance to $\partial\Omega$ for uncovered points, since every boundary point lies in a ball $B_j$ within $R_j/\kappa_B$ of its center, with $\kappa_B=1.25$. All rays of one step are evaluated together by the grouped evaluation of Section~\ref{sec:node_gen_grouping}.

\subsection{Implementation}
\label{sec:app_listings_solvers}
Listing~\ref{lst:weights} computes the Laplacian weights of Section~\ref{sec:forward_operators} for a batch of stencils.
\begin{lstlisting}[language=Python,basicstyle=\ttfamily\lstsize,caption={Laplacian weights for a batch of $n_{\mathrm{batch}}$ stencils of $n_s$ nodes each, with the center first in each stencil, without the QR fallback of Section~\ref{sec:forward_operators}.},label={lst:weights}]
@partial(jax.jit, static_argnames="m")
def laplacian_weights(X, m):                 # X: (nb, ns, d) stencil nodes
    nb, ns, d = X.shape
    r = jnp.linalg.norm(X[:, :, None] - X[:, None, :], axis=-1)  # (nb, ns, ns)
    hc = r.max(axis=(1, 2))                                     # diameters h_c
    Z = (X - X.mean(axis=1, keepdims=True)) / hc[:, None, None]
    Pm = poly_basis(Z)                                          # (nb, ns, nq)
    nq = Pm.shape[-1]
    A = jnp.concatenate([
        jnp.concatenate([r**m, Pm], axis=2),
        jnp.concatenate([Pm.swapaxes(1, 2), jnp.zeros((nb, nq, nq))], axis=2)],
        axis=1)                                                 # (nb, ns+nq, ns+nq)
    rc = r[:, 0, :]                          # distances to the stencil center
    b = jnp.concatenate([m * (m + d - 2) * rc**(m - 2),
                         poly_laplacian(Z[:, 0]) / hc[:, None]**2], axis=1)
    return jnp.linalg.solve(A, b[..., None])[:, :ns, 0]          # (nb, ns)
\end{lstlisting}

Listing~\ref{lst:matvec} is the COO product of Section~\ref{sec:forward_operators}.
\begin{lstlisting}[language=Python,basicstyle=\ttfamily\lstsize,caption={A sparse matrix in COO form and its product with a vector; \texttt{centers} holds the row of each stencil and \texttt{stencils} its $n_s$ columns.},label={lst:matvec}]
rows = jnp.repeat(centers, ns)           # row of each weight
cols = stencils.reshape(-1)              # column of each weight
vals = weights.reshape(-1)

def matvec(vals, u):
    return jnp.zeros(N).at[rows].add(vals * u[cols])
\end{lstlisting}

Listing~\ref{lst:sbdf3} is the SBDF3 integration of Section~\ref{sec:forward_time}.
\begin{lstlisting}[language=Python,basicstyle=\ttfamily\lstsize,caption={SBDF3 integration with the ghost-node closure as one compiled loop.},label={lst:sbdf3}]
@jax.jit
def integrate(theta, U0, U1, U2):
    def step(j, state):
        Um2, Um1, Un = state                    # U^{j-2}, U^{j-1}, U^j
        t = (j + 1) * dt
        rhs = (3*Un[:N] - 1.5*Um1[:N] + Um2[:N]/3
               + dt * (3*Nh(t - dt, Un, theta) - 3*Nh(t - 2*dt, Um1, theta)
                       + Nh(t - 3*dt, Um2, theta)))
        return Um1, Un, solve(jnp.concatenate([rhs, g(t)]))
    return jax.lax.fori_loop(2, n_t, step, (U0, U1, U2))[-1]
\end{lstlisting}

Listing~\ref{lst:dense_jvp} shows the derivative rule of Section~\ref{sec:inverse_optimization}, which forms the Jacobian of a time-stepping residual in one forward-mode pass, and the Levenberg--Marquardt solve.
\begin{lstlisting}[language=Python,basicstyle=\ttfamily\lstsize,caption={Residual with a dense forward-mode Jacobian, and the Levenberg--Marquardt solve; \texttt{guarded} replaces a non-finite residual by $10^6$ with zero derivative.},label={lst:dense_jvp}]
def dense_jvp(residual):
    @jax.custom_jvp
    def r(theta):
        return residual(theta)

    @r.defjvp
    def r_jvp(primals, tangents):
        (theta,), (dtheta,) = primals, tangents
        def value_and_aux(th):
            v = residual(th)
            return v, v
        J, value = jax.jacfwd(            # one rollout with n_theta tangents
            value_and_aux, has_aux=True)(theta)
        return value, J @ dtheta
    return r

r = dense_jvp(guarded(residual))  # non-finite r -> 1e6, zero derivative
lm = optx.LevenbergMarquardt(rtol=1e-12, atol=1e-12,
         linear_solver=lx.AutoLinearSolver(well_posed=False))
sol = optx.least_squares(lambda th, _: r(th), lm, jnp.zeros(n_theta),  # cold start
          options={"jac": "fwd"}, max_steps=200, throw=False)
\end{lstlisting}

\section{Proofs}
\label{sec:proofs}

\subsection{Proof of Proposition~\ref{prop:mis}}
\label{sec:app_mis}
\begin{proof}
In any round, the undecided point of highest priority has no undecided neighbor of higher priority, so it is decided in that round, and the algorithm terminates. We show by induction along the priority order that every point that is not fixed ends in the state that the sequential algorithm assigns to it. Since a point enters $\mathsf{I}$ only after all of its higher-priority neighbors are decided, no lower-priority neighbor of $i$ enters $\mathsf{I}$ while $i$ is undecided. If $i$ moves to $\mathsf{O}$, its neighbor in $\mathsf{I}$ has higher priority. By the induction hypothesis this neighbor is accepted sequentially, and so $i$ is rejected sequentially. If $i$ moves to $\mathsf{I}$, all of its higher-priority neighbors are decided and none of them is in $\mathsf{I}$. By the induction hypothesis none of them is accepted sequentially, and so $i$ is accepted sequentially.
\end{proof}

\subsection{Truncation error of a stencil}
\label{sec:proof_consistency}
\begin{lemma}
\label{lem:consistency}
If $s\leq\varpi$ and $u\in C^{\varpi+1}$ on the closed ball of radius $\varrho_c$ about $x_c$, the bound \eqref{eq:consistency} holds with $|u|_{C^{\varpi+1}}$ the maximum of $|D^{\varpi+1}u(y)[\xi,\dots,\xi]|$ over this ball and over unit vectors $\xi$.
\end{lemma}
\begin{proof}
Let $\tilde u$ be the Taylor polynomial of $u$ of degree $\varpi$ about $x_c$. Since $s\leq\varpi$, $(\mathcal{L}\tilde u)(x_c)=(\mathcal{L}u)(x_c)$, and since the weights reproduce $\mathcal{L}$ on polynomials of degree $\varpi$, the left-hand side of \eqref{eq:consistency} equals $|\sum_kw_{ck}(u-\tilde u)(x_{ck})|$. Taylor's theorem bounds $|(u-\tilde u)(x_{ck})|$ by $\varrho_c^{\varpi+1}|u|_{C^{\varpi+1}}/(\varpi+1)!$.
\end{proof}

\subsection{Proof of Theorem~\ref{thm:elliptic}}
\label{sec:proof_elliptic}
\begin{proof}
Subtracting the bordered closure from $A\Pi u=\bar\Pi f+\tau_A$ and $B\Pi u=\tilde g+\tau_B$ gives $A_{\mathrm P}E_{\mathrm P}+A_{\mathrm G}E_{\mathrm G}-\lambda\mathbf 1=\tau_A$ and $B_{\mathrm P}E_{\mathrm P}+B_{\mathrm G}E_{\mathrm G}-\lambda\mathbf 1=\tau_B$, where $E_{\mathrm P}=\bar E$. The boundary rows give $E_{\mathrm G}=B_{\mathrm G}^{-1}(\tau_B+\lambda\mathbf 1-B_{\mathrm P}\bar E)$, and inserting $E_{\mathrm G}$ into the first rows gives $M\bar E=\tau_A-G\tau_B+\lambda\upsilon$. Multiplying by $\zeta^T$ gives $0=\zeta^T(\tau_A-G\tau_B)+\lambda\zeta^T\upsilon$, which is the expression for $\lambda$, and then $M\bar E=(I-\upsilon\zeta^T/\zeta^T\upsilon)(\tau_A-G\tau_B)\in\zeta^\perp$. Since $\bar E$ has mean zero by the shift, $\bar E=M_Q^{-1}(I-\upsilon\zeta^T/\zeta^T\upsilon)(\tau_A-G\tau_B)$. The rank-one projector $\upsilon\zeta^T/\zeta^T\upsilon$ has norm $\norm{\upsilon}\norm{\zeta}/|\zeta^T\upsilon|$, and so has its complement, which is neither $0$ nor $I$. Finally $\norm{\tau_A-G\tau_B}\leq\norm{\tau_A}+\norm{G}\norm{\tau_B}$.
\end{proof}

\subsection{Truncation error of SBDF3}
\label{sec:proof_sbdf3}
We define the local truncation error of \eqref{eq:sbdf3} by inserting the exact solution $u^j=u(\cdot,t_j)$,
\begin{multline}
 \Delta t\,\tau^{n+1}=\tfrac{11}{6}\bar\Pi u^{n+1}-\Delta t\,A\Pi u^{n+1}-3\bar\Pi u^{n}+\tfrac32\bar\Pi u^{n-1}-\tfrac13\bar\Pi u^{n-2}\\
 -\Delta t\bigl[3\mathcal{N}_h(t_n,\Pi u^n)-3\mathcal{N}_h(t_{n-1},\Pi u^{n-1})+\mathcal{N}_h(t_{n-2},\Pi u^{n-2})\bigr].
\end{multline}

\begin{lemma}
\label{lem:parabolic_consistency}
If $u$ is four times and $t\mapsto\mathcal{N}(t,u(t))$ three times continuously differentiable in time, then
\begin{multline*}
 \norm{\tau^{n+1}}_\infty\leq C_t\Delta t^3+\max_{t}\norm{A\Pi u(t)-\bar\Pi\mathcal{A}u(t)}_\infty\\
 +7\max_{t}\norm{\mathcal{N}_h(t,\Pi u(t))-\bar\Pi\mathcal{N}(t,u(t))}_\infty,
\end{multline*}
where $C_t$ depends only on bounds for $\partial_t^4u$ and $\partial_t^3\mathcal{N}(t,u(t))$. The same bound holds if part of $\mathcal{N}$ is evaluated at $t_{n+1}$. If $u(\cdot,t)\in C^{p+2}$ on the stencil balls, the spatial terms and the boundary part $\tau_B^{n+1}$ are bounded by \eqref{eq:consistency}.
\end{lemma}
\begin{proof}
Subtract $\Delta t\,\bar\Pi(u_t-\mathcal{A}u-\mathcal{N})(t_{n+1})=0$. By Taylor expansion about $t_{n+1}$, the backward difference $\tfrac{11}{6}v^{n+1}-3v^n+\tfrac32v^{n-1}-\tfrac13v^{n-2}$ equals $\Delta t\,v'(t_{n+1})+O(\Delta t^4)$, and the extrapolation $3w^n-3w^{n-1}+w^{n-2}$ equals $w(t_{n+1})+O(\Delta t^3)$. The remaining terms are the spatial truncation errors of $A$ at $t_{n+1}$ and of $\mathcal{N}_h$ at three time levels, with coefficients of total magnitude $7$.
\end{proof}
The $2$-norms in Theorem~\ref{thm:parabolic} are at most $\sqrt{N+N_b}$ times these maximum norms, a factor that the relative errors cancel as in the elliptic case.

\subsection{Proof of Theorem~\ref{thm:parabolic}}
\label{sec:proof_parabolic}
\begin{proof}
Let $\tau^{n+1}=(\tau^{n+1}_A,\tau^{n+1}_B)$, where $\tau^{n+1}_A$ is defined in Appendix~\ref{sec:proof_sbdf3} and $\tau^{n+1}_B=B\Pi u^{n+1}-g(t_{n+1})$. Subtracting \eqref{eq:sbdf3} gives $BE^{n+1}=\tau^{n+1}_B$ and
\begin{equation*}
 \tfrac{11}{6}\bar E^{n+1}-\Delta t\,AE^{n+1}=3\bar E^n-\tfrac32\bar E^{n-1}+\tfrac13\bar E^{n-2}+\Delta t\,z^{n+1},
\end{equation*}
with $z^{n+1}=3\vartheta^n-3\vartheta^{n-1}+\vartheta^{n-2}+\tau^{n+1}_A$ and $\vartheta^j=\mathcal{N}_h(t_j,\Pi u^j)-\mathcal{N}_h(t_j,U^j)$. The boundary rows give $E^{n+1}_{\mathrm G}=B_{\mathrm G}^{-1}(\tau^{n+1}_B-B_{\mathrm P}\bar E^{n+1})$, and the physical rows become
\begin{equation*}
 \bigl(\tfrac{11}{6}I-\Delta t\,M\bigr)\bar E^{n+1}=3\bar E^n-\tfrac32\bar E^{n-1}+\tfrac13\bar E^{n-2}+\Delta t\,w^{n+1},
\end{equation*}
with $w^{n+1}=z^{n+1}+G\tau^{n+1}_B$. The differences $\vartheta^j$ enter only the physical rows, so $\norm{w^j}\leq L_{\mathcal N}(3\norm{\bar E^{j-1}}+3\norm{\bar E^{j-2}}+\norm{\bar E^{j-3}})+\norm{\tau^j_A}+\norm{G}\norm{\tau^j_B}$ by the Lipschitz condition, and the stability assumption gives
\begin{equation*}
 \norm{\bar E^n}\leq C_S\Bigl(\max_{j\leq2}\norm{\bar E^j}+T\max_j\bigl(\norm{\tau^j_A}+\norm{G}\norm{\tau^j_B}\bigr)\Bigr)+7C_SL_{\mathcal N}\,\Delta t\sum_{j=0}^{n-1}\norm{\bar E^j}.
\end{equation*}
The discrete Gronwall inequality~\cite{clark1987short} yields \eqref{eq:parabolic_bound}.
\end{proof}

\subsection{Proof of Theorem~\ref{thm:inverse_aposteriori}}
\label{sec:proof_aposteriori}
\begin{proof}
Let $\bar J=\int_0^1J_h(\theta_\star+t\delta)\,dt$, so that $r_h(\theta_h)=r_h(\theta_\star)+\bar J\delta$, and let $\hat J=J_h(\theta_h)$. Then $\norm{\hat J-\bar J}\leq\int_0^1L_J(1-t)\norm{\delta}\,dt=\tfrac12L_J\norm{\delta}$. Stationarity reads $\hat J^T(r_h(\theta_\star)+\hat J\delta-(\hat J-\bar J)\delta)=0$, and multiplying by $(\hat J^T\hat J)^{-1}$ gives $\delta=-\hat J^{+}r_h(\theta_\star)+\hat J^{+}(\hat J-\bar J)\delta$. With $\norm{\hat J^{+}}=1/\sigma_h$ this gives \eqref{eq:inverse_bound}.
\end{proof}

\subsection{Proof of Theorem~\ref{thm:inverse_apriori}}
\label{sec:proof_apriori}
\begin{proof}
Let $R_0=\sigma_0/(3L_0)$. For $\norm{\theta-\theta_\star}=R_0$, $r_h(\theta)=r_h(\theta_\star)+\bar J(\theta-\theta_\star)$ with $\norm{\bar J-J_h(\theta_\star)}\leq\tfrac12L_0R_0$, so $\norm{r_h(\theta)}\geq(\sigma_0-\tfrac12L_0R_0)R_0-\norm{r_h(\theta_\star)}=\tfrac{5\sigma_0^2}{18L_0}-\norm{r_h(\theta_\star)}>\norm{r_h(\theta_\star)}$. The minimum of $\norm{r_h}$ over the closed ball is therefore attained in its interior, at a stationary point. For a stationary point $\theta_h$ in the ball, Weyl's inequality~\cite{horn2012matrix} gives $\sigma_h\geq\sigma_0-L_0\norm{\delta}$, and \eqref{eq:inverse_bound} gives $\norm{\delta}(\sigma_0-\tfrac32L_0\norm{\delta})\leq\norm{r_h(\theta_\star)}$. Since $\norm{\delta}\leq R_0$, the factor in parentheses is at least $\sigma_0/2$.
\end{proof}

\section{Details of the numerical experiments}
\label{sec:app_tests}
The stability check of Section~\ref{sec:forward_linear_algebra} is satisfied for every cloud, order, problem, and parameter family.

\subsection{Implicit layer of the treasure chest}
\label{sec:app_cover_refinement}
For the treasure chest the cover of the implicit layer is refined by the second rule of Appendix~\ref{sec:app_cover} once at $h=0.08$ and $h=0.06$ and twice at $h=0.04$. Each refinement refits the cap on smaller supports from the same $41\times41$ seed nodes, which raises the distance error of the cap from $5\times10^{-6}$ without refinement to $9\times10^{-5}$ at $h=0.04$.

\subsection{Velocity fields}
\label{sec:app_velocity}
\begin{table}[htbp]
\centering
\caption{Modes $m_i$ of the transport velocity basis, with $\tilde x=x-c_x$, $\tilde y=y-c_y$, $\tilde z=z-c_z$ for the center $c$ of the bounding box of the boundary nodes of the cloud with $h=0.16$, and $\Theta(s)=s^2|s|$.}
\label{tab:velocity_modes}
\begin{tabular}{llll}
\toprule
Set & $i$ & $d=2$ & $d=3$\\
\midrule
smooth & 1 & $\tfrac12(\tilde x^2+\tilde y^2)$ & $\tfrac12(\tilde x^2+\tilde y^2)$\\
 & 2 & $\sin\pi x\,\sin\pi y$ & $\tilde x\tilde z$\\
 & 3 & $\tilde x\tilde y+0.2\,\tilde x^3$ & $\tilde y\tilde z+0.15\,\tilde x^2$\\
$C^1$ & 1 & $\Theta(\tilde x)$ & $\Theta(\tilde x)$\\
 & 2 & $\Theta(\tilde y)$ & $\Theta(\tilde y)$\\
 & 3 & $\Theta(\tilde y)+0.2\,\tilde x\tilde y$ & $\Theta(\tilde z)+0.2\,\tilde x\tilde y$\\
\bottomrule
\end{tabular}
\end{table}
The reference level set $\phi_0$ of \eqref{eq:velocity_field} is the implicit layer of Section~\ref{sec:node_gen_levelset_math} for the cloud with $h=0.16$. Its balls cover all of $\Omega$, so $\phi_0=f_{\mathrm{loc}}$ is smooth in $\Omega$. With $c$ the center of the bounding box of the boundary nodes of this cloud, the quadratic of the advection--diffusion velocity is
\begin{equation*}
 m_0(x)=\begin{cases}\tfrac12\bigl((x-c_x)^2+(y-c_y)^2\bigr), & d=2,\\ \tfrac12\bigl((x-c_x)^2+(z-c_z)^2\bigr), & d=3,\end{cases}
\end{equation*}
and for $d=3$ the field is $\tilde w=\nabla\times(\Psi e_y)=(-\partial_z\Psi,0,\partial_x\Psi)$, which lies in the planes $y=\text{const}$. The field vanishes on the zero set of $\phi_0$ because $\phi_{0,+}^2$ and its gradient vanish there. The normalization set $X_0$ consists of the points of a uniform grid inside the exact domain, with spacing $0.005$ for $d=2$ and $0.015$ for $d=3$, so the velocity is the same on every cloud.

The velocity basis \eqref{eq:velocity_basis} has $n_\theta=3$ fields $w_i=\tilde w_i/\max_{x\in X_0}\norm{\tilde w_i(x)}$, with $\tilde w_i=(\partial_y,-\partial_x)(\phi_{0,+}^2m_i)$ for $d=2$ and $\tilde w_i=\nabla\times(\phi_{0,+}^2m_ie_i)$ for $d=3$. With $\tilde x=x-c_x$, $\tilde y=y-c_y$, $\tilde z=z-c_z$ and $\Theta(s)=s^2|s|$, which is $C^2$ but not $C^3$ at $s=0$, we use the two sets of modes in Table~\ref{tab:velocity_modes}. The fields of the second set are $C^1$ but not $C^2$ across the coordinate planes through $c$. The true coefficients are multiples of $(1.45,-0.90,0.60)$ for $d=2$ and $(1.20,-0.80,0.50)$ for $d=3$, scaled so that $\max_{x\in X_0}\norm{v(x;\theta_\star)}=1.18$. The velocity $v(x;\theta)$ does not depend on $t$. The hyperviscosity coefficients use $v_{\max}=1.18$, as in the forward problem.

\subsection{Slopes of the individual families}
\label{sec:app_slopes}
Tables~\ref{tab:forward_slopes} and~\ref{tab:inverse_slopes} list the least-squares slopes of the errors of every parameter family. Figures~\ref{fig:forward_convergence} and~\ref{fig:inverse_convergence} show the largest error over the families.

\begin{table}
\centering
\caption{Least-squares slopes of the forward error against $N^{-1/d}$ over the three finest clouds ($h\leq0.08$). Ranges are over the parameter families (for the obstacle, over the six reaction laws).}
\label{tab:forward_slopes}
\begin{tabular}{llccc}
\toprule
Problem & Geometry & $p=2$ & $p=4$ & $p=6$\\
\midrule
Darcy & H-channel & 1.9--2.2 & 6.3--8.3 & 7.7--9.0\\
Darcy & treasure chest & 1.4 & 5.1--5.5 & 7.2--7.8\\
Reaction--diffusion & H-channel & 1.9--2.3 & 4.8--5.0 & 7.3--7.6\\
Reaction--diffusion & treasure chest & 1.1--1.6 & 4.5--5.0 & 7.1--7.4\\
Advection--diffusion & H-channel & 2.3 & 5.6 & 6.0\\
Advection--diffusion & treasure chest & 2.2 & 4.8 & 7.9\\
Advection--diffusion--reaction & obstacle & 1.9--2.0 & 4.7 & 8.1--8.4\\
\bottomrule
\end{tabular}
\end{table}

\begin{table}
\centering
\caption{Least-squares slopes of the relative parameter error against $N^{-1/d}$ over the three finest clouds ($h\leq0.08$). Ranges are over the parameter families (for transport, over the two velocity bases, and for the obstacle, over the six reaction laws).}
\label{tab:inverse_slopes}
\begin{tabular}{llccc}
\toprule
Problem & Geometry & $p=2$ & $p=4$ & $p=6$\\
\midrule
Darcy & H-channel & 1.3--2.8 & 4.6--6.7 & 5.4--10.2\\
Darcy & treasure chest & 0.9--2.9 & 2.8--5.2 & 6.9--9.3\\
Reaction--diffusion & H-channel & 1.4--2.9 & 4.9--5.7 & 5.8--11.0\\
Reaction--diffusion & treasure chest & 1.0--2.8 & 3.7--5.2 & 7.1--7.4\\
Transport & H-channel & 2.4--3.4 & 2.2--3.3 & 5.2--5.3\\
Transport & treasure chest & 0.4--0.9 & 4.0--4.7 & 6.5--8.7\\
Advection--diffusion--reaction & obstacle & 1.5--2.4 & 4.2--5.3 & 6.2--9.2\\
\bottomrule
\end{tabular}
\end{table}

\subsection{Error estimates}
\label{sec:app_estimates}
Tables~\ref{tab:forward_estimate} and~\ref{tab:inverse_efficiency} list the quantities discussed in Sections~\ref{sec:forward_results} and~\ref{sec:inverse_results}. The effective stability constant \eqref{eq:effective_stability} divides by the truncation error $\tau_A$ of the physical rows, and $\norm{\tau_A}$ differs from $\norm{\tau}$ by less than $1\%$ on every cloud.

For the Darcy problem, the truncation error converges with slopes against $N^{-1/d}$ between $2.1$ and $2.5$ for $p=2$, $4.5$ and $6.1$ for $p=4$, and $7.0$ and $8.3$ for $p=6$. The effective stability constant lies between $3\times10^{-3}$ and $3\times10^{-1}$ in all cases and shows no systematic growth in $h$, except in the treasure chest with $p=2$, where its median over the families grows from $0.049$ at $h=0.08$ to $0.076$ at $h=0.04$. In the treasure chest with $p=4$ it decreases, and with $p=6$ it does not grow. The scaled stencil sums $W_L$ stay below $390$ and vary by at most a factor of $1.3$ over the three finest clouds.

\begin{table}
\centering
\caption{Quantities of the elliptic error estimate \eqref{eq:elliptic_bound} for the Darcy problem over the three finest clouds ($h\leq0.08$). The columns give the least-squares slope of the truncation error $\norm{\tau}_2/\sqrt N$ against $N^{-1/d}$ (smallest and largest over the parameter families), the range of the effective stability constant $\hat S_h$ of \eqref{eq:effective_stability}, and the largest scaled stencil sum $W_L$ of \eqref{eq:stencil_sums}.}
\label{tab:forward_estimate}
\begin{tabular}{llcccc}
\toprule
Geometry & $p$ & Slope of $\norm{\tau}$ & $\min\hat S_h$ & $\max\hat S_h$ & $\max W_L$\\
\midrule
H-channel & 2 & 2.3--2.5 & $1.0\times10^{-1}$ & $2.8\times10^{-1}$ & 42\\
 & 4 & 4.7--5.1 & $2.0\times10^{-2}$ & $2.1\times10^{-1}$ & 48\\
 & 6 & 7.6--8.3 & $1.4\times10^{-2}$ & $1.5\times10^{-1}$ & 85\\
\midrule
Treasure chest & 2 & 2.1--2.3 & $4.5\times10^{-2}$ & $9.4\times10^{-2}$ & 91\\
 & 4 & 4.5--6.1 & $3.1\times10^{-3}$ & $4.1\times10^{-2}$ & 169\\
 & 6 & 7.0--8.3 & $1.2\times10^{-2}$ & $1.1\times10^{-1}$ & 386\\
\bottomrule
\end{tabular}
\end{table}

\begin{table}
\centering
\caption{The efficiency index $\eta$ of the first-order inverse error bound over all inverse solves.}
\label{tab:inverse_efficiency}
\begin{tabular}{llcccc}
\toprule
Problem & Geometry & Solves & $\min\eta$ & median $\eta$ & $\max\eta$\\
\midrule
Darcy & H-channel & 75 & 0.007 & 0.424 & 0.841\\
 & treasure chest & 75 & 0.014 & 0.090 & 0.976\\
Reaction--diffusion & H-channel & 90 & 0.002 & 0.107 & 0.966\\
 & treasure chest & 90 & 0.005 & 0.143 & 0.998\\
Transport & H-channel & 30 & 0.021 & 0.143 & 0.298\\
 & treasure chest & 30 & 0.011 & 0.039 & 0.149\\
Advection--diffusion--reaction & obstacle & 90 & 0.008 & 0.100 & 0.873\\
\bottomrule
\end{tabular}
\end{table}

\subsection{Iterations of the inverse solves}
\label{sec:app_iterations}
From $\theta=0$, Levenberg--Marquardt needs $8$ to $88$ iterations, with medians of $18$ to $28$ for the three problems, except for a few inverse solves on the coarsest H-channel cloud, which stop at the limit of $200$ iterations. From the warm starts, Gauss--Newton needs one to seven iterations. Every recovered $\theta_h$ has a smaller residual than the true parameters, $\norm{r(\theta_h)}<\norm{r(\theta_\star)}$.

\section{Declaration of generative AI and AI-assisted technologies in the manuscript preparation process}

During the preparation of this work, the authors used Claude and Codex to assist drafting, running experiments, correctness checks, integration testing. After using this tool/service, the authors reviewed and edited the content as needed and take full responsibility for the content of the published article.

\bibliographystyle{cas-model2-names}
\bibliography{references}

\end{document}